\documentclass[11pt,a4paper,reqno]{amsart}
\usepackage[applemac]{inputenc}
\usepackage[T1]{fontenc}
\usepackage{amsmath}
\usepackage{amsthm}
\usepackage{amsfonts}
\usepackage{amssymb}
\usepackage{graphicx}
\usepackage{amsbsy}
\usepackage{mathrsfs}
\usepackage{bbm}
\newtheorem{theorem}{Theorem}[section]
\newtheorem{lemma}[theorem]{Lemma}

\newtheorem{proposition}[theorem]{Proposition}
\newtheorem{corollary}[theorem]{Corollary}

\theoremstyle{remark}
\newtheorem{remark}[theorem]{\it \bf{Remark}\/}

\numberwithin{equation}{section}
\catcode`@=11
\def\section{\@startsection{section}{1}%
  \z@{1.5\linespacing\@plus\linespacing}{.5\linespacing}%
  {\normalfont\bfseries\large\centering}}
\catcode`@=12
\newcommand{\be}{\begin{equation}}
\newcommand{\ee}{\end{equation}}
\newcommand{\bea}{\begin{eqnarray}}
\newcommand{\eea}{\end{eqnarray}}
\newcommand{\bee}{\begin{eqnarray*}}
\newcommand{\eee}{\end{eqnarray*}}

\def\pa{\partial}
\def\pr{\partial}

\def\RR{\mathbb{R}}

\def\TT{\mathcal{T}}
\def\SS{\mathcal{S}}

\def\ga{\gamma}

\def\de{\delta}
\def\om{\omega}
\def\ep{\varepsilon}

\def\fref#1{{\rm (\ref{#1})}}

\catcode`@=11
\def\supess{\mathop{\operator@font Sup\,ess}}
\catcode`@=12

\def\RR{\mathbb{R}}

\def\e{\varepsilon}

\def\bar#1{{\overline #1}}
\def\fref#1{{\rm (\ref{#1})}}

\def\R2+{\RR ^2_+}

\def\lsl{\frac{\lambda_s}{\lambda}}

\def\pa{\partial}

\def\lim{\mathop{\rm lim}}

\def\sup{\mathop{\rm sup}}
\def\exp{{\rm exp}}
\def\l{\lambda}

\def\log{{\rm log}}

\def\lsl{\frac{\lambda_s}{\lambda}}
\def\xsl{\frac{x_s}{\lambda}}

\def\pa{\partial}

\def\pa{\partial}

\def\Mod{\textrm{Mod}}
\def\NL{\textrm{NL}}

\begin{document}

\title[]{On the stability of type I blow up for the energy super critical heat equation}
\author[C. Collot]{Charles Collot}
\address{Laboratoire J.A. Dieudonn\'e, Universit\'e de Nice-Sophia Antipolis, France}
\email{ccollot@unice.fr}
\author[P. Rapha\"el]{Pierre Rapha\"el}
\address{Laboratoire J.A. Dieudonn\'e, Universit\'e de Nice-Sophia Antipolis, France}
\email{praphael@unice.fr}
\author[J. Szeftel]{Jeremie Szeftel}
\address{Laboratoire Jacques-Louis Lions, Universit\'e Paris 6, France}
\email{jeremie.szeftel@upmc.fr}


\begin{abstract} 
We consider the energy super critical semilinear heat equation $$\pa_tu=\Delta u+u^{p}, \ \ x\in \Bbb R^3, \ \ p>5.$$ We first revisit the construction of radially symmetric self similar solutions performed through an ode approach in \cite{troy}, \cite{buddselfsim}, and propose a bifurcation type argument suggested in \cite{bizon} which allows for a sharp control of the spectrum of the corresponding linearized operator in suitable weighted spaces. We then show how the sole knowledge of this spectral gap in weighted spaces implies the finite codimensional non radial stability of these solutions for smooth well localized initial data using energy bounds. The whole scheme draws a route map for the derivation of the existence and stability of self similar blow up in non radial energy super critical settings.
\end{abstract}

\maketitle

%%%%%%%%%%%%%%%%%%%%%%%%%%%%%%%%%%%%%%%%%%%%%%%%%%

\section{Introduction}

%%%%%%%%%%%%%%%%%%%%%%%%%%%%%%%%%%%%%%%%%%%%%%%%%%

%%%%%%%%%%%%%%%%%%%%%%%%%%%%%%%%%%%%%%%%%%%%%%%%%%

\subsection{Setting of the problem}

%%%%%%%%%%%%%%%%%%%%%%%%%%%%%%%%%%%%%%%%%%%%%%%%%%

We consider the focusing  nonlinear heat equation 
\begin{equation}\label{eq:heat}
\left\{\begin{array}{l}
\pr_tu = \Delta u + |u|^{p-1}u,\,\,\,\, (t,x)\in \mathbb{R}\times\mathbb{R}^d,\\
u_{|_{t=0}}=u_0,
\end{array}\right.
\end{equation}
where $p>1$. This model dissipates the total energy
\bea
E(u)=\frac{1}{2}\int |\nabla u|^2-\frac{1}{p+1}\int u^{p+1}, \,\,\,\, \frac{1}{2}\frac{dE}{dt}=-\int (\pr_tu)^2<0
\eea
and admits a scaling invariance: if $u(t,x)$ is a solution, then so is
\bea
u_\l( t ,x)=\l^{\frac{2}{p-1}}u(\l^2 t, \l x),\,\,\, \l>0.
\eea
This transformation is an isometry on the homogeneous Sobolev space
$$\|u_\l(t,\cdot)\|_{\dot{H}^{s_c}}=\|u( t,\cdot)\|_{\dot{H}^{s_c}}\textrm{ for }s_c=\frac{d}{2}-\frac{2}{p-1}.$$
We address in this paper the question of the existence and stability of blow up dynamics in the energy super critical range $s_c>1$ emerging from well localized initial data.

%%%%%%%%%%%%%%%%%%%%%%%%%%%%%%%%%%%%%%%%%%%%%%%%%%

\subsection{Type I and type II blow up}

%%%%%%%%%%%%%%%%%%%%%%%%%%%%%%%%%%%%%%%%%%%%%%%%%%

There is a large litterature devoted to the question of the description of blow up solutions for \eqref{eq:heat} and we recall some key facts related to our analysis.\\

\noindent{\em Type I blow-up}. The universal scaling lower bound on blow up rate $$\|u(t,\cdot)\|_{L^{\infty}}\gtrsim \frac{1}{(T-t)^{\frac 1{p-1}}}$$ is saturated by Type I singularities: $$\|u(t,\cdot)\|_{L^{\infty}}\sim \frac{1}{(T-t)^{\frac 1{p-1}}}.$$ These solutions concentrate to leading order a blow up profile $$u(t,x)\sim \frac{1}{\l(t)^{\frac{2}{p-1}}}v\left(\frac{x}{\l(t)}\right),  \ \ \l(t)= \sqrt{T-t},$$ which solves the non linear elliptic equation
\be
\label{ellipticequation}
\Delta v-\frac 1 2 \Lambda v +v^p=0, \ \ \Lambda v=\frac{2}{p-1}v+y\cdot\nabla v.
\ee 
There are three known classes of radial solutions to \eqref{ellipticequation}:\\
\begin{itemize}
\item the constant solution 
$$\kappa=\left(\frac{1}{p-1}\right)^{\frac{1}{p-1}}$$
which generates the stable ODE type blow up \cite{Gi,Gi1,Gi2,Gi3, MZduke,MeZa2};
\item the singular at the origin homogeneous self similar solution 
\be
\label{defphistar}
\Phi_*=\frac{c_\infty}{|x|^{\frac{2}{p-1}}}, \ \ c_\infty=\left(\frac{2}{p-1}\left(d-2-\frac{2}{p-1}\right)\right)^{\frac{1}{p-1}};
\ee
\item for \be
\label{exponentpjl}
1+\frac{4}{d-2}<p<p_{JL}=\left\{\begin{array}{ll} +\infty\ \ \mbox{for}\ \ d\leq 10,\\
1+\frac{4}{d-4-2\sqrt{d-1}}\ \ \mbox{for}\ \ d\geq 11,
\end{array}\right.
\ee
where $p_{JL}$ is the so called Joseph-Lundgren exponent, there exists a quantized sequence of smooth radially symmetric solutions $\Phi_n$ to \eqref{ellipticequation} which behave like $$\Phi_n(r)\sim \frac{c_n}{r^{\frac{2}{p-1}}}\ \ \mbox{as}\ \ r\to +\infty.$$ These solutions have been constructed using global Lyapounov functionals based ODE methods, \cite{lepin,troy,buddone,buddselfsim}, and a sharp condition for their existence in the radial positive class is given in \cite{Mizoselfsim}.
\end{itemize}

Note that all these profiles have infinite energy and it is not clear how they may participate in singularity formation emerging from smooth well localized initial data. In the radially symmetric setting, the series of breakthrough works \cite{MaMe1,MaMe2} gives partial answers showing the universality of the ODE blow up, and the possiblity of threshold dynamics with $\Phi^*$ or $\Phi_n$ regimes depending on the value of $p$. The analysis however is strongly restricted to the radial setting and uses the intersection number Lyapounov functionals based on the maximum principle. In particular this approach does not provide any insight into the direct construction of these blow up profiles and their dynamical stability in the non radial setting.\\

\noindent{\em Type  II blow-up}. For $p>p_{JL}$, there exist type II blow-up solutions $$\text{lim}_{t\to T} \ \|u(t)\|_{L^{\infty}}(T-t)^{\frac{1}{p-1}}=+\infty.$$ They appear in the radial setting as threshold dynamics again at the boundary of the ODE blow up set, \cite{MaMe3}, and dynamical proofs were proposed in \cite{HV,Mizo, Mizo2}. Their construction has been revisited in \cite{MRR,Co} in the setting of dispersive Schr\"odinger and wave equations, and in \cite{Co2} for the non radial heat equation, to produce the full quantized sequence of smooth type II blow up bubbles. The blow-up profile near the singularity is a stationary profile:
$$
u(t,x)\sim \frac{1}{\l(t)^{\frac{2}{p-1}}}Q\left(\frac{x}{\l(t)}\right), \ \  \l(t)\ll \sqrt{T-t}
$$
where $Q$ solves the soliton equation:
\be \label{solitonequation}
\Delta Q+Q^p=0.
\ee
The heart of the analysis is to control the flow near $Q$ using suitable energy estimates, hence avoiding maximum principle tools or spectral arguments. Type II is intimately connected to the singular self similar profile \eqref{defphistar}, see \cite{HV,MRR} for a discussion on this fundamental matter. 

%%%%%%%%%%%%%%%%%%%%%%%%%%%%%%%%%%%%%%%%%%%%%%%%%%

\subsection{Statement of the result}

%%%%%%%%%%%%%%%%%%%%%%%%%%%%%%%%%%%%%%%%%%%%%%%%%%

Our aim in this paper is to propose a robust approach for both the existence and stability of self similar blow up with smooth self similar $\Phi_n$ like profile. For the sake of simplicity, we restrict ourselves to 
\be
\label{assumptionpd}
d=3, \ \ p>5, \ \ p_{JL}=+\infty.
\ee We first revisit the construction of self similar blow up solutions of \cite{troy,buddselfsim} and implement an abstract bifurcation argument which relies on the sole existence of the stationary profile $Q$ given by \eqref{solitonequation}. Note that this kind of argument is classical in the ODE literature, see for example \cite{buddone, wei,Cor}, and relies on the oscillatory nature of the eigenfunctions of the linearized operator close to $\Phi^*$ for $p<p_{JL}$.

\begin{proposition}[Existence and asymptotic of excited self similar solutions]
\label{propconstruction}
Assume \eqref{assumptionpd}. For all $n>N$ large enough, there exist a\footnote{locally unique in some suitable space} smooth radially symmetric solution to the self similar equation \eqref{ellipticequation} such that $$\Lambda \Phi_n\ \ \mbox{vanishes exactly n times on}\ \  (0,+\infty).$$ Moreover, there exists a small enough constant $r_0>0$ independent of $n$ such that:\\
\noindent{\em 1. Behavior at infinity}:  
\be
\label{behavioruinfity}
\lim_{n\to +\infty}\sup_{r\geq r_0}\left(1+r^{\frac{2}{p-1}}\right)|\Phi_n(r) - \Phi_*(r)| = 0.
\ee
\noindent{\em 2. Behaviour at the origin}: there exists a sequence $\mu_n>0$ with $\mu_n\to 0$ as $n\to +\infty$ such that
\be
\label{behviourselfsimlocal}
\lim_{n\to +\infty}\sup_{r\leq r_0}\left|\Phi_n(r) -  \frac{1}{\mu_n^{\frac{2}{p-1}}}Q\left(\frac{r}{\mu_n}\right)\right| =0.
\ee
\end{proposition}

Hence these solutions realize a connection between the ground state behavior $Q$ at the origin, and the homogeneous self similar decay $\Phi_*$ at infinity. We now claim that these solutions are the blow up profile of a class of finite energy initial data leaving on a non radial n codimensional manifold.

\begin{theorem}[Finite codimensional stability of $\Phi_n$]
\label{thmmain}
Assume \eqref{assumptionpd}. Let $n>N$ large enough. There exists a Lipschitz codimension $n$ manifold\footnote{see Proposition \ref{pr:lipschitz} for a precise statement of the Lipschitz regularity.} of non radial initial data with finite energy
$$u_0=\chi_{A_0}\Phi_n+w_0$$
where $A_0\gg 1$ is large enough and $w_0$ is small enough\footnote{See \eqref{scalalrho} and below for the definition of the weighted Sobolev space $H^2_{\rho}$.}
\be \label{Proximite w0}
\|w_0\|_{H^2_\rho}+\|\Delta w_0\|_{L^2}+\|w_0\|_{\dot{H}^{s_c}}\ll 1,
\ee
such that the corresponding solution to \eqref{eq:heat} blows up in finite time $0<T<+\infty$ with a decomposition $$u(t,x)=\frac{1}{\l(t)^{\frac{2}{p-1}}}(\Phi_n+v)\left(t,\frac{x-x(t)}{\l(t)}\right)$$ where:\\
\noindent{\em 1. Control of the geometrical parameters}: the blow up speed is self similar $$\l(t)=\sqrt{(2+o(1))(T-t)}\ \ \mbox{as}\ \ t\to T$$ and the blow up point converges 
\be
\label{behaviourxpoitn}
x(t)\to x(T)\ \ \mbox{as}\ \ t\to T.
\ee
\noindent{\em 2. Behaviour of Sobolev norms}: there holds the asymptotic stability of the self similar profile above scaling
\be
\label{sobolevone}
\lim_{t\to T}\|v(t)\|_{\dot{H}^s}=0\ \ \mbox{for}\ \ s_c<s\leq 2,
\ee
 the boundedness of norms below scaling
 \be
 \label{sobolevtwo}
 \limsup_{t\to T} \|u(t)\|_{\dot{H}^s}<+\infty\ \ \mbox{for}\ \ 1\leq s<s_c,
 \ee 
 and the logarithmic growth of the critical norm 
 \be
 \label{normioeoe}
 \|u(t)\|_{\dot{H^{s_c}}}=c_n(1+o_{t\to T}(1))\sqrt{|\log (T-t)|}, \ \ c_n\neq 0.
 \ee
\end{theorem}

\vspace{0.3cm}

{\it Comments on the results}.

\vspace{0.2cm}

\noindent{\it 1. On the construction of self similar solutions}. The construction of self similar solutions has been performed in \cite{troy,buddselfsim} using a global Lyapounov functional ode approach. A very interesting variational approach has also been developed in \cite{cazenaveshattah,germain} in the setting of the related wave map problem. But there are many classical problems which lack both the variational structure and the monotonicity formulas, hence the need for a more systematic approach typically connected in a way or another to a bifurcation argument, which is the method we are implementing here. This procedure has been applied in various settings, see for example \cite{buddone,wei}. One advantage is that the proof further allows for a control of the linearized operator near the bifurcated object. The prize to pay however is that we only get the bifurcated family locally near the bifurcation point, and not the whole branch\footnote{unless one works for $p_{JL}-\e<p<p_{JL}$ in which case the whole family could be bifurcated along the same lines as for the supercritical gKdV equation performed in \cite{koch}.}, in particular not the fundamental mode. A closely related theorem is the construction \cite{koch} for the KdV equation near the critical exponent.\\

\noindent{\it 2. Stability of self similar blow up}. There is an important literature devoted to the stability of self similar solutions for both parabolic and dispersive problems. We aim at developing a robust approach which will extend to more complicated systems. Hence we avoid on purpose maximum principle like tools. In \cite{donninger1, donninger2, donninger, DS}, this kind of question has also been addressed for the radially symmetric supercritical wave map problem, Yang-Mills, wave equation and Yang-Mills heat flow. In those works, the analysis requires a detailed description of the complex spectrum of the linearized operator in suitable spaces which is a delicate matter, and seems to rely heavily on the fact that in the cases under consideration, the self similar solution has an explicit formula.  Our approach is different: once we know the spectral gap estimate with exponential weight which is an elementary consequence of either the variational characterization of the self similar solution as in \cite{cazenaveshattah}, or the construction of the solution by bifurcation as in the setting of Proposition \ref{propconstruction}, then the control of the nonlinear flow follows by adapting the general strategy based on energy bounds of \cite{RaphRod,MRR}. In fact, the exponential decay bounds behind \eqref{sobolevone} considerably simplify the analysis with respect to the study of type II blow up. The connexion with type II blow up has been made in \cite{HR} using exponential weights again, and the analysis is indeed intrinsically more involved. This energy method in weighted spaces also draws a natural connexion with the analysis of ODE type I blow up for both the heat and the wave equation \cite{Gi1,MZduke,MZ}. Note also that we assume \eqref{assumptionpd} for the sake of simplicity only\footnote{Raising dimensions causes the nonlinearity to become non smooth since $\lim_{d\to +\infty}p_{JL}=1$ and hence would lead to additional but manageable technical difficulties.}. The solutions of Theorem \ref{thmmain} will be obtained using first a by now classical Brouwer like topological argument \cite{martelmulti,MRR}, which is then complemented by a local uniqueness statement to construct the Lipschitz manifold as in \cite{Co,MMNR,KS, DKSW}.\\

\noindent{\it 3. The flow near the ground state}. The question of the classification of the flow near the special class of stationary solutions $Q$ has attracted a considerable attention in the past ten years in connection with the construction of the unstable manifold \cite{NS}, or the complete classification of the flow near $Q$ in energy subcritical \cite{NS,MaMeRa} and critical settings \cite{CMR}. The corresponding instabilities are central in the derivation of unstable type II blow up bubbles, \cite{MRR}.  From \eqref{behviourselfsimlocal}, the self similar solution ressembles the solitary wave $Q$ up to scaling near the origin, and hence the stability Theorem \ref{thmmain} can be viewed as describing one instability of the solitary wave solution in a suitable function space. Here a fundamental issue is that the linearized operator $H=-\Delta-pQ^{p-1}$ is {\it unbounded from below} in the sense of quadratic forms for $p<p_{JL}$. This is a major difference with respect to the case $p>p_{JL}$ where $H>0$. Our analysis in this paper shows how the nonlinear bifurcated solution $\Phi_n$ precisely allows for the suitable modification of the linearized operator which fixes this unboundedness from below of $H$. One also observes the same behaviour of Sobolev norms \eqref{sobolevone}, \eqref{sobolevtwo} as in \cite{MRR} which illustrates the deep non trivial structure in space of the associated blow up scenario\footnote{and hence its relevance in particular for more geometric problems like the harmonic heat flow of surfaces.}. Let us also stress that the nature of our energy like non linear estimates goes far beyond the stability issues of specific dynamics, and has allowed in \cite{MaMeRa} in a dispersive setting and \cite{CMR} in the parabolic setting for a complete description of the flow near the ground states.\\

This paper and \cite{RaphRod,MRR,Co, Co2} hence display a deep unity and design a route map based on robust energy estimates for the proof of the existence and stability of type I or type II blow up bubbles in both radial and non radial settings.

\subsection*{Acknowledgements}  All three authors are supported by the ERC-2014-CoG 646650 SingWave. P.R. is a junior member of the Institut Universitaire de France.

\subsection*{Notations} From now on and for the rest of this paper we fix $$d=3, \ \ p>5.$$
\noindent{\em The ground state expansion}. We let $\Phi_*$ given by \eqref{defphistar} be the unique radial homogenous self similar solution to \eqref{ellipticequation}. We let $Q(r)$ denote the unique radially symmetric solution to  
$$\left\{\begin{array}{ll} Q''+\frac{2}{r}Q'+Q^p=0,\\ Q(0)=1, \ \ Q'(0)=0,\end{array}\right.$$ 
which asymptotic behavior at infinity is from standard ODE argument\footnote{see \cite{Di,Jo,YiLi}.} given by
\bee
Q(r)=(1+o_{r\to +\infty}(1)) \Phi_*(r).
\eee
The next term is this expansion relates to the $p_{JL}$ exponent \eqref{exponentpjl} which is infinite in dimension $d=3$. Hence the quadratic polynomial $$\ga^2-\ga+pc_{\infty}^{p-1}=0$$
has complex roots 
\be
\label{eq:defofgamma}
\ga=\frac{1}{2} \pm i\om, \ \ \Delta:=1-4pc_{\infty}^{p-1}<0, \ \ \om:=\frac{\sqrt{-\Delta}}{2}
\ee
and the asymptotic behavior of $Q$ may be precised\footnote{see \cite{Di,Jo,YiLi}.}:
\be
\label{cejnenieie}
Q(r)=\Phi_*(r)+\frac{c_1\sin\left(\om\log(r)+c_2\right)}{r^{\frac{1}{2}}}+o\left(\frac{1}{r^{\frac{1}{2}}}\right)\textrm{ as }r\to +\infty
\ee
where $c_1\neq 0$ and $c_2\in\mathbb{R}$. Note that
$$\frac{1}{2}-\frac{2}{p-1}=s_c-1>0$$
so that the second term in the expansion of $Q$ is indeed a correction term.\\

\noindent{\em Weighted spaces}. We define the derivation operator 
$$D^k:=\left\{\begin{array}{ll} \Delta^m & \mbox{for}\ \ m=2k,\\ \nabla \Delta^k & \mbox{for}\ \ m=2k+1.
\end{array}\right.$$ We define the scalar product 
\be
\label{scalalrho}
(f,g)_{\rho}=\int_{\Bbb R^3}f(x)g(x)\rho dx, \ \ \rho=e^{-\frac{|x|^2}{2}}
\ee 
and let $L^2_\rho$ be the corresponded weighted $L^2$ space. We let $H^k_\rho$ be the completion of $\mathcal C^{\infty}_c(\Bbb R^d)$ for the norm 
$$\|u\|_{H^k_\rho}=\sqrt{\sum_{j=0}^k\|D^ju\|^2_{L^2_\rho}}.$$ 

\noindent{\em Linearized operators}. The scaling semi-group on functions $u:\mathbb R^d\rightarrow \mathbb R$:
\be \label{definition ulambda}
u_{\lambda}(x):=\lambda^{\frac{2}{p-1}}u(\lambda x)
\ee
has for infinitesimal generator the linear operator
$$
\Lambda u:= \frac{2}{p-1}u+x.\nabla u=\frac{\partial}{\partial \lambda} (u_{\lambda})_{|\lambda=1}.
$$
We define the linearized operator corresponding to \eqref{ellipticequation} around respectively $\Phi_*$ and $\Phi_n$ by
\bee
\mathcal{L}_\infty := -\Delta+\Lambda -\frac{pc_\infty^{p-1}}{r^2}, \ \ \mathcal{L}_n:=-\Delta +\Lambda -p\Phi_n^{p-1}
\eee
and their projection onto spherical harmonics:
\bee
&&\mathcal{L}_{\infty,m}:= -\partial_{rr}-\frac{2}{r}\partial_r+\frac{2}{p-1}+r\pr_r +\frac{m(m+1)}{r^2}-p\Phi_*^{p-1}, \ \ m\in \Bbb N,\\
&&\mathcal{L}_{n,m}:= -\partial_{rr}-\frac{2}{r}\partial_r+\frac{2}{p-1}+r\pr_r +\frac{m(m+1)}{r^2}-p\Phi_n^{p-1}, \ \ m\in \Bbb N.
\eee
Note that $\mathcal L_\infty$ is formally self adjoint for the $L^2_\rho$ scalar product but \eqref{eq:defofgamma} implies that the associated quadratic form is not bounded from below\footnote{this is a limit point circle case as $r\to 0$, \cite{RS}.} on $H^1_{\rho}$. We similarly define the linearized operator corresponding to \eqref{solitonequation} around $Q$:
\bee
&&H := -\Delta -pQ^{p-1}\\
&&H_m:= -\partial_{rr}-\frac{2}{r}\partial_r+\frac{m(m+1)}{r^2}-pQ^{p-1}, \ \ m\in \Bbb N.
\eee
and again $H$ is not bounded from below on $\dot{H}^1$.\\

\noindent{\em General notation}. We let $\chi(x)$ denote a smooth radially symmetric function with 
$$\chi(x):=\left\{\begin{array}{ll} 1 & \mbox{for}\ \ |x|\leq \frac 14,\\[2mm] 0 & \mbox{for}\ \ |x|\geq \frac 12,\end{array}\right.$$ and for $A>0$ (note the difference with \fref{definition ulambda}), $$\chi_A(x)=\chi\left(\frac xA\right).$$ 

\subsection*{Organization of the paper} This paper is organized as follows. In section \ref{sectionconstruction}, we construct the family of self similar solutions $\Phi_n$ using a nonlinear matching argument. The argument is classical, but requires a careful track of various estimates to obtain the sharp bounds \eqref{behavioruinfity}, \eqref{behviourselfsimlocal}. In section \ref{sectionspectral}, we show how these bounds coupled with Sturm-Liouville like arguments allow for a sharp counting of the number of instabilities of the linearized operator close to $\Phi_n$ which is self adjoint against the confining measure $\rho (y)dy$, Proposition \ref{prop:spectral}. In section \ref{sectiondynamical}, we turn to the heart of the dynamical argument and show how the spectral estimates in the weighted space coupled with the control of the super critical $\dot{H}^2$ norm design a stability zone for well localized initial data. 
 %%%%%%%%%%%%%%%%%%%%%%%%%%%%%%%%%%%%%%%%%%%%%%%%%%
%%%%%%%%%%%%%%%%%%%%%%%%%%%%%%%%%%%%%%%%%%%%%%%%%%

\section{Construction of self-similar profiles}
\label{sectionconstruction}

%%%%%%%%%%%%%%%%%%%%%%%%%%%%%%%%%%%%%%%%%%%%%%%%%%%%%%%%%%%%%%%%
%%%%%%%%%%%%%%%%%%%%%%%%%%%%%%%%%%%%%%%%%%%%%%%%%%

Our aim in this section is to construct radially symmetric solutions to the self similar equation
\be
\label{eq:selfsimilareq}
\Delta v-\Lambda v +v^p=0,
\ee
by using the classical strategy of gluing solutions which behave like $\Phi_*$ at infinity, and like $Q$ at the origin. As in \cite{bizon,wei,buddone}, the matching is made possible by the oscillatory behaviour \eqref{cejnenieie} for $p<p_{JL}$. The strength of this approach it that it relies on the implicit function theorem and not on fine monotonicity properties, and in this sense it goes far beyond the scalar parabolic setting, see for example \cite{koch} for a deeply related approach. The sharp control of the obtained solution \eqref{behavioruinfity}, \eqref{behviourselfsimlocal} will allow us to control the eigenvalues of the associated linearized operator in suitable exponentially weighted spaces, see Proposition \ref{prop:spectral}.

%%%%%%%%%%%%%%%%%%%%%%%%%%%%%%%%%%%%%%%%%%%%%%%%%%

\subsection{Exterior solutions}\label{sec:exteriorsolution}

%%%%%%%%%%%%%%%%%%%%%%%%%%%%%%%%%%%%%%%%%%%%%%%%%%

Recall that $\Phi_*$ given by \eqref{defphistar} is a solution to \eqref{eq:selfsimilareq} on $(0,+\infty)$. Our aim in this section is to construct the full family of solutions to \eqref{eq:selfsimilareq} on $[r_0,+\infty)$ for some small $r_0>0$ with the suitable behaviour at infinity.  The argument is a simple application of the implicit function theorem and continuity properties of the resolvent of $\mathcal L_\infty$ in suitable weighted spaces.\\

Given $0<r_0<1$, we define $X_{r_0}$ as the space of functions on $(r_0,+\infty)$ such that the following norm is finite
\bee
\|w\|_{X_{r_0}}=\sup_{r_0\leq r\leq 1}r^{\frac{1}{2}}|w|+\sup_{r\geq 1}r^{\frac{2}{p-1}+2}|w|.
\eee

\begin{lemma}[Outer resolvent of $\mathcal L_\infty$]
\label{lemma:constructionofpsi1andpsi2}{\em 1. Basis of fundamental solutions}: 
there exists two solutions $\psi_1$ and $\psi_2$ of
\be
\label{eqpsione}
\mathcal{L}_\infty(\psi_j)= 0\textrm{ for }j=1, 2\ \ \mbox{on}\ \ (0,+\infty)
\ee
with the following asymptotic behavior:
\begin{equation}
\label{behvorigin}
\psi_1= \frac{1}{r^{\frac{2}{p-1}}}\left(1+O\left(\frac{1}{r^2}\right)\right), \,\,\psi_2=  r^{\frac{2}{p-1}-3}e^{\frac{r^2}{2}}\left(1+O\left(\frac{1}{r^2}\right)\right),\textrm{ as } r\to+\infty
\end{equation}
and 
\begin{equation}\label{behvoriginbis}
\psi_1= \frac{c_3\sin\left(\om\log(r)+c_4\right)}{r^{\frac{1}{2}}}+O\left(r^{\frac{3}{2}}\right),\,\,\psi_2= \frac{c_5\sin\left(\om\log(r)+c_6\right)}{r^{\frac{1}{2}}}+O\left(r^{\frac{3}{2}}\right)\textrm{ as } r\to 0
\end{equation}
where $c_3, c_5\neq 0$ and $c_4, c_6\in\mathbb{R}$. Moreover, there exists  $c\neq 0$ such that 
\be
\label{refinedbeahviourpsione}
\Lambda\psi_1=\frac{c}{r^{\frac{2}{p-1}+2}}\left(1+O\left(\frac{1}{r^2}\right)\right)\textrm{ as } r\to+\infty.
\ee
{\em 2. Continuity of the resolvent}: let the inverse 
$$
\TT(f) = \left(\int_r^{+\infty} f\psi_2{r'}^{2}e^{-\frac{{r'}^2}{2}}dr'\right)\psi_1- \left(\int_r^{+\infty} f\psi_1{r'}^{2}e^{-\frac{{r'}^2}{2}}dr'\right)\psi_2,
$$
then 
$$\mathcal{L}_\infty(\TT(f))=f$$
and
\be
\label{estresolvent}
\|\TT(f)\|_{X_{r_0}}\lesssim \int_{r_0}^1 |f|{r'}^{\frac{3}{2}}dr' + \sup_{r\geq 1}r^{\frac{2}{p-1}+2}|f|.
\ee
\end{lemma}

\begin{proof} The proof is classical and we sketch the details for the reader's convenience.\\

\noindent{\bf step 1} Basis of homogeneous solutions. Recall \eqref{eq:defofgamma}. Let the change of variable and unknown
$$\psi(r)=\frac{1}{y^{\frac{\ga}{2}}}\phi(y),\,\,\,\, y=r^2,$$
then
$$\pr_r = 2r\pr_y,\,\,\,\, \pr^2_r=4r\pr_y(r\pr_y)=4r^2\pr_y^2+4r\pr_y(r)\pr_y=4y\pr_y^2+2\pr_y,\,\,\,\, r\pr_r = 2y\pr_y.$$
This yields
\bee
\mathcal{L}_\infty(\psi) &=&  \left(-4y\pr_y^2-2\pr_y - 4\pr_y +\frac{2}{p-1}+2y\pr_y - \frac{pc_\infty^{p-1}}{y}\right)\left(\frac{1}{y^{\frac{\ga}{2}}}\phi(y)\right).
\eee
Since 
\bee
\pr_y\left(\frac{1}{y^{\frac{\ga}{2}}}\phi(y)\right) &=& \frac{1}{y^{\frac{\ga}{2}}}\phi'(y) - \frac{\ga}{2y^{\frac{\ga}{2}+1}}\phi(y),\\
\pr_y^2\left(\frac{1}{y^{\frac{\ga}{2}}}\phi(y)\right) &=& \frac{1}{y^{\frac{\ga}{2}}}\phi''(y)- \frac{\ga}{y^{\frac{\ga}{2}+1}}\phi'(y)+\frac{\ga}{2}\left(\frac{\ga}{2}+1\right)\frac{1}{y^{\frac{\ga}{2}+2}}\phi(y),
\eee
we infer
\bee
\mathcal{L}_\infty(\psi) &=&  \Bigg\{-4y\left(\frac{1}{y^{\frac{\ga}{2}}}\phi''(y)- \frac{\ga}{y^{\frac{\ga}{2}+1}}\phi'(y)+\frac{\ga}{2}\left(\frac{\ga}{2}+1\right)\frac{1}{y^{\frac{\ga}{2}+2}}\phi(y)\right)\\
&+& (-6+2y)\left( \frac{1}{y^{\frac{\ga}{2}}}\phi'(y) - \frac{\ga}{2y^{\frac{\ga}{2}+1}}\phi(y)\right) +\left(\frac{2}{p-1} - \frac{pc_\infty^{p-1}}{y}\right)\frac{1}{y^{\frac{\ga}{2}}}\phi(y)\Bigg\}\\
&=& \frac{1}{y^{\frac{\ga}{2}}}\Bigg\{-4y\phi''(y)+\Big(4\ga-6+2y)\Big)\phi'(y)\\
&+& \left(\frac{2}{p-1}-\ga+\left(3\ga-\ga\left(\ga+2\right)- pc_\infty^{p-1}\right)\frac{1}{y}\right)\phi(y)\Bigg\}.
\eee
Since $\ga$ satisfies 
$$\ga^2-\ga+pc_\infty^{p-1}=0,$$
we infer
\bee
\mathcal{L}_\infty(\psi) &=& -\frac{4}{y^{\frac{\ga}{2}}}\left\{y\phi''(y)+\left(-\ga+\frac{3}{2}-\frac{y}{2}\right)\phi'(y)+\frac{1}{4}\left(-\frac{2}{p-1}+\ga\right)\phi(y)\right\}.
\eee
We change again variable by setting
$$\phi(y)=w(z),\,\,\,\, z=\frac{y}{2}.$$
We have
$$\phi'(y)=\frac{1}{2}w'(z),\,\,\,\, \phi''(y)=\frac{1}{4}w''(z)$$
and obtain
\bee
\mathcal{L}_\infty(\psi) &=& -\frac{2}{y^{\frac{\ga}{2}}}\left(zw''(z)+\left(-\ga+\frac{3}{2}-z\right)w'(z)-\left(\frac{1}{p-1}-\frac{\ga}{2}\right)w(z)\right).
\eee
Thus, $\mathcal{L}_\infty(\psi)=0$ if and only if
\bea\label{eq:kummersode}
z\frac{d^2w}{dz^2}+(b-z)\frac{dw}{dz}-aw=0
\eea
where we have used the notations
\bea\label{eq:actualvaluesofaandb}
a=\frac{1}{p-1}-\frac{\ga}{2},\,\,\,\, b=-\ga+\frac{3}{2}.
\eea
\eqref{eq:kummersode} is known as Kummer's equation. As long as $a$ is not a negative integer - which holds in particular for our choice of $a$  in \eqref{eq:actualvaluesofaandb} -, a basis of solutions to Kummer's equation consists of the Kummer's function $M(a,b,z)$ and the Tricomi function $U(a,b,z)$. These special functions have the following asymptotic behavior for $z\geq 0$ (see for example \cite{Handbookmathfunctions})
\begin{equation}\label{eq:kummersfunctionasymptotic1}
M(a,b,z) = \frac{\Gamma(b)}{\Gamma(a)}z^{a-b}e^z(1+O(z^{-1})),\,\, U(a,b,z) = z^{-a}(1+O(z^{-1}))\textrm{ as }z\to +\infty,
\end{equation}
\bea\label{eq:kummersfunctionasymptotic2}
M(a,b,z)=1+O(z)\textrm{ as }z\to 0,
\eea
and\footnote{Note that our choice of $b$  in \eqref{eq:actualvaluesofaandb} is such that $\Re(b)=1$ and $b\neq 1$.} for $1\leq\Re(b)<2$ with $b\neq 1$, 
\bea\label{eq:kummersfunctionasymptotic3}
U(a,b,z)=\frac{\Gamma(b-1)}{\Gamma(a)}z^{1-b}+\frac{\Gamma(1-b)}{\Gamma(a-b+1)}+O(z^{2-\Re(b)})\textrm{ as }z\to 0.
\eea
Since $w$ is a linear combination of $M(a,b,z)$ and $U(a,b,z)$, we immediately infer from \eqref{eq:kummersfunctionasymptotic1}, \eqref{eq:kummersfunctionasymptotic2} and \eqref{eq:kummersfunctionasymptotic3} the asymptotic of $w$ both as $z\to +\infty$ and $z\to 0_+$. Finally, since
$$\psi(r)=\frac{1}{r^\ga}w\left(\frac{r^2}{2}\right),$$
we infer from the asymptotic of $w$ the claimed asymptotic for $\psi$ both as $r\to +\infty$ and $r\to 0_+$. This concludes the proof of \eqref{behvorigin}, \eqref{behvoriginbis}.\\

\noindent{\bf step 2} Estimate on the resolvent. The Wronskian \bee
W &:=& \psi_1'\psi_2-\psi_2'\psi_1.
\eee
satisfies
$$
W' = \left(-\frac{2}{r}+r\right)W, \ \ W=\frac{C}{r^{2}}e^{\frac{r^2}{2}}$$
where we may without loss of generality assume $C=1$. We then solve
$$\mathcal{L}_\infty(w) = f$$
using the variation of constants which yields
\be
\label{variationconstats}
w = \left(a_1+\int_r^{+\infty} f\psi_2{r'}^{2}e^{-\frac{{r'}^2}{2}}dr'\right)\psi_1+ \left(a_2-\int_r^{+\infty} f\psi_1{r'}^{2}e^{-\frac{{r'}^2}{2}}dr'\right)\psi_2.
\ee
In particular, $\TT(f)$ corresponds to the choice $a_1=a_2=0$ and thus satisfies
$$\mathcal{L}_\infty(\TT(f))=f.$$

Next, we estimate $\TT(f)$ using the asymptotic behavior \fref{behvorigin} and \fref{behvoriginbis} of $\psi_1$ and $\psi_2$ as $r\to 0_+$ and $r\to +\infty$. For $r\geq 1$, we have
\bee
&&r^{\frac{2}{p-1}+2}|\TT(f)|\\
&=&r^{\frac{2}{p-1}+2}\left|\left(\int_r^{+\infty} f\psi_2{r'}^{2}e^{-\frac{{r'}^2}{2}}dr'\right)\psi_1- \left(\int_r^{+\infty} f\psi_1{r'}^{2}e^{-\frac{{r'}^2}{2}}dr'\right)\psi_2\right|\\
&\lesssim & r^2\left(\int_r^{+\infty} |f|{r'}^{\frac{2}{p-1}-1} dr'\right) +r^{\frac{4}{p-1}-1}e^{\frac{r^2}{2}} \left(\int_r^{+\infty} |f|\frac{1}{{r'}^{\frac{2}{p-1}}}{r'}^{2}e^{-\frac{{r'}^2}{2}}dr'\right)\\
&\lesssim& \left\{\sup_{r>1}\left(r^2\left(\int_r^{+\infty}\frac{dr'}{{r'}^3}\right) +r^{\frac{4}{p-1}-1}e^{\frac{r^2}{2}} \left(\int_r^{+\infty} {r'}^{-\frac{4}{p-1}}e^{-\frac{{r'}^2}{2}}dr'\right)\right)\right\}\sup_{r\geq 1}r^{\frac{2}{p-1}+2}|f|\\
&\lesssim&\sup_{r\geq 1}r^{\frac{2}{p-1}+2}|f|.
\eee

Also, for $r_0\leq r\leq 1$, we have
\bee
&&r^{\frac{1}{2}}\left|\left(\int_r^{+\infty} f\psi_2{r'}^{2}e^{-\frac{{r'}^2}{2}}dr'\right)\psi_1- \left(\int_r^{+\infty} f\psi_1{r'}^{2}e^{-\frac{{r'}^2}{2}}dr'\right)\psi_2\right|\\
&\lesssim& \int_r^1 |f|{r'}^{\frac{3}{2}}dr' + \int_1^{+\infty}{r'}^{\frac{2}{p-1}-1}|f|dr'\lesssim \int_{r_0}^1 |f|{r'}^{\frac{3}{2}}dr' + \sup_{r\geq 1}r^{\frac{2}{p-1}+2}|f|
\eee
and \eqref{estresolvent} is proved.\\

\noindent{\bf step 3} Refined control of $\psi_1$. We now turn to the proof of \eqref{refinedbeahviourpsione}. We decompose
\be
\label{defpsigonteilde}
\psi_1= \frac{1}{r^{\frac{2}{p-1}}}+\widetilde{\psi}_1.
\ee
Since $\mathcal{L}_\infty(\psi_1)=0$, we infer
$$\mathcal{L}_\infty(\widetilde{\psi}_1) = f$$
where $f$ is given by
\bee
f &=& -\mathcal{L}_\infty\left(\frac{1}{r^{\frac{2}{p-1}}}\right)= \pr_r^2\left(\frac{1}{r^{\frac{2}{p-1}}}\right) +\frac{2}{r}\pr_r\left(\frac{1}{r^{\frac{2}{p-1}}}\right) + \frac{pc_\infty^{p-1}}{r^2}\frac{1}{r^{\frac{2}{p-1}}}\\
 &=& \frac{2(p-3)}{p-1)}\frac{1}{r^{\frac{2}{p-1}+2}}.
\eee
In view of \eqref{variationconstats}, we infer
\bee
\widetilde{\psi}_1 &=& \left(a_1+\frac{2(p-3)}{p-1)}\int_r^{+\infty} \psi_2\frac{e^{-\frac{{r'}^2}{2}}}{{r'}^{\frac{2}{p-1}}}dr'\right)\psi_1+ \left(a_2-\frac{2(p-3)}{p-1}\int_r^{+\infty} \psi_1\frac{e^{-\frac{{r'}^2}{2}}}{{r'}^{\frac{2}{p-1}}}dr'\right)\psi_2.
\eee
On the other hand, we deduce from the asymptotic behavior of $\psi_1$
$$\widetilde{\psi}_1 = o\left(\frac{1}{r^{\frac{2}{p-1}}}\right)\textrm{ as }r\to +\infty.$$
In view of the  asymptotic behavior of $\psi_1$ and $\psi_2$ as $r\to +\infty$, this forces $a_1=a_2=0$ and hence
\bee
\widetilde{\psi}_1 &=& \frac{2(p-3)}{p-1}\left(\int_r^{+\infty} \psi_2\frac{e^{-\frac{{r'}^2}{2}}}{{r'}^{\frac{2}{p-1}}}dr'\right)\psi_1 - \frac{2(p-3)}{p-1}\left(\int_r^{+\infty} \psi_1\frac{e^{-\frac{{r'}^2}{2}}}{{r'}^{\frac{2}{p-1}}}dr'\right)\psi_2.
\eee
Then, applying $\Lambda$ to both sides, and using the asymptotic behavior of $\psi_1$ and $\psi_2$ as $r\to +\infty$ yields
$$\Lambda\widetilde{\psi}_1 =\frac{c}{r^{\frac{2}{p-1}+2}}\left(1+O\left(\frac{1}{r^2}\right)\right)\textrm{ as } r\to+\infty$$
for some constant\footnote{Actually, $c$ is explicitly given by
$$c=-\frac{p-3}{p-1}\neq 0.$$} $c\neq 0$. Injecting this into \eqref{defpsigonteilde}
 yields
$$
\Lambda\psi_1 = \Lambda\widetilde{\psi}_1 = \frac{c}{r^{\frac{2}{p-1}+2}}\left(1+O\left(\frac{1}{r^2}\right)\right)\textrm{ as } r\to+\infty$$
for some constant $c\neq 0$ and concludes the proof of Lemma \ref{lemma:constructionofpsi1andpsi2}.
\end{proof}

We are now in position to construct the family of outer self similar solutions as a classical consequence of the implicit function theorem.

\begin{proposition}[Exterior solutions]
\label{prop:exteriorsolution}
Let $0<r_0<1$ a small enough universal constant. For all 
\be
\label{estrovvove}
0<\ep \ll r_0^{s_c-1},
\ee
there exists a solution $u$ to 
\be
\label{sefsimiloutside}
\Delta u-\Lambda u+u^p=0\ \ \mbox{on}\ \ (r_0,+\infty)
\ee
of the form
$$u=\Phi_*+\ep\psi_1+\ep w$$
with the bounds:
\be
\label{wihfeihwwomeha}
\|w\|_{X_{r_0}}\lesssim \ep r_0^{1-s_c},\,\,\,\, \|\Lambda w\|_{X_{r_0}}\lesssim \ep r_0^{1-s_c}.
\ee
Furthermore,  $$w_{|_{\ep=0}}=0 \ \mbox{and}\ \ \|\pr_\ep w_{|_{\ep=0}}\|_{X_{r_0}}\lesssim r_0^{1-s_c}.$$
\end{proposition}

\begin{proof} This a classical consequence of Lemma \ref{lemma:constructionofpsi1andpsi2}.\\

\noindent{\bf step 1}. Setting up the Banach fixed point. Let $v$ such that
$$u=\Phi_*+\ep v,$$
then u solves \eqref{sefsimiloutside} iff:
$$\mathcal{L}_\infty(v)=\ep\frac{p(p-1)}{2}\Phi_*^{p-2}v^2+\ep F(\Phi_*,v,\ep)\textrm{ on }r>r_0,$$
where
$$F(\Phi_*,v,\ep) = \frac{1}{\ep^2}\left((\Phi_*+\ep v)^p-\Phi_*^p-p\Phi_*^{p-1}\ep v -\frac{p(p-1)}{2}\Phi_*^{p-2}\ep^2v^2\right).$$
Furthermore, we decompose
$$v=\psi_1+w$$
and hence, using in particular the fact that $\mathcal{L}_\infty(\psi_1)=0$, $w$ is a solution to
$$\mathcal{L}_\infty(w)= p(p-1) \ep G[\Phi_*,\psi_1, \ep]w\textrm{ on }r>r_0$$
where we defined the map: 
$$
G[\Phi_*,\psi_1, \ep]w =\left(\int_0^1(1-s)(\Phi_*+s\ep(\psi_1+w))^{p-2}ds\right)(\psi_1+w)^2.
$$ 
We claim the non linear bounds: assume that 
$$\|w\|_{X_{r_0}}\leq 1,$$
then 
\be
\label{firsbound}
 \int_{r_0}^1|G[\Phi_*,\psi_1, \ep]w|{r'}^{\frac{3}{2}}dr' + \sup_{r\geq 1}r^{\frac{2}{p-1}+2}|G[\Phi_*,\psi_1, \ep]w|  \lesssim  r_0^{1-s_c}
\ee
and
\bea
\label{seconbound}
\nonumber && \int_{r_0}^1|G[\Phi_*,\psi_1, \ep]w_1- G[\Phi_*,\psi_1, \ep]w_2|{r'}^{\frac{3}{2}}dr' \\
\nonumber && + \sup_{r\geq 1}r^{\frac{2}{p-1}+2} |G[\Phi_*,\psi_1, \ep]w_1- G[\Phi_*,\psi_1, \ep]w_2| \\
& \lesssim &  r_0^{1-s_c}\|w_1 - w_2\|_{X_{r_0}}. 
\eea
Assume \eqref{firsbound}, \eqref{seconbound}, then we look for $w$ as the solution of the following fixed point 
\bea\label{eq:fixedpointexteriorproblem}
w= \ep p(p-1)  \TT\Big(G[\Phi_*,\psi_1, \ep]w\Big),\,\,\,\,  w\in X_{r_0}.
\eea
In view of the assumption $\ep r_0^{1-s_c}\ll 1$, the continuity estimate on the resolvent \eqref{estresolvent} and the nonlinear estimates \eqref{firsbound}, \eqref{seconbound}, the Banach fixed point theorem applies and yields a unique solution $w$ to \eqref{eq:fixedpointexteriorproblem} with
\bee
\|w\|_{X_{r_0}}\lesssim \ep r_0^{1-s_c}.
\eee
Differentiating \eqref{eq:fixedpointexteriorproblem} in space, we immediately infer
\bee
\|\Lambda w\|_{X_{r_0}}\lesssim \ep r_0^{1-s_c}.
\eee
Finally, we compute $w_{|_{\ep=0}}$ and $\pr_\ep w_{|_{\ep=0}}$. In view of \eqref{eq:fixedpointexteriorproblem}, we have
$$w_{|_{\ep=0}}=0.$$
Also, we have
\bee
\pr_\ep w=p(p-1) \TT\Big(G[\Phi_*,\psi_1, \ep]w\Big)+\ep p(p-1)  \TT\Big(\pr_\ep G[\Phi_*,\psi_1, \ep]w\Big)
\eee
and hence
\bee
\pr_\ep w_{|_{\ep=0}}= p(p-1) \TT\Big(G[\Phi_*,\psi_1, \ep]w\Big)_{|_{\ep=0}}.
\eee
We have
\bee
G[\Phi_*,\psi_1, \ep]w_{|_{\ep=0}} &=& \left(\int_0^1(1-s)\Phi_*^{p-2}ds\right)\psi_1^2= \frac 1 2 \Phi_*^{p-2}\psi_1^2
\eee
which yields
\bee
\pr_\ep w_{|_{\ep=0}}= \frac{p(p-1)}{2} \TT\Big(\Phi_*^{p-2}\psi_1^2\Big).
\eee
The continuity estimate \eqref{estresolvent} and the asymptotic behavior of $\psi_1$ \fref{behvorigin} \fref{behvoriginbis} yield
$$\|\pr_\ep w_{|_{\ep=0}}\|_{X_{r_0}}\lesssim r_0^{1-s_c}.$$

\noindent{\bf step 2} Proof of the nonlinear estimates \eqref{firsbound}, \eqref{seconbound}. Note first that in view of Lemma \ref{lemma:constructionofpsi1andpsi2} and the definition of $\|\cdot\|_{X_{r_0}}$, we have for $r_0\leq r\leq 1$, 
\bee
|w(r)|+|\psi_1(r)|\lesssim r^{-\frac{1}{2}} = r^{1-\frac{2}{p-1} -s_c}\lesssim r^{1-s_c}|\Phi_*(r)|\leq r_0^{1 -s_c}|\Phi_*(r)|
\eee
while for $r\geq 1$, we have
\bee
|w(r)|+|\psi_1(r)|\lesssim |\Phi_*(r)|,
\eee
and hence, our choice of $\ep$ yields for all $r\geq r_0$
\bee
\ep|\psi_1(r)|+\ep|w(r)|\lesssim |\Phi_*(r)|.
\eee

Next, we estimate $G[\Phi_*,\psi_1, \ep]w$. For $r_0\leq r\leq 1$, we have
\bee
&&|G[\Phi_*,\psi_1, \ep]w| \leq (|\Phi_*(r)|+\ep(|\psi_1(r)|+|w(r)|))^{p-2} (|\psi_1(r)|+|w(r)|)^2\\
&\lesssim& |\Phi_*(r)|^{p-2} (|\psi_1(r)|+|w(r)|)^2\lesssim \left(\frac{1}{r^\frac{2}{p-1}}\right)^{p-2} \left(\frac{1}{r^{\frac{1}{2}}}\right)^2(1+\|w\|_{X_{r_0}})^2\lesssim r^{\frac{2}{p-1}-3}
\eee
and hence
$$
 \int_{r_0}^1|G[\Phi_*,\psi_1, \ep]w|{r'}^{\frac{3}{2}}dr'  \lesssim  \left(\int_{r_0}^1 {r'}^{-s_c}dr'\right)\lesssim  r_0^{1-s_c}.
$$
Also, for $r\geq 1$, we have
\bee
|G[\Phi_*,\psi_1, \ep]w| &\leq& (|\Phi_*(r)|+\ep(|\psi_1(r)|+|w(r)|))^{p-2} (|\psi_1(r)|+|w(r)|)^2\\
&\lesssim & \left(\frac{1}{r^\frac{2}{p-1}}\right)^p(1+\|w\|_{X_{r_0}})^2\lesssim  \frac{1}{r^{2+\frac{2}{p-1}}}
\eee
and hence
\bee
\sup_{r\geq 1}r^{\frac{2}{p-1}+2}|G[\Phi_*,\psi_1, \ep]w| &\lesssim & 1
\eee
and \eqref{firsbound} is proved. We now prove the contraction estimate:
\bee
&& G[\Phi_*,\psi_1, \ep]w_1- G[\Phi_*,\psi_1, \ep]w_2\\
 &=& \left(\int_0^1(1-s)(\Phi_*+s\ep(\psi_1+w_1))^{p-2}ds\right)(\psi_1+w_1)^2 \\
& -& \left(\int_0^1(1-s)(\Phi_*+s\ep(\psi_1+w_2))^{p-2}ds\right)(\psi_1+w_2)^2\\
&=&  \left(\int_0^1(1-s)(\Phi_*+s\ep(\psi_1+w_1))^{p-2}ds\right)\Big((\psi_1+w_1)^2 - (\psi_1+w_2)^2\Big)\\
& +&\Bigg( \int_0^1(1-s)(\Phi_*+s\ep(\psi_1+w_1))^{p-2}ds\\
& -& \int_0^1(1-s)(\Phi_*+s\ep(\psi_1+w_2))^{p-2}ds \Bigg)(\psi_1+w_2)^2\\
&=&  \left(\int_0^1(1-s)(\Phi_*+s\ep(\psi_1+w_1))^{p-2}ds\right)\Big(2\psi_1+ w_1+w_2\Big)(w_1 -w_2)\\
&+&(p-2)\left(\int_0^1s(1-s)\int_0^1(\Phi_*+s\ep(\psi_1+w_1)+\sigma s\ep(w_2-w_1))^{p-3}d\sigma ds\right)\\
&\times& (\psi_1+w_2)^2\ep (w_1-w_2)
\eee
and hence
\bee
&& |G[\Phi_*,\psi_1, \ep]w_1- G[\Phi_*,\psi_1, \ep]w_2|\\
&\lesssim& (|\Phi_*(r)|+\ep(|\psi_1(r)|+|w_1(r)|))^{p-2} \Big(|\psi_1(r)|+ |w_1(r)|+|w_2(r)|\Big)|w_1(r) -w_2(r)|\\
&& +(|\Phi_*(r)|+\ep(|\psi_1(r)|+|w_1(r)|))^{p-3}(|\psi_1(r)|+|w_2(r)|)^2\ep |w_1(r)-w_2(r)|\\
&\lesssim & \left\{ |\Phi_*(r)|^{p-2} (|\psi_1(r)|+ |w_1(r)|+|w_2(r)|)+\ep|\Phi_*(r)|^{p-3}(|\psi_1(r)|+|w_2(r)|)^2\right\} |w_1(r)-w_2(r)|.
\eee
For $r_0\leq r\leq 1$, we have
\bee
&&\Big| G[\Phi_*,\psi_1, \ep]w_1- G[\Phi_*,\psi_1, \ep]w_2 \Big|\\
&\lesssim& \left(\frac{1}{r^\frac{2}{p-1}}\right)^{p-2} \left(\frac{1}{r^{\frac{1}{2}}}\right)^2(1+\|w_1\|_{X_{r_0}}+\|w_2\|_{X_{r_0}})\|w_1 - w_2\|_{X_{r_0}}\\
&+& \ep\left(\frac{1}{r^\frac{2}{p-1}}\right)^{p-3} \left(\frac{1}{r^{\frac{1}{2}}}\right)^3(1+\|w_1\|_{X_{r_0}}+\|w_2\|_{X_{r_0}})^2\|w_1 - w_2\|_{X_{r_0}}\\
&\lesssim& \left(r^{\frac{2}{p-1}-3}+\ep r^{\frac{4}{p-1}-\frac{7}{2}}\right)\|w_1 - w_2\|_{X_{r_0}}
\eee
and hence
\bee
&& \int_{r_0}^1|G[\Phi_*,\psi_1, \ep]w_1- G[\Phi_*,\psi_1, \ep]w_2|{r'}^{\frac{3}{2}}dr' \\
&\lesssim& \left(\int_{r_0}^1 {r'}^{-s_c}dr'+\ep\int_{r_0}^1 {r'}^{1-2s_c}dr'\right)\|w_1 - w_2\|_{X_{r_0}}\\
 & \lesssim & r_0^{1-s_c}(1+\ep r_0^{1-s_c})\|w_1 - w_2\|_{X_{r_0}} \lesssim  r_0^{1-s_c}\|w_1 - w_2\|_{X_{r_0}}. 
 \eee
Similarly, for $r\geq 1$, 
\bee
&& |G[\Phi_*,\psi_1, \ep]w_1- G[\Phi_*,\psi_1, \ep]w_2|\\
&\lesssim& \left(\frac{1}{r^\frac{2}{p-1}}\right)^p(1+\|w_1\|_{X_{r_0}}+\|w_2\|_{X_{r_0}})^3\|w_1 - w_2\|_{X_{r_0}}\\
&\lesssim& \frac{1}{r^{2+\frac{2}{p-1}}}\|w_1 - w_2\|_{X_{r_0}}\\
\eee
and hence
\bee
\sup_{r\geq 1}r^{\frac{2}{p-1}+2} |G[\Phi_*,\psi_1, \ep]w_1- G[\Phi_*,\psi_1, \ep]w_2| &\lesssim & \|w_1 - w_2\|_{X_{r_0}}.
\eee
This concludes the proof of \eqref{firsbound}, \eqref{seconbound} and of Proposition \ref{prop:exteriorsolution}.
\end{proof}

%%%%%%%%%%%%%%%%%%%%%%%%%%%%%%%%%%%%%%%%%%%%%%%%%%

\subsection{Constructing interior self-similar solutions}\label{sec:interiorsolution}

%%%%%%%%%%%%%%%%%%%%%%%%%%%%%%%%%%%%%%%%%%%%%%%%%%

We now construct the family of inner solutions to \eqref{eq:selfsimilareq} in $[0, r_0]$ which after renormalization bifurcate from the {\em stationary} equation and the ground state solution $Q$.\\

We start with the continuity of the resolvent of the linearized operator $H$ close to Q in suitable weighted spaces. Given $r_1\gg1$, we define $Y_{r_1}$ as the space of functions on $(0,r_1)$ such that the following norm is finite
\bee
\|w\|_{Y_{r_1}}=\sup_{0\leq r\leq r_1}(1+r)^{-\frac{3}{2}}(|w|+r|\partial_r w|).
\eee

\begin{lemma}[Interior resolvent of $H$]
\label{lemma:homogeneoussolutionsofH}
\noindent{\em 1. Basis of fundamental solutions}: we have
$$H(\Lambda Q)=0, \ \ H\rho=0$$ with the following asymptotic behavior as $r\to+\infty$
$$
\Lambda Q(r)=\frac{c_7\sin\left(\om\log(r)+c_8\right)}{r^{\frac{1}{2}}}+O\left(\frac{1}{r^{s_c-\frac{1}{2}}}\right),\,\,\rho(r)=\frac{c_9\sin\left(\om\log(r)+c_{10}\right)}{r^{\frac{1}{2}}}+O\left(\frac{1}{r^{s_c-\frac{1}{2}}}\right),
$$
where $c_7, c_9\neq 0$, $c_8, c_{10}\in\mathbb{R}$.\\
\noindent{\em 2. Continuity of the resolvent}: let the inverse 
$$
\SS(f) = \left(\int_0^r f\rho {r'}^{2}dr'\right)\Lambda Q- \left(\int_0^r f\Lambda Q {r'}^{2}dr'\right)\rho
$$
then 
\be
\label{resneonoess}
\|\SS(f)\|_{  Y_{r_1}}\lesssim \sup_{0\leq r\leq r_1}(1+r)^{\frac{1}{2}}|f|.
\ee
\end{lemma}

\begin{proof} {\bf step 1} Fundamental solutions. Define
$$Q_\l(r)=\l^{\frac{2}{p-1}}Q(\l r),\,\,\,\, \l>0,$$
then
$$ \Delta Q_\l +Q_\l^p=0\textrm{ for all }\l>0$$
and differentiating w.r.t. $\l$ and evaluating at $\l=1$ yields 
$$H(\Lambda Q)=0.$$
Let $\rho$ be another solution to $H(\rho)=0$ which does not depend linearly on $\Lambda Q$, we aim at deriving the asymptotic of both $\Lambda Q$ and $\rho$ as $r\to +\infty$.\\
\noindent{\em Limiting problem} We first solve
\bea\label{eq:equationasymptotictoH}
-\pr^2_r\varphi - \frac{2}{r}\pr_r\varphi -\frac{pc_\infty^{p-1}}{r^2}\varphi=f.
\eea
The homogeneous problem admits the explicit basis of solutions
\bea\label{eq:explictiformulaforvarphi1andvarphi2}
\varphi_1=\frac{\sin(\om\log(r))}{r^{\frac{1}{2}}},\,\,\,\,\varphi_2=\frac{\cos(\om\log(r))}{r^{\frac{1}{2}}},
\eea
and the corresponding Wronskian is given by
$$W(r)=\varphi_1'(r)\varphi_2(r)-\varphi_2'(r)\varphi_1(r)=\frac{\om}{r^2}.$$
Using the variation of constants, the solutions to \eqref{eq:equationasymptotictoH} are given by
$$\varphi(r)=\left(a_{1,0}+\int_r^{+\infty}f\varphi_2\frac{{r'}^2}{\om}dr'\right)\varphi_1+\left(a_{2,0}-\int_r^{+\infty}f\varphi_1\frac{{r'}^2}{\om}dr'\right)\varphi_2.$$
\noindent{\em Inverting $H$.} We now claim that all solutions to $H(\phi)=0$ admit an expansion
\be
\label{probeminifity}
\phi(r) = a_{1,0}\varphi_1+a_{2,0}\varphi_2+O\left(\frac{1}{r^{s_c-\frac{1}{2}}}\right)\textrm{ as }r\to +\infty.
\ee
Indeed, we rewrite the equation
$$-\pr^2_r\phi - \frac{2}{r}\pr_r\phi -\frac{pc_\infty^{p-1}}{r^2}\phi =f, \ \ f= p\left(Q^{p-1}(r)-\frac{c_\infty^{p-1}}{r^2}\right)\phi(r),$$
and hence \bea\label{eq:easylinearfixedpoint}
\phi=a_{1,0}\varphi_1+a_{2,0}\varphi_2+\widetilde{\phi},\,\,\,\,\widetilde{\phi}=\mathcal{F}\left(\widetilde{\phi}\right)
\eea
where
\bee
\mathcal{F}\left(\widetilde{\phi}\right)(r) &=& -\left(\int_r^{+\infty}p\left(Q^{p-1}(r')-\frac{c_\infty^{p-1}}{{r'}^2}\right)\left(a_{1,0}\varphi_1+a_{2,0}\varphi_2+\widetilde{\phi}\right)(r')\varphi_2\frac{{r'}^2}{\om}dr'\right)\varphi_1\\
&&+\left(\int_r^{+\infty}p\left(Q^{p-1}(r')-\frac{c_\infty^{p-1}}{{r'}^2}\right)\left(a_{1,0}\varphi_1+a_{2,0}\varphi_2+\widetilde{\phi}\right)(r')\varphi_1\frac{{r'}^2}{\om}dr'\right)\varphi_2.
\eee
Recall that
\bee
Q(r)=\frac{c_{\infty}}{r^{\frac{2}{p-1}}}+O\left(\frac{1}{r^{\frac{1}{2}}}\right)\textrm{ as }r\to +\infty
\eee
so that
\bee
\left|p\left(Q^{p-1}(r)-\frac{c_\infty^{p-1}}{{r}^2}\right)\right|\lesssim \frac{1}{r^{1+s_c}}\textrm{ for }r\geq 1.
\eee
We infer for $r\geq 1$
\bee
\left|\mathcal{F}\left(\widetilde{\phi}\right)(r)\right| &\lesssim& \frac{1}{r^{\frac{1}{2}}}\left(\int_r^{+\infty}\left(\frac{1}{{r'}^{s_c}}+\frac{1}{{r'}^{s_c-\frac{1}{2}}}\left|\widetilde{\phi}\right|(r')\right)dr'\right)\\
&\lesssim& \frac{1}{r^{s_c-\frac{1}{2}}}+\frac{1}{r^{\frac{1}{2}}}\left(\int_r^{+\infty}\frac{1}{{r'}^{s_c-\frac{1}{2}}}\left|\widetilde{\phi}\right|(r') dr'\right)
 \eee
 and
\bee
\left|\mathcal{F}\left(\widetilde{\phi}_1\right)(r)-\mathcal{F}\left(\widetilde{\phi}_2\right)(r)\right| &\lesssim& \frac{1}{r^{\frac{1}{2}}}\left(\int_r^{+\infty}\frac{1}{{r'}^{s_c-\frac{1}{2}}}\left|\widetilde{\phi}_1-\widetilde{\phi}_2\right|(r') dr'\right).
 \eee 
 Thus, for $R\geq 1$ large enough, the Banach fixed point theorem applies in the space corresponding to the norm
 $$\sup_{r\geq R}r^{s_c-\frac{1}{2}}\left|\widetilde{\phi}\right|(r)$$
 and yields a unique solution $\widetilde{\phi}$ to \eqref{eq:easylinearfixedpoint} with 
 $$\sup_{r\geq R}r^{s_c-\frac{1}{2}}\left|\widetilde{\phi}\right|(r)\leq 1,$$ 
 and \eqref{probeminifity} is proved.\\
 In particular, in view of the explicit formula \eqref{eq:explictiformulaforvarphi1andvarphi2} for $\varphi_1$ and $\varphi_2$, and in view of the fact that $H(\Lambda Q)=0$ and $H(\rho)=0$, we infer as $r\to +\infty$
\be
\label{estqpioen}
\Lambda Q(r) = \frac{c_7\sin\left(\om\log(r)+c_8\right)}{r^{\frac{1}{2}}}+O\left(\frac{1}{r^{s_c-\frac{1}{2}}}\right),\,\, \rho = \frac{c_9\sin\left(\om\log(r)+c_{10}\right)}{r^{\frac{1}{2}}} +O\left(\frac{1}{r^{s_c-\frac{1}{2}}}\right)
\ee
where $c_7, c_9\neq 0$, $c_8, c_{10}\in\mathbb{R}$.\\

\noindent{\bf step 2} Continuity of the resolvent. We compute
$$
W := \Lambda Q'\rho-\rho'\Lambda Q, \ \ W' = -\frac{2}{r}W, \ \ W=\frac{-1}{r^{2}},$$
without loss of generality. Still without loss of generality for $R_0>0$ small enough such that $\Lambda Q>0$ on $[0,R_0]$ the integration of the Wronskian law yields
$$
\rho = - \Lambda Q \int_r^{R_0} \frac{1}{(\Lambda Q)^2r'^{2}} dr'
$$
on $(0,R_0]$ which ensures
\be
\label{estrhorgigine}
|\rho(r)|\lesssim \frac1 r , \ |\partial_r \rho(r)|\lesssim \frac{1}{r^2}  \ \mbox{as}\ \ r\to 0.
\ee We now solve
$$H(w) = f,$$
using the variation of constants which yields
\bee
w &=&  \left(a_1+\int_0^r f\rho {r'}^{2}dr'\right)\Lambda Q+ \left(a_2-\int_0^r f\Lambda Q {r'}^{2}dr'\right)\rho.
\eee
In particular, $\SS(f)$ corresponds to the choice $a_1=a_2=0$ and thus 
$$H(\SS(f))=f.$$ Finally, using the estimates \eqref{estqpioen}, \eqref{estrhorgigine}, we estimate for $0\leq r\leq 1$:
\bee
&&|\SS(f)|= \left|\left(\int_0^r f\rho {r'}^{2}dr'\right)\Lambda Q- \left(\int_0^r f\Lambda Q {r'}^{2}dr'\right)\rho\right|\\
&\lesssim& \left(\int_0^r r'dr'+\frac 1 r \int_0^r {r'}^{2}dr' \right)\sup_{0\leq r\leq 1}|f| \lesssim \sup_{0\leq r\leq r_1}(1+r)^{\frac{1}{2}}|f|,
\eee
\bee
&&|r\partial_r\SS(f)|= \left|\left(\int_0^r f\rho {r'}^{2}dr'\right)r\partial_r\Lambda Q- \left(\int_0^r f\Lambda Q {r'}^{2}dr'\right)r\partial_r\rho\right|\\
&\lesssim& \left(r^2\int_0^r r'dr'+\frac 1 r \int_0^r {r'}^{2}dr' \right)\sup_{0\leq r\leq 1}|f| \lesssim \sup_{0\leq r\leq r_1}(1+r)^{\frac{1}{2}}|f|,
\eee 
and for $1\leq r \leq r_1$:
\bee
&&(1+r)^{-\frac{3}{2}}|\SS(f)|= (1+r)^{-\frac{3}{2}}\left|\left(\int_0^r f\rho {r'}^{2}dr'\right)\Lambda Q- \left(\int_0^r f\Lambda Q {r'}^{2}dr'\right)\rho\right|\\
&\lesssim& (1+r)^{-2}\left(\int_0^r f (1+r')^{\frac{3}{2}} dr'\right)\lesssim (1+r)^{-2}\left(\int_0^r (1+r') dr'\right)\sup_{0\leq r\leq r_1}(1+r)^{\frac{1}{2}}|f|\\
&\lesssim& \sup_{0\leq r\leq r_1}(1+r)^{\frac{1}{2}}|f|
\eee
\bee
&&(1+r)^{-\frac{3}{2}}|r\partial_r \SS(f)|= (1+r)^{-\frac{3}{2}}\left|\left(\int_0^r f\rho {r'}^{2}dr'\right)r\partial_r \Lambda Q- \left(\int_0^r f\Lambda Q {r'}^{2}dr'\right)r\partial_r \rho\right|\\
&\lesssim& (1+r)^{-2}\left(\int_0^r f (1+r')^{\frac{3}{2}} dr'\right)\lesssim (1+r)^{-2}\left(\int_0^r (1+r') dr'\right)\sup_{0\leq r\leq r_1}(1+r)^{\frac{1}{2}}|f|\\
&\lesssim& \sup_{0\leq r\leq r_1}(1+r)^{\frac{1}{2}}|f|,
\eee 
which concludes the proof of \eqref{resneonoess} and Lemma \ref{lemma:homogeneoussolutionsofH}.
\end{proof}

We are now in position to build the family of interior solutions:

\begin{proposition}[Construction of the interior solution]\label{prop:interiorsolution}
Let $r_0>0$ small enough and let $0<\l\leq r_0$. Then, there exists a solution $u$ to 
$$\Delta u-\Lambda u +u^p=0\textrm{ on }0\leq r\leq r_0$$
of the form
$$u=\frac{1}{\l^{\frac{2}{p-1}}}(Q+\l^2 T_1)\left(\frac{r}{\l}\right)$$
with
\be
\label{estinterirotone}
\|T_1\|_{Y_{\frac{r_0}{\l}}}+\|\Lambda T_1\|_{Y_{\frac{r_0}{\l}}}+\|\Lambda^2 T_1\|_{Y_{\frac{r_0}{\l}}}\lesssim 1.
\ee
\end{proposition}

\begin{proof} This is again a classical consequence of Lemma \ref{lemma:homogeneoussolutionsofH}.\\

\noindent{\bf step 1} Setting up the Banach fixed point. 
We look for $u$ of the form
$$u=\frac{1}{\l^{\frac{2}{p-1}}}(Q+\l^2 T_1)\left(\frac{r}{\l}\right)$$
so that $u$ solves $\Delta u -\Lambda u+u^p=0$ on $[0,r_0]$ if and only if
$$H(T_1)=J[Q, \l^2]T_1 \ \text{on} \ 0\leq r \leq r_1$$
where
$$
r_1=\frac{r_0}{\lambda}\geq 1
$$
so that
$$
\lambda^2r_1^2=r_0^2\ll 1
$$
and with
\bee
J[Q, \l^2]T_1 &=& -\Lambda Q-\lambda^2\Lambda T_1 + p(p-1)\l^2\left(\int_0^1(1-s)(Q+s\l^2T_1)^{p-2}ds\right)T_1^2.
\eee
We claim the nonlinear estimates: assume $\|w\|_{Y_{r_1}}\lesssim 1$, then
\bea
\label{esenneoe}
&&\sup_{0\leq r\leq r_1}(1+r)^{\frac{1}{2}}|J[Q, \l^2]w| \lesssim 1,\\
\label{esenneoebis}
&&\sup_{0\leq r\leq r_1}(1+r)^{\frac{1}{2}}|J[Q, \l^2]w_1 - J[Q, \l^2]w_2| \lesssim r_1^2\l^2\|w_1-w_2\|_{Y_{r_1}}.
\eea
Assume \eqref{esenneoe}, \eqref{esenneoebis}, we then look for $T_1$ as the solution to the fixed point
\bea\label{eq:fixedpointinteriorproblem}
T_1=\SS(J[Q, \l^2]T_1). 
\eea
In view of the bound $\l^2r_1^2\ll 1$, the resolvent estimate \eqref{resneonoess} and the nonlinear estimates \eqref{esenneoe}, \eqref{esenneoebis}, the Banach fixed point theorem applies and yields a unique solution $T_1$ to \eqref{eq:fixedpointinteriorproblem} which furthermore satisfies:
\bee
\|T_1\|_{Y_{\frac{r_0}{\l}}}\lesssim 1.
\eee

\noindent{\bf step 2} Proof of \eqref{esenneoe}, \eqref{esenneoebis}. Note first that for $0\leq r\leq r_1$, we have 
\bee
|w(r)|\lesssim (1+r)^{\frac{3}{2}}=r_1^2(1+r)^{-\frac{1}{2}}\lesssim r_1^2|Q(r)|.
\eee
Thus, we infer for all $0\leq r\leq r_1$
\bee
\l^2|w(r)|\lesssim \l^2r_1^2|Q(r)|
\eee
and hence, our choice of $\l$ yields for all $0\leq r\leq r_1$
\bee
\l^2|w(r)|\lesssim |Q(r)|.
\eee

Next, we estimate $J[Q, \l^2]w$. For $0\leq r\leq r_1$, we have
\bee
&&|J[Q, \l^2]w|  \leq  |\Lambda Q| + p(p-1)\l^2(|Q|+\l^2|w|)^{p-2}|w|^2+\lambda^2\left|\frac{1}{2}w+r\partial_r w\right| \\
&\lesssim &  |\Lambda Q| + \l^2|Q|^{p-2}|w|^2+\lambda^2|\frac{1}{2}w+r\partial_r w|\\
&\lesssim& (1+r)^{-\frac{1}{2}}+\l^2 (1+r)^{-\frac{2(p-2)}{p-1}}(1+r)^3\|w\|^2_{Y_{r_1}} +\lambda^2 (1+r)^{\frac 3 2} \|w\|^2_{Y_{r_1}}  \\
&\lesssim & (1+r)^{-\frac{1}{2}}\left(1+\l^2 (1+r)^{\frac{2}{p-1}+\frac{3}{2}} +\lambda^2 (1+r)^2 \right)\\
&\lesssim& (1+r)^{-\frac{1}{2}}\left(1+\l^2 (1+r)^{-s_c+3}+\lambda^2r_1^2\right)\lesssim (1+r)^{-\frac{1}{2}}\left(1+\l^2r_1^2\right)\lesssim (1+r)^{-\frac{1}{2}}
\eee
and hence
\bee
\sup_{0\leq r\leq r_1}(1+r)^{\frac{1}{2}}|J[Q, \l^2]w| &\lesssim 1.
\eee

Next, we estimate $|J[Q, \l^2]w_1 - J[Q, \l^2]w_2|$. We have
\bee
&& J[Q, \l^2]w_1 - J[Q, \l^2]w_2 \\
&=& p(p-1)\l^2\left(\int_0^1(1-s)(Q+s\l^2w_1)^{p-2}ds\right)w^2_1 -  p(p-1)\l^2\left(\int_0^1(1-s)(Q+s\l^2w_2)^{p-2}ds\right)w^2_2\\
&& +\lambda^2 \left(\frac{1}{2}(w_1-w_2)+r(\partial_r w_1-\partial_r w_2) \right) w\\
&=&  p(p-1)\l^2\left(\int_0^1(1-s)(Q+s\l^2w_1)^{p-2}ds\right)(w^2_1-w_2^2)\\
&+& p(p-1)\l^2\left(\int_0^1(1-s)(Q+s\l^2w_1)^{p-2}ds - \int_0^1(1-s)(Q+s\l^2w_2)^{p-2}ds\right)w^2_2\\
&& +\lambda^2 \left(\frac{1}{2}(w_1-w_2)+r(\partial_r w_1-\partial_r w_2) \right) w\\
&=&  p(p-1)\l^2\left(\int_0^1(1-s)(Q+s\l^2w_1)^{p-2}ds\right)(w_1+w_2)(w_1-w_2)\\
&+& p(p-1)(p-2)\l^4\left(\int_0^1s(1-s)\int_0^1(Q+s\l^2w_1+\sigma s\l^2(w_2-w_1))^{p-3}d\sigma ds\right)w^2_2(w_1-w_2)\\
&& +\lambda^2 \left(\frac{1}{2}(w_1-w_2)+r(\partial_r w_1-\partial_r w_2) \right) w\\
\eee
and hence
\bee
&& |J[Q, \l^2]w_1 - J[Q, \l^2]w_2| \lesssim  \l^2(|Q(r)|+\l^2|w_1(r)|)^{p-2}(|w_1(r)|+|w_2(r)|)|w_1(r)-w_2(r)|\\
&& +\l^4(|Q(r)|+\l^2|w_1(r)|+\l^2|w_2(r)|)^{p-3}|w_2(r)|^2|w_1(r)-w_2(r)| \\
&&+\lambda^2 \left(\frac{1}{2}(w_1-w_2)+r(\partial_r w_1-\partial_r w_2) \right) w \\
&\lesssim&  \l^2|Q(r)|^{p-2}(|w_1(r)|+|w_2(r)|)|w_1(r)-w_2(r)| +\l^4|Q(r)|^{p-3}|w_2(r)|^2|w_1(r)-w_2(r)|\\
&& +\lambda^2 \left(\frac{1}{2}(w_1-w_2)+r(\partial_r w_1-\partial_r w_2) \right) w.
\eee
This yields
\bee
&& |J[Q, \l^2]w_1 - J[Q, \l^2]w_2| \lesssim  \l^2(1+r)^{-\frac{2(p-2)}{p-1}}(1+r)^3(\|w_1\|_{Y_{r_1}}+\|w_2\|_{Y_{r_1}})\|w_1-w_2\|_{Y_{r_1}}\\
&& \l^4(1+r)^{-\frac{2(p-3)}{p-1}}(1+r)^{\frac{9}{2}}\|w_2\|_{Y_{r_1}}^2\|w_1-w_2\|_{Y_{r_1}} +\lambda^2(1+r)^{\frac 3 2}\|w_1-w_2\|_{Y_{r_1}} \\
&\lesssim&  \l^2(1+r)^{-\frac{1}{2}}\Big((1+r)^{\frac{2}{p-1}+\frac{3}{2}}+ \l^2(1+r)^{\frac{4}{p-1}+3}+(1+r)^2\Big)\|w_1-w_2\|_{Y_{r_1}}\\
&\lesssim&  \l^2(1+r)^{-\frac{1}{2}}\Big((1+r)^{-s_c+3}+ \l^2(1+r)^{-2s_c+6}+(1+r)^2\Big)\|w_1-w_2\|_{Y_{r_1}}\\
&\lesssim&  r_1^2\l^2(1+r)^{-\frac{1}{2}}\Big(1+ \l^2r_1^2\Big)\|w_1-w_2\|_{Y_{r_1}}\lesssim  r_1^2\l^2(1+r)^{-\frac{1}{2}}\|w_1-w_2\|_{Y_{r_1}}
\eee
which concludes the proof of \eqref{esenneoebis} and Proposition \ref{prop:interiorsolution}.
\end{proof}

%%%%%%%%%%%%%%%%%%%%%%%%%%%%%%%%%%%%%%%%%%%%%%%%%%

\subsection{The matching}\label{sec:matching}

%%%%%%%%%%%%%%%%%%%%%%%%%%%%%%%%%%%%%%%%%%%%%%%%%%

We now construct a solution to \eqref{eq:selfsimilareq} by matching the exterior solution to \eqref{eq:selfsimilareq} constructed in section \ref{sec:exteriorsolution} on $[r_0,+\infty)$ to the interior solution to \eqref{eq:selfsimilareq} constructed in section \ref{sec:interiorsolution} on $[0, r_0]$. The oscillations \eqref{cejnenieie} allow to perform the matching at $r_0$ for a quantized sequence of the small parameter $\e$ introduced in Proposition \ref{prop:exteriorsolution}.

\begin{proposition}[Existence of a countable number of smooth selfsimilar profiles]\label{prop:constructionPhin} 
There exists $N\in\mathbb{N}$ large enough so that for all $n\geq N$, there exists a smooth solution $\Phi_n$ to  \eqref{eq:selfsimilareq} such that $\Lambda \Phi_n$ vanishes exactly $n$ times.
\end{proposition}

\begin{proof}
{\bf step 1} Initialization. Since 
$$\psi_1(r)=\frac{c_3\sin(\omega\log(r)+c_4)}{r^{\frac{1}{2}}}+O\left(r^{\frac{3}{2}}\right)\textrm{ as }r\to 0, \ \ c_3\neq 0$$
we compute
\bee
\Lambda \psi_1(r) = c_3\frac{(1-s_c)\sin(\om\log(r)+c_4)+\om\cos(\om\log(r)+c_4)}{r^{\frac{1}{2}}}+O\left(r^{\frac{3}{2}}\right)\textrm{ as }r\to 0.
\eee
We may therefore choose $0<r_0\ll1$ such that  
\be
\label{intiialirnot}
\psi_1(r_0) = \frac{c_3}{r_0^{\frac{1}{2}}}+O\left(r_0^{\frac{3}{2}}\right),\,\,\,\, \Lambda \psi_1(r_0) = \frac{c_3(1-s_c)}{r_0^{\frac{1}{2}}}+O\left(r_0^{\frac{3}{2}}\right), 
\ee
and Proposition \ref{prop:exteriorsolution} and Proposition \ref{prop:interiorsolution} apply. We therefore choose $\ep$ and $\l$ such that
$$0<\ep\ll r_0^{s_c-1},\,\,\,\, 0<\l\leq r_0,$$
and have from Proposition \ref{prop:exteriorsolution} an exterior solution $u_{ext}$ to 
$$-\Delta u_{ext}+\Lambda u_{ext}-u_{ext}^p,\,\,\, r\geq r_0$$
such that 
$$u_{ext}[\ep]=\Phi_*+\ep\psi_1+\ep w$$
and
\be
\label{estsgnnone}
\|w\|_{X_{r_0}}\lesssim \ep r_0^{1-s_c},\,\,\,\, \|\Lambda w\|_{X_{r_0}}\lesssim \ep r_0^{1-s_c}.
\ee
We also have from Proposition \ref{prop:interiorsolution} an interior solution $u_{int}$ to 
$$-\Delta u_{int}+\Lambda u_{int}-u_{int}^p,\,\,\, 0\leq r\leq r_0$$
such that 
\bee
u_{int}[\l] &=& \frac{1}{\l^{\frac{2}{p-1}}}(Q+\l^2 T_1)\left(\frac{r}{\l}\right).
\eee
with
\be \label{estimation T1}
\|T_1\|_{Y_{\frac{r_0}{\l}}}\lesssim 1.
\ee
We now would like to match the two solutions at $r=r_0$ which is equivalent to requiring that $$u_{ext}(r_0)-u_{int}(r_0)=0 \ \  \text{and} \ \ u_{ext}'(r_0)-u_{int}'(r_0).$$

\noindent{\bf step 2} Matching the functions. We introduce the map
$$\mathcal{F}[r_0](\ep, \l) := u_{ext}[\ep](r_0)-u_{int}[\l](r_0).$$
We compute
\bee
\pr_\ep \mathcal{F}[r_0](\ep, \l) & = & \pr_\ep u_{ext}[\ep](r_0)= \psi_1(r_0)+w(r_0)+\ep\pr_\ep w(r_0).
\eee
In particular, since $w_{|_{\ep=0}}=0$ and $\|\pr_\ep w_{|_{\ep=0}}\|_{X_{r_0}}\lesssim r_0^{1-s_c}$ in view of Proposition \ref{prop:exteriorsolution}, we have
$$\pr_\ep \mathcal{F}[r_0](0, 0) = \psi_1(r_0) \neq 0$$
since we assumed that $\psi_1(r_0)\neq 0$. Also, in view of the asymptotic behavior of $Q$ at infinity, we have as $\l\to 0_+$
\bee
\left|\frac{1}{\l^{\frac{2}{p-1}}}(Q-\Phi_*+\l^2 T_1)\left(\frac{r_0}{\l}\right)\right| &\lesssim & \frac{1}{\l^{\frac{2}{p-1}}}\left(\frac{1}{r^{\frac{1}{2}}}+\frac{\l^2r^2}{r^{\frac{1}{2}}}\right)\left(\frac{r_0}{\l}\right)\lesssim  \frac{\l^{\frac{1}{2}-\frac{2}{p-1}}}{r_0^{\frac{1}{2}}}\lesssim \frac{\l^{s_c-1}}{{r_0^{\frac{1}{2}}}} 
\eee
and hence, since $s_c>1$, we infer
$$\lim_{\l\to 0_+}\frac{1}{\l^{\frac{2}{p-1}}}(Q-\Phi_*+\l^2 T_1)\left(\frac{r_0}{\l}\right)=0.$$
Since 
$$\frac{1}{\l^{\frac{2}{p-1}}}\Phi_*\left(\frac{r_0}{\l}\right)=\Phi_*(r_0),$$
this yields
$$\mathcal{F}[r_0](0, 0) = \Phi_*(r_0)-\Phi_*(r_0)=0.$$
We may thus apply the implicit function theorem\footnote{We actually apply the implicit function theorem to 
$$\widetilde{\mathcal{F}}[r_0](\ep, \mu):=\mathcal{F}(\ep, \mu^{\frac{1}{s_c-1-\delta}})$$
for any $0<\delta<s_c-1$ so that $\widetilde{\mathcal{F}}\in \mathcal{C}^1$. This yields the existence of $\widetilde{\ep}\in \mathcal{C}^1$ and we choose $\ep(\l)=\widetilde{\ep}(\l^{s_c-1-\delta})$ so that $\ep$ belongs indeed to $\mathcal{C}^{\min(1,(s_c-1)_-)}$.} which yields the existence of $\l_0>0$ and a $\mathcal{C}^{\min(1,(s_c-1)_-)}$ function $\ep(\l)$ defined on $[0,\l_0)$  such that $\mathcal{F}(\ep(\l),\l)=0$ and hence
\bee
u_{ext}[\ep(\l)](r_0) = u_{int}[\l](r_0) \ \textrm{ on }[0,\l_0).\\
\eee

\noindent{\bf step 3} Control of $\e(\l)$. We claim for $\l\in [0,\l_0)$
\be
\label{estimateellamnbdabis}
\ep(\l)=\frac{1}{\psi_1(r_0)\l^{\frac{2}{p-1}}}(Q-\Phi_*)\left(\frac{r_0}{\l}\right)+O\left[\l^{s_c-1}(r^2_0+\l^{s_c-1}r_0^{1-s_c})\right].
\ee
Indeed, by construction
$$u_{ext}[\ep(\l)](r_0) = u_{int}[\l](r_0)$$
which is equivalent to 
\be
\label{cnecnnoe}
\ep(\l)\psi_1(r_0) +\ep(\l)w(r_0) = \frac{1}{\l^{\frac{2}{p-1}}}(Q-\Phi_*+\l^2 T_1)\left(\frac{r_0}{\l}\right).
\ee
We infer from \fref{intiialirnot}, \eqref{estsgnnone}, \fref{estimation T1} and the asymptotic of $Q$:
$$
\ep(\l)\psi_1(r_0) + \ep(\l) w(r_0) =\ep(\l)\frac{c_3}{r_0^{\frac 1 2}}\left(1+O(r_0^2)+O(\ep(\l)r_0^{1-s_c}) \right),
$$
$$
\frac{1}{\l^{\frac{2}{p-1}}}\left|Q-\Phi_*+\l^2 T_1)\left(\frac{r_0}{\l}\right)\right|\lesssim \frac{\l^{s_c-1}}{r_0^{\frac 1 2}}(1+O(r_0^2)).
$$
This first yields using \fref{estrovvove}
\be
\label{estimateellamnbda}
|\e(\l)|\lesssim \l^{s_c-1} .
\ee
which reinjected into \eqref{cnecnnoe} yields \eqref{estimateellamnbdabis}.\\

\noindent{\bf step 4} Computation of the spatial derivatives. We consider the difference of spatial derivatives at $r_0$ for $\l\in[0,\l_0)$
\bee
\mathcal{G}[r_0](\l) := u_{ext}[\ep(\l)]'(r_0)-u_{int}[\l]'(r_0)
\eee
and claim the leading order expansion:
\bea
\label{leadingorderderivative}
\mathcal{G}[r_0](\l) &=& \l^{s_c-1}\left[\frac{c_1c_3\om}{\psi_1(r_0)r_0^2}\sin\left(-\om\log(\l)+c_2-c_4\right)\right.\\
\nonumber &+& \left.O\left(r_0^{-s_c-\frac{1}{2}}\l^{s_c-1}+r_0^{\frac{1}{2}}\right)\right].
\eea
Indeed, 
\bee
\mathcal{G}[r_0](\l) &=&  \ep(\l)\psi_1'(r_0) +\ep(\l)w'(r_0) - \frac{1}{\l^{\frac{2}{p-1}+1}}(Q'-\Phi_*'+\l^2 T_1')\left(\frac{r_0}{\l}\right).
\eee
From \eqref{estimateellamnbda}, \eqref{wihfeihwwomeha}:
\bee
| \ep(\l) w'(r_0)| &\lesssim & \l^{s_c-1}|w'(r_0)|\lesssim \l^{2(s_c-1)} r_0^{-\frac 12 -s_c}
\eee
and from \eqref{estinterirotone}
\bee
\left|\frac{1}{\l^{\frac{2}{p-1}+1}}\l^2 T_1'\left(\frac{r_0}{\l}\right)\right| &\lesssim &  r_0^{\frac{1}{2}}\l^{s_c-1}
\eee
and hence using \eqref{estimateellamnbdabis}, \eqref{intiialirnot}:
\bee
\mathcal{G}[r_0](\l) &=&\ep(\l)\psi_1'(r_0)  - \l^{s_c-1}\frac{1}{\l^{\frac{3}{2}}}(Q'-\Phi_*')\left(\frac{r_0}{\l}\right)+O\left(\left(r_0^{-\frac{3}{2}}\l^{s_c-1}+r_0^{\frac{1}{2}}\right)\l^{s_c-1}\right)\\
&=&  \l^{s_c-1}\left(\frac{1}{\l^{\frac{1}{2}}\psi_1(r_0)}(Q-\Phi_*)\left(\frac{r_0}{\l}\right)\psi_1'(r_0)  - \frac{1}{\l^{\frac{3}{2}}}(Q'-\Phi_*')\left(\frac{r_0}{\l}\right)\right)\\
&&+O\left(\left( r_0^{-s_c-\frac 12}\l^{s_c-1}+r_0^{\frac{1}{2}}\right)\l^{s_c-1}\right)\\
&=& \frac{1}{r_0^{\frac{1}{2}}\psi_1(r_0)}\l^{s_c-1}\left\{\left(\frac{r_0}{\l}\right)^{\frac{1}{2}}(Q-\Phi_*)\left(\frac{r_0}{\l}\right)\psi_1'(r_0)  - \left(\frac{r_0}{\l}\right)^{\frac{3}{2}}(Q'-\Phi_*')\left(\frac{r_0}{\l}\right)\frac{\psi_1(r_0)}{r_0}\right\}\\
&&+O\left(\left( r_0^{-s_c-\frac 12}\l^{s_c-1}+r_0^{\frac{1}{2}}\right)\l^{s_c-1}\right).
\eee

Recall that
\bee
&&\psi_1(r)=\frac{c_3\sin(\omega\log(r)+c_4)}{r^{\frac{1}{2}}}+O\left(r^{\frac{3}{2}}\right)\textrm{ as }r\to 0,\\
&&\psi_1'(r)= -\frac{c_3\sin(\omega\log(r)+c_4)}{2r^{\frac{3}{2}}}+\frac{c_3\omega\cos(\omega\log(r)+c_4)}{r^{\frac{3}{2}}}+O\left(r^{\frac{1}{2}}\right)\textrm{ as }r\to 0,\\
&&Q(r)-\Phi_*(r)=\frac{c_1\sin\left(\om\log(r)+c_2\right)}{r^{\frac{1}{2}}}+O\left(\frac{1}{r^{s_c-\frac{1}{2}}}\right)\textrm{ as }r\to +\infty,\\
&&
Q'(r)-\Phi'_*(r)= -\frac{c_1\sin\left(\om\log(r)+c_2\right)}{2r^{\frac{3}{2}}}+\frac{c_1\omega\cos\left(\om\log(r)+c_2\right)}{r^{\frac{3}{2}}}+O\left(\frac{1}{r^{s_c+\frac{1}{2}}}\right)\textrm{ as }r\to +\infty
\eee
and hence:
\bee
&&\left(\frac{r_0}{\l}\right)^{\frac{1}{2}}(Q-\Phi_*)\left(\frac{r_0}{\l}\right)\psi_1'(r_0)  - \left(\frac{r_0}{\l}\right)^{\frac{3}{2}}(Q'-\Phi_*')\left(\frac{r_0}{\l}\right)\frac{\psi_1(r_0)}{r_0}\\
&=& \frac{c_1c_3}{r_0^{\frac{3}{2}}}\Bigg(\sin\left(\om\log(r_0)-\om\log(\l)+c_2\right)\left( -\frac{\sin(\omega\log(r_0)+c_4)}{2}+\omega\cos(\omega\log(r_0)+c_4)\right)\Bigg)\\
&& -\left(-\frac{\sin\left(\om\log(r_0)-\om\log(\l)+c_2\right)}{2}+\omega\cos\left(\om\log(r_0)-\om\log(\l)+c_2\right)\right)\sin(\omega\log(r_0)+c_4)\\
&& +O\left(r_0^{\frac{1}{2}}+\l^{s_c-1}r_0^{-s_c-\frac{1}{2}}\right)\\
&=& \frac{c_1c_3\om}{r_0^{\frac{3}{2}}}\Bigg(\sin\left(\om\log(r_0)-\om\log(\l)+c_2\right)\cos(\omega\log(r_0)+c_4) \\
&& -\cos\left(\om\log(r_0)-\om\log(\l)+c_2\right)\sin(\omega\log(r_0)+c_4)\Bigg)+O\left(r_0^{\frac{1}{2}}+\l^{s_c-1}r_0^{-s_c-\frac{1}{2}}\right)\\
&=& \frac{c_1c_3\om}{r_0^{\frac{3}{2}}}\sin\left(-\om\log(\l)+c_2-c_4\right) +O\left(r_0^{\frac{1}{2}}+\l^{s_c-1}r_0^{-s_c-\frac{1}{2}}\right). 
\eee
The collection of above bounds and \fref{intiialirnot} yields \eqref{leadingorderderivative}.\\

\noindent{\bf step 5} Discrete matching. For $\de_0>0$ a small enough universal constant such that $\de_0\geq r_0$ to be chosen later, we consider
\bea\label{eq:definitionoflambdakpm}
\l_{k,+}=\exp\left(\frac{-k\pi-c_4+c_2+\de_0}{\om}\right) ,\,\,\,\, \l_{k,-}=\exp\left(\frac{-k\pi-c_4+c_2-\de_0}{\om}\right).
\eea
From 
$$\lim_{k\to +\infty}\l_{k,\pm}=0,$$
there holds for $k\geq k_0$ large enough:
$$0<\cdots<\l_{k,+}<\l_{k,-}<\cdots<\l_{k_0,+}<\l_{k_0,-}\leq \l_0$$
With the above definition of $\l_{k,\pm}$, we have for all $k\geq k_0$
$$\sin\left(-\om\log(\l_{k,+})+c_2-c_4\right)=(-1)^k\sin(\de_0),\,\,\,\, \sin\left(-\om\log(\l_{k,-})+c_2-c_4\right)=-(-1)^k\sin(\de_0),$$
and hence
\bee
\mathcal{G}[r_0](\l_{k,\pm}) &=& \pm (-1)^k\l_{k,\pm}^{s_c-1}\left(\frac{c_1c_3\om}{\psi_1(r_0)r_0^2}\sin(\de_0)+O\left(r_0^{-s_c-\frac{1}{2}}\l_{k,\pm}^{s_c-1}+r_0^{\frac{1}{2}}\right)\right).
\eee
Since $\de_0\geq r_0$, this yields for $r_0$ small enough and for any $k\geq k_0$ large enough:
$$\mathcal{G}[r_0](\l_{k,-})\mathcal{G}[r_0](\l_{k,+})<0.$$
Since the function $\l\to \mathcal{G}[r_0](\l)$ is continuous, we infer from the mean value theorem applied to the intervals $[\l_{k,+}, \l_{k,-}]$ the existence of $\mu_k$ such that
$$\l_{k,+}<\mu_k<\l_{k,-}\textrm{ and }\mathcal{G}[r_0](\mu_k)=0\textrm{ for all }k\geq k_0.$$ 
Finally, for $k\geq k_0$, we have
$$\mathcal{F}[r_0](\ep(\mu_k),\mu_k)=0\textrm{ and } \mathcal{G}[r_0](\mu_k)=0$$
which yields
\bee
u_{ext}[\ep(\mu_k)](r_0) = u_{int}[\mu_k](r_0)\textrm{ and }u_{ext}[\ep(\mu_k)]'(r_0) = u_{int}[\mu_k]'(r_0).
\eee
and hence the function
$$
u_k(r):=\left\{\begin{array}{ll}
u_{int}[\mu_k](r) & \textrm{ for }0\leq r\leq r_0,\\
u_{ext}[\ep(\mu_k)](r) & \textrm{ for } r> r_0
\end{array}\right.
$$
is smooth and satisfies \eqref{eq:selfsimilareq}.\\
The rest of the proof is devoted to counting the number of zeroes of $\Lambda u_k$ and showing that this number is an unambiguous way of counting the number of self similar solutions $u_k$ as $k\to +\infty$.\\

\noindent{\bf step 6} Zeroes of $\Lambda u_{ext}[\ep]$. We claim that
\be
\label{cenjnenoeo}
\Lambda u_{ext}[\ep]\ \ \mbox{has as many zeros as}\ \ \Lambda\psi_1 \ \ \mbox{on}\ \ r\geq r_0.
\ee
Indeed, $\Lambda \psi_1+\Lambda w$ does not vanish on $[R_0,+\infty)$ for $R_0$ large enough from \eqref{refinedbeahviourpsione} and the uniform bound \eqref{wihfeihwwomeha}. Moreover, $\Lambda\psi_1(r_0)\neq 0$ from the normalization \eqref{intiialirnot}, and the absolute derivative of $\Lambda\psi_1$ at any of its zeroes is uniformly lower bounded using \eqref{eqpsione}, \eqref{behvoriginbis}, and hence the uniform smallness  \eqref{wihfeihwwomeha} $$\|\Lambda w\|_{X_{r_0}}\lesssim \ep r_0^{1-s_c}\ll 1$$ yields the claim.\\

\noindent{\bf step 7} Zeroes of $\Lambda u_{int}[\mu_k]$. We now claim that 
\be
\label{innerzeroes}
\Lambda u_{int}[\mu_k]\ \ \mbox{has as many zeros as}\ \  \Lambda Q \ \ \mbox{on}\ \  0\leq r\leq r_0/\mu_k.
\ee
Indeed, recall that
\bee
\Lambda u_{int}[\mu_k](r) &=& \frac{1}{\mu_k^{\frac{2}{p-1}}}(\Lambda Q+\mu_k^2 \Lambda T_1)\left(\frac{r}{\mu_k}\right).
\eee
We now claim 
\be
\label{tobeprovedlambdaq}
\left(\frac{r_0}{\mu_k}\right)^{\frac 12} \left|\Lambda Q\left(\frac{r_0}{\mu_k}\right)\right|\gtrsim 1.
\ee
Assume \eqref{tobeprovedlambdaq}, then since the zeros of $\Lambda Q$ are simple, since we have
\bee
\Lambda Q(r) = \frac{c_7\sin\left(\om\log(r)+c_8\right)}{r^{\frac{1}{2}}} +O\left(\frac{1}{r^{s_c-\frac{1}{2}}}\right) \textrm{ as }r\to +\infty,
\eee
since
$$\|\Lambda T_1\|_{Y_{\frac{r_0}{\mu_k}}}=\sup_{0\leq r\leq \frac{r_0}{\mu_k}}(1+r)^{-\frac{3}{2}}|\Lambda T_1|\lesssim 1$$
so that
$$\sup_{0\leq r\leq \frac{r_0}{\mu_k}}(1+r)^{\frac{1}{2}}|\mu_k^2\Lambda T_1|\lesssim r_0^2,$$
and similarily for $\Lambda^2 T_1$, and since
\bee
\Lambda Q(0) = \frac{2}{p-1}\neq 0,
\eee
we conclude that $\Lambda Q+\mu_k^2\Lambda T_1$ has as many zeros as $\Lambda Q$ on $0\leq r\leq r_0/\mu_k$. We deduce that on $0\leq r\leq r_0$, $\Lambda u_{int}[\mu_k]$ has as many zeros as $\Lambda Q$ on $0\leq r\leq r_0/\mu_k$.\\

\noindent{\it Proof of \eqref{tobeprovedlambdaq}}: Recall that
\bee
u_{ext}[\ep(\mu_k)](r_0) = u_{int}[\mu_k](r_0)\textrm{ and }u_{ext}[\ep(\mu_k)]'(r_0) = u_{int}[\mu_k]'(r_0),
\eee
which implies
$$\Lambda u_{ext}[\ep(\mu_k)](r_0) = \Lambda u_{int}[\mu_k](r_0).$$
This yields using \eqref{estimateellamnbdabis}:
\bee
\frac{\ep(\mu_k)}{\mu_k^{s_c-1}} &=& \frac{1}{\psi_1(r_0)\mu_k^{\frac{1}{2}}}(Q-\Phi_*)\left(\frac{r_0}{\mu_k}\right) + O\Big(\mu_k^{s_c-1}r_0^{s_c-1}+r_0^2\Big)
\eee
and differentiating \eqref{cnecnnoe}:
\bee
\frac{\ep(\mu_k)}{\mu_k^{s_c-1}} &=& \frac{1}{\Lambda\psi_1(r_0)\mu_k^{\frac{1}{2}}}\Lambda Q\left(\frac{r_0}{\mu_k}\right) + O\Big(\mu_k^{s_c-1} r_0^{s_c-1}+r_0^2\Big).
\eee
We infer
\bee
 \frac{1}{\psi_1(r_0)\mu_k^{\frac{1}{2}}}(Q-\Phi_*)\left(\frac{r_0}{\mu_k}\right) &=& \frac{1}{\Lambda\psi_1(r_0)\mu_k^{\frac{1}{2}}}\Lambda Q\left(\frac{r_0}{\mu_k}\right)+ O\Big(\mu_k^{s_c-1} r_0^{s_c-1}+r_0^2\Big).
\eee
In view of \eqref{intiialirnot} which we recall below
$$
\psi_1(r_0) = \frac{c_3}{r_0^{\frac{1}{2}}}+O\left(r_0^{\frac{3}{2}}\right),\,\,\,\, \Lambda \psi_1(r_0) = \frac{c_3(1-s_c)}{r_0^{\frac{1}{2}}}+O\left(r_0^{\frac{3}{2}}\right), 
$$
this yields
\bea\label{eq:usefulboundtocountnumberofzeros}
 \left|\left(\frac{r_0}{\mu_k}\right)^{\frac{1}{2}}(Q-\Phi_*)\left(\frac{r_0}{\mu_k}\right)\right| \leq \frac{2}{s_c-1}\left|\left(\frac{r_0}{\mu_k}\right)^{\frac{1}{2}}\Lambda Q\left(\frac{r_0}{\mu_k}\right)\right|+ O\Big(\mu_k^{s_c-1}+r_0^2\Big).\\
\nonumber
\eea
On the other hand,
\be
\label{espenknonv}
Q(r) - \Phi_*(r) = \frac{c_1\sin\left(\om\log(r)+c_2\right)}{r^{\frac{1}{2}}}+O\left(\frac{1}{r^{s_c-\frac{1}{2}}}\right)\textrm{ as }r\to +\infty
\ee
and hence as $r\to +\infty$
\bea\label{eq:equationwithtwostars}
\nonumber\Lambda Q(r) &=& c_1\frac{(1-s_c)\sin(\om\log(r)+c_2)+\om\cos(\om\log(r)+c_2)}{r^{\frac{1}{2}}}+O\left(\frac{1}{r^{s_c-\frac{1}{2}}}\right)\\
&=&c_1\sqrt{(s_c-1)^2+\omega^2}\,\frac{\sin(\om\log(r)+c_2+\alpha_0)}{r^{\frac{1}{2}}}+O\left(\frac{1}{r^{s_c-\frac{1}{2}}}\right)
\eea
where
$$\cos(\alpha_0)=\frac{1-s_c}{\sqrt{(s_c-1)^2+\omega^2}},\,\,\,\, \sin(\alpha_0)=\frac{\omega}{\sqrt{(s_c-1)^2+\omega^2}},\,\,\,\,\alpha_0\in \left(
\frac{\pi}{2}, \pi\right).$$  Thus there exists $r_2>0$ sufficiently small and a constant $\delta_1>0$ sufficiently small only depending on $\omega$ and $s_c-1$ such that for $0<r<r_2$, we have
$$\textrm{dist}\Big(\om\log(r)+c_2+\alpha_0, \pi\mathbb{Z}\Big)<\delta_1\,\Rightarrow\, r^{\frac{1}{2}}|Q(r) - \Phi_*(r)|\geq \frac{4}{s_c-1}r^{\frac{1}{2}}|\Lambda Q(r)|+\frac{c_1\sin(\alpha_0)}{2}.$$
In view of \eqref{eq:usefulboundtocountnumberofzeros}, we infer for $k\geq k_1$ large enough
\bea\label{eq:equationwithonestar}
\textrm{dist}\left(\om\log\left(\frac{r_0}{\mu_k}\right)+c_2+\alpha_0, \pi\mathbb{Z}\right)\geq \delta_1\\
\nonumber
\eea
and \eqref{tobeprovedlambdaq} is proved.\\

\noindent{\bf step 8} Counting. We have so far obtained
\bee
&&\#\{r\geq 0\textrm{ such that }\Lambda u_k(r)=0\}\\
 &=& \#\left\{0\leq r\leq \frac{r_0}{\mu_k}\textrm{ such that }\Lambda Q(r)=0\right\}+ \#\{r>r_0\textrm{ such that }\Lambda\psi_1(r)=0\}
\eee
which implies
\bee
\#\{r\geq 0\textrm{ such that }\Lambda u_{k+1}(r)=0\} &=& \#\{r\geq 0\textrm{ such that }\Lambda u_k(r)=0\} + \# A_k,
\eee
with
\bee
A_k := \left\{\frac{r_0}{\mu_k}< r\leq \frac{r_0}{\mu_{k+1}}\textrm{ such that }\Lambda Q(r)=0\right\}.\\
\eee
We claim for $k\geq k_0$ large enough:
\be
\label{estcarajo}
\# A_k=1
\ee
which by possibly shifting the numerotation by a fixed amount ensures that $\Lambda u_k$ vanishes exactly $k$ times.\\

\noindent{\em Upper bound}. We first claim 
\be
\label{upperbound}
\# A_k\le 1
\ee 
Recall that
\bea\label{eq:equationwiththreestars}
\Lambda Q(r) = \frac{c_7\sin\left(\om\log(r)+c_8\right)}{r^{\frac{1}{2}}}+O\left(\frac{1}{r^{s_c-\frac{1}{2}}}\right)\textrm{ as }r\to +\infty,
\eea
so that there exists $R\geq 1$ large enough such that 
\bea\label{eq:equationwithfourstars}
\{r\geq R\,/\,\Lambda Q(r)=0\}= \{r_q,\,\,\,q\geq q_1\},\,\,\omega\log(r_q)+c_8 =q\pi +O\left(\frac{1}{r_q^{s_c-1}}\right).
\eea
In view of \eqref{eq:equationwithtwostars} and \eqref{eq:equationwiththreestars}, we have
$$c_2+\alpha_0=c_8$$
and hence, together with \eqref{eq:equationwithonestar} and \eqref{eq:equationwithfourstars}, we infer 
\be
\label{distanceinninimisee}
\inf_{q\geq q_1, k\geq k_1}\left|\log\left(\frac{r_0}{\mu_k}\right)-\log(r_q)\right|\geq \frac{\delta_1}{2\omega}.
\ee
This implies for $k\geq k_1$
\bea\label{eq:inclusioninalargerintervalofAk}
A_k &=& \left\{q\geq q_1\textrm{ such that }r_q\in \left(\frac{r_0}{\mu_k}, \frac{r_0}{\mu_{k+1}}\right)\right\}\\
\nonumber&\subset& \left\{q\geq q_1\textrm{ such that }\log\left(\frac{r_0}{\mu_k}\right)+\frac{\delta_1}{2\omega}\leq \log(r_q) \leq \log\left(\frac{r_0}{\mu_{k+1}}\right)-\frac{\delta_1}{2\omega}\right\}.
\eea
Since $\l_{k,+}<\mu_k<\l_{k,-}$ with $\l_{k,\pm}$ given by \eqref{eq:definitionoflambdakpm}, we have for $k\geq k_1$
\bee
&&\log\left(\frac{r_0}{\mu_{k+1}}\right)-\frac{\delta_1}{2\omega} - \left(\log\left(\frac{r_0}{\mu_k}\right)+\frac{\delta_1}{2\omega}\right) = \log(\mu_k)-\log(\mu_{k+1})-\frac{\delta_1}{\omega}\\
&\leq& \log(\l_{k_+})-\log(\l_{k+1,-}) -\frac{\delta_1}{\omega}\leq \frac{\pi+2\delta_0-\delta_1}{\omega}.
\eee
Also, we have for $q\geq q_1$
\bee
\log(r_{q+1})-\log(r_q) &=& \frac{\pi}{\omega}+O\left(\frac{1}{r_q^{s_c-1}}\right).
\eee
We now choose $\delta_0$ such that
\be
\label{choicedeltao}
0<\delta_0<\frac{\delta_1}{4}.
\ee
Then, we infer
 for $k\geq k_1$
\bee
&& \log\left(\frac{r_0}{\mu_{k+1}}\right)-\frac{\delta_1}{2\omega} - \left(\log\left(\frac{r_0}{\mu_k}\right)+\frac{\delta_1}{2\omega}\right)\leq \frac{\pi}{\omega} -\frac{\delta_1}{2\omega}
\eee
and hence  for $k\geq k_1$ and $q\geq q_1$, we have
\bee
\log(r_{q+1})-\log(r_q) &>&  \log\left(\frac{r_0}{\mu_{k+1}}\right)-\frac{\delta_1}{2\omega} - \left(\log\left(\frac{r_0}{\mu_k}\right)+\frac{\delta_1}{2\omega}\right)
\eee
which in view of \eqref{eq:inclusioninalargerintervalofAk} implies \eqref{upperbound}.\\

\noindent{\em Lower bound}. We now prove \eqref{estcarajo} and assume by contradiction: $$\# A_{k_2}=0.$$
Then, let $q_2\geq q_1$ such that
$$r_{q_2} < \frac{r_0}{\mu_{k_2}} <  \frac{r_0}{\mu_{k_2+1}} < r_{q_2+1}.$$
We infer from \eqref{distanceinninimisee}:
\bea\label{eq:equationwithfivestars}
\log(r_{q_2}) \leq  \log\left(\frac{r_0}{\mu_{k_2}}\right)-\frac{\delta_1}{2\omega} <  \log\left(\frac{r_0}{\mu_{k_2+1}}\right)+\frac{\delta_1}{2\omega} \leq \log(r_{q_2+1}).
\eea
However, we have for $k\geq k_1$
\bee
&& \log\left(\frac{r_0}{\mu_{k_2+1}}\right)+\frac{\delta_1}{2\omega} - \left(\log\left(\frac{r_0}{\mu_{k_2}}\right)-\frac{\delta_1}{2\omega}\right)= \log(\mu_{k_2})-\log(\mu_{k_2+1})+\frac{\delta_1}{\omega}\\
&\geq& \log(\l_{k_2, -})-\log(\l_{k_2+1,+}) +\frac{\delta_1}{\omega}\geq \frac{\pi-2\delta_0+\delta_1}{\omega}\geq  \frac{\pi}{\omega} + \frac{\delta_1}{2\omega}
\eee
in view of our choice \eqref{choicedeltao}. Hence, we infer
\bee
&& \log\left(\frac{r_0}{\mu_{k_2+1}}\right)+\frac{\delta_1}{2\omega} - \left(\log\left(\frac{r_0}{\mu_{k_2}}\right)-\frac{\delta_1}{2\omega}\right) > \log(r_{q_2+1}) -\log(r_{q_2})
\eee
which contradicts \eqref{eq:equationwithfivestars}.\\
 This concludes the proof of Proposition \ref{prop:constructionPhin}.
\end{proof}

We now collect final estimates on the constructed solution $\Phi_n$ which conclude the proof of Proposition \ref{propconstruction}.

\begin{corollary}\label{cor:consequenceforPhinofproponexistence}
Let  $\Phi_n$ the solution to \eqref{eq:selfsimilareq} constructed in Proposition \ref{prop:constructionPhin}. Then there exists a small enough constant $r_0>0$ independent of $n$ such that:\\
\noindent{\em 1. Convergence to $\Phi_*$ as $n\to +\infty$}: 
\be
\label{estoneaprouver}
\lim_{n\to +\infty}\sup_{r\geq r_0}\left(1+r^{\frac{2}{p-1}}\right)|\Phi_n(r) - \Phi_*(r)| = 0.
\ee
\noindent{\em 2. Convergence to $Q$ at the origin}: there holds for some $\mu_n\to 0$ as $n\to +\infty$:
\be
\label{estaprvoeige}
\lim_{n\to +\infty}\sup_{r\leq r_0}\left|\Phi_n(r) -  \frac{1}{\mu_n^{\frac{2}{p-1}}}Q\left(\frac{r}{\mu_n}\right)\right| =0.
\ee
\noindent{\em 3. Last zeroes}: let $r_{0,n}<r_0$ denote the last zero of $\Lambda\Phi_n$ before $r_0$. Then, for $n\geq N$ large enough, we have
$$e^{-\frac{2\pi}{\omega}}r_0 <r_{0,n}<r_0.$$
Let $r_{\Lambda Q, n}<r_0/\mu_n$ denote  the last zero of $\Lambda Q$ before $r_0/\mu_n$, then
$$r_{0,n}=\mu_n r_{\Lambda Q, n}(1+O(r_0^2)).$$
\end{corollary}

\begin{proof}
We choose $r_0>0$ small enough as in the proof of Proposition \ref{prop:constructionPhin}. We start with the proof of the first claim. Recall from the proof of Proposition \ref{prop:constructionPhin} that we have for $r\geq r_0$
\bee
\Phi_n(r) &=& \Phi_*(r)+\ep(\mu_n)\psi_1(r)+\ep(\mu_n)w(r)
\eee
where we have in particular
$$\sup_{r_0\leq r\leq 1}r^{\frac{1}{2}}(|\psi_1|+|w|)+\sup_{r\geq 1}r^{\frac{2}{p-1}}(|\psi_1|+|w|)\lesssim 1$$
and 
$$\lim_{n\to +\infty}\ep(\mu_n)=0.$$
We infer
\bee
&&\sup_{r\geq r_0}\left(1+r^{\frac{2}{p-1}}\right)|\Phi_n(r) - \Phi_*(r)|\\
 &\lesssim & \ep(\mu_n)\left(\sup_{r\geq r_0}(|\psi_1(r)|+|w(r)|)+\sup_{r\geq 1}r^{\frac{2}{p-1}}(|\psi_1(r)|+|w(r)|)\right)\\
&\lesssim& \ep(\mu_n)r_0^{-\frac{1}{2}}
\eee
and hence
$$\lim_{n\to +\infty}\sup_{r\geq r_0}\left(1+r^{\frac{2}{p-1}}\right)|\Phi_n(r) - \Phi_*(r)| = 0.$$

Next, recall from the proof of Proposition \ref{prop:constructionPhin} that we have for $r\leq r_0$
\bee
\Phi_n(r) &=&  \frac{1}{\mu_n^{\frac{2}{p-1}}}(Q+\mu_n^2 T_1)\left(\frac{r}{\mu_n}\right)
\eee
with
\bee
\sup_{0\leq r\leq \frac{r_0}{\mu_n}}(1+r)^{-\frac{3}{2}}|T_1|\lesssim 1.
\eee
We infer for $r\leq r_0$
\bee
\left|\Phi_n(r) -  \frac{1}{\mu_n^{\frac{2}{p-1}}}Q\left(\frac{r}{\mu_n}\right)\right| &\leq& \mu_n^{2-\frac{2}{p-1}} |T_1|\left(\frac{r}{\mu_n}\right)\lesssim \mu_n^{\frac{1}{2}-\frac{2}{p-1}}
\eee
and hence
\be \label{estimation interieure convergence phin Q}
\sup_{r\leq r_0}\left|\Phi_n(r) -  \frac{1}{\mu_n^{\frac{2}{p-1}}}Q\left(\frac{r}{\mu_n}\right)\right| \lesssim\mu_n^{s_c-1}.
\ee
and since $\mu_n\to 0$ as $n\to +\infty$, \eqref{estaprvoeige} is proved.\\
We now estimate the localization of the last zeroes of $\Phi_n$ and $\Lambda Q$ before $r_0$. Recall that
$$\Lambda Q(r)\sim \frac{c_7\sin(\om\log(r)+c_8)}{r^{\frac{1}{2}}}\textrm{ as }r\to +\infty.$$
Since $\sin(\om\log(r)+c_8)$ changes sign on the interval
$$e^{-\frac{3\pi}{2\omega}}\frac{r_0}{\mu_n}\leq r\leq \frac{r_0}{\mu_n},$$
and since $r\gg 1$ on this interval, we infer by the mean value theorem that $\Lambda Q(r)$ has a zero on this interval. In particular, this yields
$$e^{-\frac{3\pi}{2\omega}}\frac{r_0}{\mu_n}\leq r_{\Lambda Q, n}\leq \frac{r_0}{\mu_n}.$$
Also, recall from the proof of Proposition \ref{prop:constructionPhin} that we have for $r\leq r_0$
\bee
\Lambda\Phi_n(r) &=& \frac{1}{\mu_n^{\frac{2}{p-1}}}(\Lambda Q+\mu_n^2\Lambda T_1)\left(\frac{r}{\mu_n}\right),
\eee
Since 
$$\Lambda Q(r)\sim \frac{c_7\sin(\om\log(r)+c_8)}{r^{\frac{1}{2}}}\textrm{ as }r\to +\infty,$$
and 
$$\sup_{0\leq r\leq \frac{r_0}{\mu_n}}(1+r)^{\frac{3}{2}}|\Lambda T_1|\lesssim 1,$$
and since  
$$e^{-\frac{2\pi}{\omega}}r_0\leq r\leq r_0,$$
we have $r/\mu_n\sim r_0/\mu_n\gg 1$ for $n\geq N$ large enough, we infer
\bee
\Lambda\Phi_n(r) &\sim&\frac{c_7\sin(\om\log(r)-\om\log(\mu_n)+c_8)+O(r_0^2)}{\mu_n^{\frac{2}{p-1}}\left(\frac{r}{\mu_n}\right)^{\frac{1}{2}}}.
\eee
This yields
$$\Big|\om\log(r_{0,n})-\om\log(\mu_n)+c_8 -(\om\log(r_{\Lambda Q,n})+c_8)\Big| \lesssim r_0^2$$
and hence
\bee
r_{0,n} &=& \mu_n r_{\Lambda Q,n} e^{O(r_0^2)}\\
&=& \mu_n r_{\Lambda Q,n}(1+ O(r_0^2)).
\eee
Furthermore, since we have
$$e^{-\frac{3\pi}{2\omega}}\frac{r_0}{\mu_n}\leq r_{\Lambda Q, n}\leq \frac{r_0}{\mu_n},$$
we deduce
$$e^{-\frac{2\pi}{\omega}}r_0\leq r_{0, n}\leq r_0.$$
This concludes the proof of the corollary.
\end{proof}

%%%%%%%%%%%%%%%%%%%%%%%%%%%%%%%%%%%%%%%%%%

\section{Spectral gap in weighted norms}
\label{sectionspectral}
%%%%%%%%%%%%%%%%%%%%%%%%%%%%%%%%%%%%%%%%%%

Our aim in this section is to produce a spectral gap for the linearized operator corresponding to \eqref{ellipticequation} around $\Phi_n$:
\bea\label{defh} 
\mathcal{L}_n:=-\Delta +\Lambda -p\Phi_n^{p-1}.
\eea 
Recall \eqref{scalalrho}, then $\mathcal{L}_n$  is self adjoint for the $L^2_\rho$ scalar product. Moreover, from \eqref{weightedesimate} and the local compactness of the Sobolev embeddings  $H^1(|x|\leq R)\hookrightarrow L^2(|x|\leq R)$, and the fact that $\Phi_n\in L^\infty$, the selfadjoint operator $\mathcal{L}_n+M_n$ for the measure $\rho dx$ is for $M_n\geq 1$ large enough invertible with compact resolvent. Hence $\mathcal{L}_n$ is diagonalizable in a Hilbert basis of $L^2_\rho$, and we claim the following sharp spectral gap estimate:

\begin{proposition}[Spectral gap for $\mathcal{L}_n$]
\label{prop:spectral}
Let $n>N$ with $N\gg 1$ large enough, then the following holds:\\
\noindent\emph{1. Eigenvalues.} The spectrum of $\mathcal{L}_n$ is given by 
\be
\label{eignevalurenvon}
-\mu_{n+1,n}<\dots<-\mu_{2,n}<-\mu_{1,n}=-2<-\mu_{-1,n}=-1 <0<\l_{0,n}<\l_{1,n}<\dots
\ee 
with 
\be
\label{spectralgapeigenvalue}
\l_{j,n}>0\textrm{ for all }j\geq 0 \textrm{ and }\lim_{j\to+\infty}\l_{j,n}=+\infty.
\ee
The eigenvalues $(-\mu_{j,n})_{1\leq j\leq n+1}$ are simple and associated to spherically symmetric eigenvectors 
$$\psi_{j,n}, \ \ \|\psi_{j,n}\|_{L^2_\rho}=1, \ \ \psi_{1,n}=\frac{\Lambda \Phi_n}{\|\Lambda \Phi_n\|_{\rho}},$$ 
and the eigenspace for $\mu_{-1,n}$ is spanned by 
\be
\label{tranaltionmode}
\psi^k_{-1,n}=\frac{\pa_k\Phi_n}{\|\pa_k\Phi_n\|_\rho}, \ \ 1\leq k\leq 3.
\ee
Moreover, there holds as $r\to +\infty$
\be
\label{estgrowth}
|\pa_k\psi_{j,n}(r)|\lesssim (1+r)^{-\frac{2}{p-1}-\mu_{j,n}-k}, \ \ 1\leq j\leq n+1, \ \ k\geq 0.
\ee

\noindent\emph{2. Spectral gap.} There holds for some constant $c_n>0$:
\bea\label{coerciviteapoids}
\forall \e\in H^1_\rho, \,\, 
(\mathcal{L}_n\e,\e)_\rho\geq c_n\|\ep\|_{H^1_\rho}^2-\frac{1}{c_n}\left[\sum_{j=1}^{n+1}(\e,\psi_{j,n})^2_{\rho}+\sum_{k=1}^3(\e,\psi_{0,n}^k)^2_{\rho}\right].
\eea
\end{proposition}

In other words, $\mathcal L_n$ admits $n+1$ instability directions when $\Lambda \Phi_n$ vanishes $n$ times, and $0$ is never in the spectrum. Moreover, there are no additional non radial instabilities apart from the trivial translation invariance \eqref{tranaltionmode}.\\

The rest of this section is devoted to preparing the proof of Proposition \ref{prop:spectral} which is completed in section \ref{sec:proofopropspectral}. 

%%%%%%%%%%%%%%%%%%%%%%%%%%%%%%%%%%%%%%%%%%

\subsection{Decomposition in spherical harmonics}

%%%%%%%%%%%%%%%%%%%%%%%%%%%%%%%%%%%%%%%%%%

 We first recall some basic facts about spherical harmonics. Spherical harmonics are the eigenfunctions of the Laplace-Beltrami operator on the sphere $\mathbb{S}^2$. The spectrum of this self-adjoint operator with compact resolvent is 
$$\left\{ - m(m+1), \ m\in \mathbb N\right\}.$$ 
For each $m\in \mathbb{N}$ the eigenvalue $m(m+1)$ has geometric multiplicity $2m+1$. We then denote the associated orthonormal family of eigenfunctions by $(Y^{(m,k)})_{m\in \mathbb N, \ -m\leq k \leq m}$ so that we have
$$
L^2(\mathbb{S}^2)= \underset{m=0}{\overset{+\infty}{\oplus}} \text{Span}\left\langle Y^{(m,k)}, \ -m\leq k \leq m\right\rangle
$$
and
\be \label{intro:eq:def Ynk}
-\Delta_{\mathbb S^2}Y^{(m,k)}=m(m+1)Y^{(m,k)}, \,\,\,\, \int_{\mathbb{S}^2} Y^{(m,k)}Y^{(m',k')}d\sigma_{\mathbb{S}^2}=\delta_{(m,k),(m',k')}.
\ee

In particular, $u\in H^1_{\rho}$ is decomposed as
$$u=\sum_{m=0}^{+\infty}\sum_{k=-m}^m u_{m,k}Y^{(m,k)}$$
where $u_{m,k}$ are radial functions satisfying the Parseval formula
$$\|u\|_{\rho}^2=\sum_{m=0}^{+\infty}\sum_{k=-m}^m \|u_{m,k}\|^2_{\rho}.$$
This allows us to write
\bea\label{eq:decompositionLninLmn}
(\mathcal{L}_n(u),u)_{\rho} &=&  \sum_{m=0}^{+\infty}\sum_{k=-m}^m (\mathcal{L}_{n,m}(u_{m,k}),u_{m,k})_{\rho}
\eea
where we recall
 $$\mathcal{L}_{n,m}:= -\partial_{rr}-\frac{2}{r}\partial_r+\frac{2}{p-1}+r\pr_r +\frac{m(m+1)}{r^2}-p\Phi_n^{p-1}.$$ 
We also recall for further use the definition of the operators:
 \bee
 && \mathcal{L}_{\infty,m}:= -\partial_{rr}-\frac{2}{r}\partial_r+\frac{2}{p-1}+r\pr_r +\frac{m(m+1)}{r^2}-p\Phi_*^{p-1},\\
 && H_m:= -\partial_{rr}-\frac{2}{r}\partial_r+\frac{m(m+1)}{r^2}-pQ^{p-1}.
 \eee

%%%%%%%%%%%%%%%%%%%%%%%%%%%%%%%%%%%%%%%%%%

\subsection{Linear ODE analysis}

%%%%%%%%%%%%%%%%%%%%%%%%%%%%%%%%%%%%%%%%%%

We compute in this section the fundamental solutions of $\mathcal L_{n,m}$, $H_m$ and we recall the behavior of the eigenvalues of $\mathcal L_\infty$. The claims are standard and follow from a classical ODE perturbation analysis using in an essential way the uniform bound \eqref{behviourselfsimlocal}.

\begin{lemma}[Fundamental solution for $\mathcal L_{n,m}$, $H_m$]
\label{lemma:asymptoticbehavioratinfinityforLmn}
Let $m\geq 1$. Let $\Delta_m>0$ be given by \eqref{eq:definitionofdiscriminantDeltam}.\\
\noindent{\em 1. Basis for $\mathcal L_{n,m}$}. Let $\phi_{n,m}$ be the solution to $\mathcal L_{n,m}\phi_{n,m}=0$ with the behaviour at the origin 
\be
\label{cneoneoenove}
\varphi_{n,m}=r^m[1+O(r^2)] \ \  \mbox{as} \ \ r\to 0,
\ee
then
\be
\label{beahvriour}\varphi_{n,m} \sim \frac{c_1}{r^{\frac{2}{p-1}}}+c_2 r^{\frac{2}{p-1}-3}e^{\frac{r^2}{2}}\textrm{ as } r\to+\infty, \ \ (c_1,c_2)\neq (0,0).
\ee
\noindent{\em 2. Basis for $H_{1}$}: let $m=1$, then there exists a fundamental basis $(\nu_1,\phi_1)$ with
\be
\label{beahviorunumbis}\nu_1(r)=\frac{Q'(r)}{Q''(0)}\left|\begin{array}{ll} =r[1+O(r^2)]\ \ \mbox{as} \ \ r\to 0\\ \sim \frac{c_{1,+}}{r^{\frac{1+\sqrt{\Delta_1}}{2}}} \ \ \mbox{as}\ \ r\to +\infty\end{array}\right.\ee
and
\be
\label{havviournuone}
\phi_1(r)=\left|\begin{array}{ll}\frac{1}{r^2}[1+O(r^2)]\ \ \mbox{as} \ \ r\to 0\\ \sim \frac{c_{1,-}}{r^{\frac{1-\sqrt{\Delta_1}}{2}}} \ \ \mbox{as}\ \ r\to +\infty,\ \ c_{1,-}\neq 0.\end{array}\right.
\ee

\noindent{\em 2. Basis for $H_{m}$}:  let $m\ge 2$, then there exists a fundamental basis $(\nu_m,\phi_m)$ with
\be
\label{beahviorunum}\nu_m\left|\begin{array}{ll}=r^m[1+O(r^2)] \ \ \mbox{as}\ \ r\to 0\\
\sim \frac{c_{m,-}}{r^{\frac{1-\sqrt{\Delta_m}}{2}}}\ \ \textrm{ as }r\to +\infty,\ \ c_{m,-}>0\end{array}\right.
\ee
and
\be
\label{havviournuonebis}
\phi_m(r)=\left|\begin{array}{ll}\frac{ 1}{r^{1+m}}[1+O(r^2)]\ \ \mbox{as} \ \ r\to 0\\ \sim \frac{c_{m,+}}{r^{\frac{1+\sqrt{\Delta_m}}{2}}} \ \ \mbox{as}\ \ r\to +\infty, \ \ c_{m,+}\neq 0.\end{array}\right.
\ee

\noindent{\em 4. Positivity}:  
\be
\label{psitioivoovov}
\nu_m(r)>0 \ \ \mbox{on}\ \ (0,+\infty).
\ee
\noindent{\em 5. Uniform closeness}: Fix $m\geq 1$. There exists a sequence\footnote{$(\mu_n)_{n\geq N}$ is the same sequence of scales as in \fref{behviourselfsimlocal} in Proposition \ref{propconstruction} and Corollary \ref{cor:consequenceforPhinofproponexistence}.} $\mu_n\to 0$ as $n\to +\infty$ such that for $n\geq N$ large enough
\be\label{vnvnonvnornor}
\sup_{0\leq r\leq r_0}\frac{\left|\mu_n^{-m}\varphi_{n,m}(r) - \nu_m\left(\frac{r}{\mu_n}\right)\right|}{\left|\nu_m\left(\frac{r}{\mu_n}\right)\right|} +\sup_{0\leq r\leq r_0}\frac{\left|\mu_n^{-m+1}\varphi_{n,m}'(r) - \nu_m'\left(\frac{r}{\mu_n}\right)\right|}{\left|\nu_m'\left(\frac{r}{\mu_n}\right)\right|}\lesssim r_0^2.
\ee
\end{lemma}

The uniform in $n$ bound \eqref{vnvnonvnornor} follows from the uniform control \eqref{behviourselfsimlocal} using a standard ODE analysis. We provide a detailed proof of Lemma \ref{lemma:asymptoticbehavioratinfinityforLmn} in Appendix \ref{appendixlemmaun} for the sake of completeness.\\
 
We now detail the structure of the smooth zero of $\mathcal L_{n,0}$ which is the key to the counting of non positive eigenvalues. Let $\varphi_{n,0}$ be the solution to
\be
\label{defphinoneveev}
\mathcal{L}_{n,0}(\varphi_{n,0}) = 0,\,\,\,\, \varphi_{n,0}(0)=1,\,\,\,\,\varphi_{n,0}'(0)=0.
\ee
We recall that $r_{0,n}<r_0$ denotes the last zero of $\Lambda\Phi_n$ before $r_0$, and we let $r_{1,n}<r_0$ denote the last zero of $\varphi_{n,0}$ before $r_0$. We claim:

\begin{lemma}[Zeroes of $\Phi_{n,0}$]
\label{lemma:comparisionforrleqr0toodeH}
There holds
\be
\label{uniformroximity}
\sup_{0\leq r\leq r_0}\left(1+\frac{r}{\mu_n}\right)^{\frac{1}{2}}\left|\varphi_{n,0}(r) - \frac{p-1}{2}\Lambda Q\left(\frac{r}{\mu_n}\right)\right| \lesssim r_0^2
\ee
and 
\be
\label{eronvoinv}r_{1,n}=r_{0, n}+O(r_0^3), \ \ e^{-\frac{2\pi}{\omega}}r_0\leq r_{1, n}\leq r_0.
\ee
\end{lemma}

This is again a simple perturbative analysis which proof is detailed in Appendix \ref{appendixphinzero}.\\

We now claim the following classical result which relies on the standard analysis of explicit special functions:

\begin{lemma}[Special functions lemma]
\label{lemma:spectrumLinfty}
Let $\lambda\in \mathbb{R}$. The solutions to 
$$\mathcal{L}_\infty(\psi)=\l\psi, \ \ \psi\in H^1_{\rho}(1,+\infty)$$
behaves for $r\to +\infty$ as
$$\psi\sim r^{-\frac{2}{p-1}+\l}$$
and for $r\to 0_+$ as 
\be \label{asymptotique 0 fonction propre Linfty}
\psi = \frac{1}{r^\frac{1}{2}}\cos(\omega\log(r)-\Phi(\l))+O\left(r^{\frac{3}{2}}\right)
\ee
where
$$\Phi(\l) = \arg\left(\frac{2^{\frac{i\omega}{2}}\Gamma(i\omega)}{\Gamma\left(\frac{1}{p-1}-\frac{\l}{2}-\frac{1}{4}+\frac{i\omega}{2}\right)}\right).$$
\end{lemma}

\begin{proof}
We consider the solution $\psi$ to 
$$\mathcal{L}_\infty(\psi)=\l\psi.$$
The change of variable and unknown
$$\psi(r)=\frac{1}{(2z)^{\frac{\ga}{2}}}w(z),\,\,\,\, z=\frac{r^2}{2}$$
leads to 
\bee
\mathcal{L}_\infty(\psi) -\l\psi &=& -\frac{2}{(2z)^{\frac{\ga}{2}}}\left(zw''(z)+\left(-\ga+\frac{3}{2}-z\right)w'(z)-\left(\frac{1}{p-1}-\frac{\l}{2}-\frac{\ga}{2}\right)w(z)\right)
\eee
and thus  $\mathcal{L}_\infty(\psi)=\l\psi$ if and only if
$$
z\frac{d^2w}{dz^2}+(b-z)\frac{dw}{dz}-aw=0
$$
with
\bea\label{eq:actualvaluesofaandb:bis}
a=\frac{1}{p-1}-\frac{\l}{2}-\frac{\ga}{2},\,\,\,\, b=-\ga+\frac{3}{2}.
\eea
Hence $w$ is a linear combination of the special functions $M(a,b,z)$ and $U(a,b,z)$ whose asymptotic at infinity is given by \eqref{eq:kummersfunctionasymptotic1}:
\bee
M(a,b,z)\sim \frac{\Gamma(b)}{\Gamma(a)}z^{a-b}e^z,\,\,\,\, U(a,b,z)\sim z^{-a}\textrm{ as }z\to +\infty,
\eee
In particular, a non zero contribution of $M(a,b,z)$ to $w$ would yield for $\psi(r)$ the following asymptotic 
$$\psi(r)\sim r^{\frac{2}{p-1}-3-\l}e^{\frac{r^2}{2}}\textrm{ as }r\to +\infty.$$
which contradicts $\psi\in H^1_{\rho}(1,+\infty)$. Hence $$w(z) = U(a,b,z).$$
In view of the asymptotic of $U$ recalled in \eqref{eq:kummersfunctionasymptotic1}, we have
$$w(z)\sim z^{-a}\textrm{ as }z\to +\infty.$$
Since 
$$\psi(r)=\frac{1}{r^{\ga}}w\left(\frac{r^2}{2}\right),$$
this yields
$$\psi\sim r^{-\frac{2}{p-1}+\l}\textrm{ as }r\to +\infty.$$
Also, in view of the asymptotic of $U$ recalled in \eqref{eq:kummersfunctionasymptotic3}, we have 
\bee
w(z)=\frac{\Gamma(b-1)}{\Gamma(a)}z^{1-b}+\frac{\Gamma(1-b)}{\Gamma(a-b+1)}+O(z^{2-\Re(b)})\textrm{ as }z\to 0,
\eee
which in view of \eqref{eq:actualvaluesofaandb:bis} and the fact that $\ga=1/2+i\om$ yields
\bee
w(z)=\frac{\Gamma(-i\omega)}{\Gamma\left(\frac{1}{p-1}-\frac{\l}{2}-\frac{1}{4}-\frac{i\omega}{2}\right)}z^{i\omega}+\frac{\Gamma(i\omega)}{\Gamma\left(\frac{1}{p-1}-\frac{\l}{2}-\frac{1}{4}+\frac{i\omega}{2}\right)}+O(z)\textrm{ as }z\to 0.
\eee
Since 
$$\psi(r)=\frac{1}{r^{\ga}}w\left(\frac{r^2}{2}\right),$$
this yields
\bee
\psi(r) &=& \frac{2^{-\frac{i\omega}{2}}}{r^{\frac{1}{2}}}\left(\frac{2^{-\frac{i\omega}{2}}\Gamma(-i\omega)}{\Gamma\left(\frac{1}{p-1}-\frac{\l}{2}-\frac{1}{4}-\frac{i\omega}{2}\right)}r^{i\omega}+\frac{2^{\frac{i\omega}{2}}\Gamma(i\omega)}{\Gamma\left(\frac{1}{p-1}-\frac{\l}{2}-\frac{1}{4}+\frac{i\omega}{2}\right)}r^{-i\omega}\right)\\
&&+O\left(r^{\frac{3}{2}}\right)\textrm{ as }r\to 0,
\eee
and since $\psi$ is real valued, we infer\footnote{Note in particular that $\Gamma$ satisfies $\overline{\Gamma(z)}=\Gamma(\overline{z})$ for all $z\in\mathbb{C}$.}
$$\psi(r)= \frac{\cos(\omega\log(r)-\Phi(\l))}{r^{\frac{1}{2}}}+O\left(r^{\frac{3}{2}}\right)\textrm{ as } r\to 0,$$
where
$$\Phi(\l) = \arg\left(\frac{2^{\frac{i\omega}{2}}\Gamma(i\omega)}{\Gamma\left(\frac{1}{p-1}-\frac{\l}{2}-\frac{1}{4}+\frac{i\omega}{2}\right)}\right).$$
This concludes the proof of the lemma.
\end{proof}

%%%%%%%%%%%%%%%%%%%%%%%%%%%%%%%%%%%%%%%%%%%%%%%%%%%%%%%%%%%%%%%%%%%%

\subsection{Perturbative spectral analysis}

%%%%%%%%%%%%%%%%%%%%%%%%%%%%%%%%%%%%%%%%%%%%%%%%%%%%%%%%%%%%%%%%%%%%

We now prove elementary spectral analysis perturbation results based on the uniform bounds \eqref{behavioruinfity}, \eqref{behviourselfsimlocal} which allow us to precisely count the number of instabilities of $\mathcal L_{n,0}$.

\begin{lemma}[Control of the outside spectrum]
\label{lemma:perturbationspectrumAn}
Let $r_0>0$ and let $r_{n,2}$ such that $r_{n,2}>e^{-\frac{2\pi}{\omega}}r_0$. 
Let us define the operators 
\be
\label{nceinveoneoninoe}
\left|\begin{array}{ll} A_n[r_{n,2}](f)=\mathcal{L}_{n,0}(f)\textrm{ on }r>r_{n,2},\,\,\,\, f(r_{n,2})=0,\\
A_\infty[r_{n,2}](f)=\mathcal{L}_\infty(f)\textrm{ on }r>r_{n,2},\,\,\,\, f(r_{n,2})=0,\end{array}\right.
\ee
then
\be
\label{cneneonelmepobis}
\underset{\lambda \in \text{Spec}(A_n[r_{n,2}])}{\text{sup}} \ \underset{\mu \in \text{Spec}(A_{\infty}[r_{n,2}])}{\text{inf}} |\lambda-\mu|+\underset{\mu \in \text{Spec}(A_{\infty}[r_{n,2}])}{\text{sup}} \ \underset{\lambda \in \text{Spec}(A_n[r_{n,2}])}{\text{inf}} |\lambda-\mu| \rightarrow 0
\ee
as $n\rightarrow +\infty$.
\end{lemma}

\begin{proof}
In view of \eqref{weightedesimate}, the local compactness of the Sobolev embeddings 
$$H^1(|x|\leq R)\hookrightarrow L^2(|x|\leq R)\textrm{ for all }1\leq R<+\infty,$$
and the fact that $\Phi_n\in L^\infty$ and , the selfadjoint operators $A_n[r_{n,2}]+M_n$ for the measure $\rho dx$ are for $M_n\geq 1$ large enough invertible with compact resolvent, and $A_n[r_{n,2}]$ is diagonalizable. Since $\Phi_*\in L^\infty(r>r_0)$, we deduce similarly that $A_\infty[r_{n,2}]$ is diagonalizable. Let then $\l_n$ be an eigenvalue of $A_n[r_{n,2}]$ with normalized eigenvector $w_n$:
$$\mathcal{L}_n(w_n)=0\textrm{ on }r>r_{n,2},\,\,\,\, w_n(r_{n,2})=0,\,\,\,\, \|w_n\|_{L^2_{\rho}(r>r_{n,2})}=1.$$
Since $A_\infty[r_{n,2}]$ is diagonalizable in a Hilbert basis of $L^2_\rho$, we have
\bee
\|A_\infty[r_{n,2}](w_n)-\l_n w_n\|_{L^2_{\rho}(r>r_{n,2})} &\geq& \textrm{dist}(\l_n, \textrm{spec}(A_\infty[r_{n,2}]))\|w_n\|_{L^2_{\rho}(r>r_{n,2})}\\
&=& \textrm{dist}(\l_n, \textrm{spec}(A_\infty[r_{n,2}])).
\eee
On the other hand,
\bee
\|A_\infty[r_{n,2}](w_n)-\l_n w_n\|_{L^2_{\rho}(r>r_{n,2})} &=& \|(A_\infty[r_{n,2}]-A_n[r_{n,2}])(w_n)\|_{L^2_{\rho}(r>r_{n,2})}
\eee
from which:
\bee
&&\textrm{dist}(\l_n, \textrm{spec}(A_\infty[r_{n,2}])) \leq \|(A_\infty[r_{n,2}]-A_n[r_{n,2}])(w_n)\|_{L^2_{\rho}(r>r_{n,2})}\\
&\leq& \left(\sup_{r\geq r_{n,2}}\Big(p|\Phi_n(r)-\Phi_*(r)|^{p-1}\Big)\right)^{\frac{1}{2}}\|w_n\|_{L^2_{\rho}(r>r_{n,2})}\leq  \left(\sup_{r\geq r_{n,2}}\Big(p|\Phi_n(r)-\Phi_*(r)|^{p-1}\Big)\right)^{\frac{1}{2}}\\
&\to &  0\ \ \mbox{as}\ \ n\to +\infty
\eee
from \eqref{behavioruinfity}. \eqref{cneneonelmepobis} follows by exchanging the role $A_n[r_{n,2}]$ and $A_\infty[r_{n,2}]$.
\end{proof}

\begin{lemma}[Local continuity of the spectrum]
\label{lemma:perturbationspectrumAinfty}
Let $r_0>0$ and let $r_1$ and $r_2$ such that 
$$e^{-\frac{2\pi}{\omega}}r_0\leq r_1, r_2\leq r_0$$
and
$$r_1=r_2+O(r_0^3).$$
Then, for any eigenvalue $\l_1$ of $A_\infty[r_1]$ such that $\l_1\in [-3,1]$, we have
\be
\label{toberpovedvinvo}
\textrm{dist}(\l_1, \textrm{Spec}(A_\infty[r_2]))\lesssim r_0^{\frac{3}{2}}.
\ee
\end{lemma}

\begin{proof}
Recall from the proof of Lemma \ref{lemma:perturbationspectrumAn} that both $A_\infty[r_1]$ and $A_\infty[r_2]$ are diagonalizable. Furthermore, by  Sturm-Liouville, their eigenvalues are simple. Let $\l_1$ be an eigenvalue of  $A_\infty[r_1]$. We claim the existence of a nearby eigenvalue $\l_2$ of  $A_\infty[r_2]$ using a classical Lyapunov Schmidt procedure.\\
Let $\varphi_1$ the normalized eigenfunction of $A_\infty[r_1]$ associated to $\l_1$ so that
$$A_\infty[r_1](\varphi_1)=\l_1\varphi_1,\,\,\,\, \|\varphi_1\|_{\rho}=1.$$
The eigenvalue equation
$$A_\infty[r_2](\varphi_2)=\l_2\varphi_2$$
is equivalent to
\bea\label{eq:eigenvalueequationforg}
A_\infty[r_1](g)=\l_2 g+hg+(r_2-r_1)\pa_rg
\eea
where
$$g(r)=\varphi_2(r+r_2-r_1),\,\,\,\, h(r)=\frac{pc_{\infty}^{p-1}}{(r+r_2-r_1)^2}-\frac{pc_{\infty}^{p-1}}{r^2}.$$
We decompose
$$g=\varphi_1+r_0\tilde{g},\,\,\,\, \l_2=\l_1+cr_0$$
where the constant $c$ will be chosen later. Then, $g$ satisfies \eqref{eq:eigenvalueequationforg} if and only if $\tilde{g}$ satisfies
\bea\label{eq:eigenvalueequationforgtilde}
(A_\infty[r_1]-\l_1)(\tilde{g}) = c\varphi_1+c r_0\tilde{g}+\frac{h}{r_0}(\varphi_1+r_0\tilde{g})+\frac{r_2-r_1}{r_0}\pa_rg.
\eea
We choose $c$ such that 
$$c(\varphi_1, r_0, \tilde{g}):=-\frac{1}{1+r_0(\tilde{g}, \varphi_1)_{\rho}}\left(\frac{h}{r_0}(\varphi_1+r_0\tilde{g})+\frac{r_2-r_1}{r_0}\pa_r(\varphi_1+r_0\tilde{g}),\varphi_1\right)_{\rho}.$$
Then, the right-hand side of \eqref{eq:eigenvalueequationforgtilde} is orthogonal to $\varphi_1$ and hence to the kernel of $A_\infty[r_1]-\l_1$ since $\l_1$ is a simple eigenvalue. Thus, we infer
\bea\label{eq:eigenvalueequationforgtildebis}
\tilde{g} = \mathcal{F}(\tilde{g})
\eea
where
\bee
\mathcal{F}(\tilde{g}): = B_\infty[r_1,\l_1]^{-1}\Big(c(\varphi_1, r_0, \tilde{g})(\varphi_1+r_0\tilde{g})+\frac{h}{r_0}(\varphi_1+r_0\tilde{g})+\frac{r_2-r_1}{r_0}\pa_r(\varphi_1+r_0\tilde{g})\Big)
\eee
with the operator $B_\infty[r_1,\l_1]$ being the restriction of $A_\infty[r_1]-\l_1$ to the orthogonal complement of the kernel of $A_\infty[r_1]-\l_1$, i.e.
\bee
B_\infty[r_1,\l_1]=(A_\infty[r_1]-\l_1)_{|_{\varphi_1^{\perp}}}.
\eee
Since $\lambda_1$ is an eigenvalue of $A_{\infty}[r_1]$, from the explicit behavior \fref{asymptotique 0 fonction propre Linfty} of the eigenfunctions of $\mathcal L_{\infty}$ and the boundary condition \fref{nceinveoneoninoe} at $r_1$ one deduces that there exists $k\in \mathbb Z$ such that
$$
\omega \text{log}(r_1)-\Phi(\lambda_1)=k\pi +\frac{\pi}{2}+O(r_0^2).
$$
Let $\lambda_1'$ be the smallest eigenvalue of $A_{\infty}[r_1]$ greater than $\lambda_1$. It then satisfies:
$$
\omega \text{log}(r_1)-\Phi(\lambda_1')=k\pi +\frac{\pi}{2} \pm \pi+O(r_0^2)
$$
and so
$$
|\Phi (\lambda_1) - \Phi (\lambda_1')|= \pi +O(r_0^2)\geq \frac{\pi}{2}.
$$
As $\Phi$ is a continuous function we deduce that there exists $c>0$ independent of $r_0$ such that $\lambda_1'\geq \lambda_1+c$ and we infer
$$
\text{inf}\{|\lambda-\lambda_1|, \ \lambda \in \text{Spec}(A_{\infty}[r_1]), \ \lambda>\lambda_1 \}\geq c .
$$
Similarly
$$
\text{inf}\{|\lambda-\lambda_1|, \ \lambda \in \text{Spec}(A_{\infty}[r_1]), \ \lambda <\lambda_1 \}\geq c', \ \ c'>0 
$$
and we conclude that
\bee
\|B_\infty[r_1,\l_1]^{-1}\|_{\mathcal{L}(L^2_{\rho}, H^2_{\rho})}\lesssim 1
\eee
with a bound that does not depend on $r_0$. Also, note that
\bee
\frac{h}{r_0} &=& \frac{pc_{\infty}^{p-1}}{r_0r^2}\left(\frac{1}{(1+\frac{r_2-r_1}{r})^2}-1\right)=\frac{h_1(r)}{r^2}
\eee
where
\bee
h_1(r) &=& -pc_{\infty}^{p-1}\frac{\left(\frac{2(r_1-r_2)}{r_0r}+\frac{(r_1-r_2)^2}{r_0r^2}\right)}{(1+\frac{r_2-r_1}{r})^2}.
\eee
Since 
$$e^{-\frac{2\pi}{\omega}}r_0\leq r_1\leq r_0\textrm{ and }r_1=r_2+O(r_0^3),$$
we infer 
\bee
\|h_1\|_{L^\infty(r>r_1)} &\lesssim& \|h_1\|_{L^\infty(r>e^{-\frac{2\pi}{\omega}}r_0)}\lesssim  r_0.
\eee
Moreover,
\bee
\|r^{-2}\|_{L^2(r_1<r<1)} &\lesssim& \left(\int_{r_1}^1\frac{dr}{r^2}\right)^{\frac{1}{2}}\lesssim \frac{1}{r_1^{\frac{1}{2}}}\lesssim \frac{1}{r_0^{\frac{1}{2}}}.
\eee
Collecting the previous estimates, we infer
\bee
&&\|\mathcal{F}(\tilde{g})\|_{H^2_{\rho}(r>r_1)}\\
 &\lesssim&  \|B_\infty[r_1,\l_1]^{-1}\|_{\mathcal{L}(L^2_{\rho}, H^2_{\rho})}\left\|c(\varphi_1, r_0, \tilde{g})(\varphi_1+r_0\tilde{g})+\frac{h}{r_0}(\varphi_1+r_0\tilde{g})+\frac{r_2-r_1}{r_0}\pa_r(\varphi_1+r_0\tilde{g})\right\|_{L^2_{\rho}}\\
&\lesssim& |c(\varphi_1, r_0, \tilde{g})|(1+r_0\|\tilde{g}\|_{L^2_{\rho}})+r_0\|\tilde{g}\|_{H^1_\rho}\\
&&+\|h_1\|_{L^\infty(r>r_1)}(1+r_0\|\tilde{g}\|_{L^2_{\rho}}+\|\varphi_1+r_0\tilde{g}\|_{L^\infty_{(r_1<r<1)}}\|r^{-2}\|_{L^2(r_1<r<1)})\\
&\lesssim& \frac{r_0^{\frac{1}{2}}}{1-r_0\|\tilde{g}\|_{L^2_{\rho}}}(1+r_0\|\tilde{g}\|_{L^2_{\rho}})+r_0\|\tilde{g}\|_{H^1_\rho}
\eee
and
\bee
\|\mathcal{F}(\tilde{g}_1) - \mathcal{F}(\tilde{g}_2)\|_{H^2_{\rho}(r>r_1)} &\lesssim& \frac{r_0^{\frac{3}{2}}}{1-r_0\|\tilde{g}\|_{L^2_{\rho}}}(1+r_0\|\tilde{g}\|_{L^2_{\rho}})\|\tilde{g_1}-\tilde{g_2}\|_{L^2_{\rho}}+r_0\|\tilde{g}_1-\tilde{g}_2\|_{H^1_\rho}.
\eee
Thus, for $r_0>0$ small enough, the Banach fixed point theorem applies in the space $H^2_{\rho}(r>r_1)$ and yields a unique solution $\tilde{g}$ to \eqref{eq:eigenvalueequationforgtildebis} with
 $$\|\tilde{g}\|_{H^2_{\rho}(r>r_1)}\lesssim r_0^{\frac{1}{2}}.$$
 Hence, $\varphi_2$ with
 $$\varphi_2(r)=g(r+r_1-r_2),\,\,\,\, g=\varphi_1+r_0\tilde{g}$$
 satisfies
$$A_\infty[r_2](\varphi_2)=\l_2\varphi_2$$
where
\bee
\l_2 &=& \l_1+c(\varphi_1, r_0, \tilde{g})r_0\\
&=&  \l_1-\frac{r_0}{1+r_0(\tilde{g}, \varphi_1)_{\rho}}\left(\frac{h}{r_0}(\varphi_1+r_0\tilde{g})+\frac{r_2-r_1}{r_0}\pa_r(\varphi_1+r_0\tilde{g}),\varphi_1\right)_{\rho}.
\eee
Thus, $\l_2$ belongs to the spectrum of $A_\infty[r_2]$ and hence
\bee
\textrm{dist}(\l_1, \textrm{Spec}(A_\infty[r_2])) &\leq & |\l_2-\l_1|\\
&\leq& \left|\frac{r_0}{1+r_0(\tilde{g}, \varphi_1)_{\rho}}\left(\frac{h}{r_0}(\varphi_1+r_0\tilde{g})+\frac{r_2-r_1}{r_0}\pa_r(\varphi_1+r_0\tilde{g}),\varphi_1\right)_{\rho}\right|.
\eee
In view of the previous estimates, we infer
\bee
\textrm{dist}(\l_1, \textrm{Spec}(A_\infty[r_2])) &\lesssim& \frac{r_0^{\frac{3}{2}}}{1-r_0\|\tilde{g}\|_{L^2_{\rho}}}(1+r_0\|\tilde{g}\|_{L^2_{\rho}})\lesssim r_0^{\frac{3}{2}}.
\eee
and \eqref{toberpovedvinvo} is proved.
\end{proof}

%%%%%%%%%%%%%%%%%%%%%%%%%%%%%%%%%%%%%%%%%%

\subsection{Proof of Proposition \ref{prop:spectral}}\label{sec:proofopropspectral}

%%%%%%%%%%%%%%%%%%%%%%%%%%%%%%%%%%%%%%%%%%

Recall that  $\mathcal{L}_n$ is diagonalizable in a Hilbertian basis of $L^2_\rho$, and hence the spectral gap estimate \eqref{coerciviteapoids} follows from the explicit distribution of eigenvalues \eqref{eignevalurenvon} which we now prove. Observe that the symmetry group of dilations and translations generates the explicit eigenmodes
\be
\label{explicitspectrum}
\mathcal{L}_n\Lambda \Phi_n=-2\Lambda \Phi_n, \ \ \mathcal{L}_n\nabla\Phi_n=-\nabla\Phi_n.
\ee  
Using the decomposition into spherical harmonics \eqref{eq:decompositionLninLmn}, the further study of the quadratic form $(\mathcal{L}_n(u), u)_{\rho}$ reduces to the study of the quadratic form $(\mathcal{L}_{n,m}(u), u)_{\rho}$ for $m\geq 0$ for which classical Strum Liouville arguments are now at hand.\\
 
\noindent{\bf step 1} The case $m=1$. Let $\varphi_{n,1}$ be defined in Lemma \ref{lemma:asymptoticbehavioratinfinityforLmn}. In particular, $\varphi_{n,1}$ satisfies   
$$\mathcal{L}_{n,1}(\varphi_{n,1})=0,\,\,\,\, \varphi_{n,1}(0)=0,\,\,\,\, \varphi_{n,1}'(0)=1.$$
Then from standard Sturm Liouville oscillation argument for central potentials, \cite{RS}, the number of zeros of $\varphi_{n,1}$ in $r>0$ correspond to the number of strictly negative eigenvalues of $\mathcal{L}_{n,1}$. 

Since we have
$$\nabla\Phi_n(x) = \Phi_n'(r)\frac{x}{r} = \Phi_n'(r)(Y^{(1,-1)}, Y^{(1,1)}, Y^{(1,0)})$$ 
and hence
$$\mathcal{L}_n(\nabla\Phi_n)=-\nabla\Phi_n \ \ \mbox{implies}\ \ \mathcal{L}_{n,1}(\Phi_n')=-\Phi_n'.$$
Thus, $\mathcal{L}_{n,1}$ has at least one strictly negative eigenvalue, and hence $\varphi_{n,1}$ has at least one zero which we denote by $r_{n,1}>0$. On $[0,r_0]$, we have by \eqref{vnvnonvnornor}:
\bee
\sup_{0\leq r\leq r_0}\frac{\left|\mu_n^{-1}\varphi_{n,1}(r) - \nu_1\left(\frac{r}{\mu_n}\right)\right|}{\left|\nu_1\left(\frac{r}{\mu_n}\right)\right|}  &\lesssim& r_0^2
\eee
Since $\nu_1(r)>0$ for all $r>0$, we infer that $\varphi_{n,1}$ can not vanish on $[0,r_0]$. Hence, $r_{n,1}\geq r_0$.\\

\noindent{\em No other zero}. Assume by contradiction that there exists a second zero $r_{n,2}>r_{n,1}$. Let $f_{n,1}$ being given as
$$f_{n,1}:=\left\{\begin{array}{ll}
\varphi_{n,1} & \textrm{ on }r_{n,1}<r<r_{n,2},\\
0 & \textrm{ on }r<r_{n,1},\\
0 & \textrm{ on }r>r_{n,2}.
\end{array}\right.
$$
Then, we have $f_{n,1}\in H^1_{\rho}$ and 
\be
\label{vnenvonveoneoneo}
(\mathcal{L}_{n,1}(f_{n,1}),f_{n,1})_{\rho} = 0.
\ee
On the other hand, using \eqref{defphistar}:
\bea
\label{cnieoenvoelinfiit}
\nonumber (\mathcal{L}_{\infty,1}(u),u)_{\rho} &=& \|u'\|_{\rho}^2+\int_0^{+\infty}\frac{2-pc_{\infty}^{p-1}}{r^2}u^2r^2\rho dr\\
&=&\|u'\|_{\rho}^2+\frac{2(p+1)}{(p-1)^2}\left(\int_0^{+\infty}\frac{u^2}{r^2}r^2\rho dr\right)\gtrsim  \|\frac{u}{r}\|_{L^2_\rho}^2.
\eea
We now estimate from \eqref{behviourselfsimlocal} 
\be
\label{estpojoeepe}
\sup_{r\geq r_0} r^2|\Phi_n^{p-1}-(\Phi_*)^{p-1}|=o_{n\to +\infty}(1)
\ee
and hence for $u$ supported in $(r_0, +\infty)$:
\bea
\label{vneionveoenone}
\nonumber \left|(\mathcal{L}_{\infty,1}(u),u)_{\rho} - (\mathcal{L}_{n,1}(u),u)_{\rho}\right|& \lesssim & \int_{r_0}^{+\infty}\left|\Phi_n^{p-1}-\Phi_*^{p-1}\right|u^2r^2\rho(r)dr\\
&\leq & o_{n\to +\infty}(1)\left\|\frac{u}{r}\right\|_{L^2_\rho}^2.
\eea
Since $f_{n,1}$ is supported in $(r_{n,1}, r_{n,2})\subset (r_0, +\infty)$, \eqref{cnieoenvoelinfiit}, \eqref{vneionveoenone} applied to $f_{n,1}$ and \eqref{vnenvonveoneoneo} yield a contradiction for $n\geq N$ large enough. Thus, $r_{n,2}$ can not exist, and hence $\varphi_{n,1}$ vanishes only once.\\

\noindent{\em $\varphi_{n,1}$ is not an eigenstate}. Since $\varphi_{n,1}$ vanishes only once, $\mathcal{L}_{n,1}$ has exactly one strictly negative eigenvalue. It remains to check the $\varphi_{n,1}\notin L^2_{\rho}$, i.e. $\varphi_{n,1}$ is not an eigenvector associated to the eigenvalue 0. To this end, note that $\varphi_{n,1}$ is strictly positive on $(0,r_{n,1})$ from \eqref{cneoneoenove} and strictly negative on $(r_{n,1}, +\infty)$. In particular, we have
$$\varphi_{n,1}'(r_{n,1})<0.$$
Since $\mathcal{L}_{n,1}(\varphi_{n,1})=0$, we have
$$(r^2\rho \varphi_{n,1}')' = r^2\rho \left[ \frac{2}{p-1}+\frac{(2-p r^2\Phi_n^{p-1})}{r^2}\right]\varphi_{n,1}$$
and from \eqref{vneionveoenone} for $r\geq r_{n,1}\geq r_0$:
\be \label{borne potentiel harmonique}
2-r^2p\Phi_n^{p-1} = 2 - pc_{\infty}^{p-1} + pc_{\infty}^{p-1} -r^2p\Phi_n^{p-1}\geq \frac{2(p+1)}{(p-1)^2} +o(1)> 0.
\ee
Since $\varphi_{n,1}$ is strictly negative on $(r_{n,1}, +\infty)$, we deduce
$$ r^2\rho\varphi_{n,1}' (r)\leq r_{n,1}^2\rho(r_{n,1})\varphi_{n,1}'(r_{n,1})=c_1<0 \textrm{ on }(r_{n,1},+\infty)$$ which implies $$\int_{r_{n,1}}^{+\infty}|\varphi_{n,1}' (r)|^2\rho r^2dr\gtrsim \int_{r_{n,1}}^{+\infty} \frac{dr}{r^2\rho}=+\infty$$ and hence $\varphi_{n,1}\notin H^1_{\rho}$ and is therefore not an eigenvector.\\

\noindent{\em Conclusion}. We conclude that $-1$ is the only negative eigenvalue of $\mathcal{L}_{n,1}$, and is associated to the single eigenvector $\Phi_n'$. Hence, there exists a constant $c_n>0$ such that for all $u\in H^1_{\rho}$:
\bea\label{eq:coerciveestimatem=1}
(\mathcal{L}_{n,1}(u),u)_{\rho}\geq c_n\|u\|^2_{L^2_\rho} - \frac{1}{c_n}(u, \Phi_n')_{\rho}^2.\\
\nonumber
\eea

\noindent{\bf step 2} The case $m\geq 2$. Let $\varphi_{n,m}$ be defined in Lemma \ref{lemma:asymptoticbehavioratinfinityforLmn}. In particular, $\varphi_{n,m}$ satisfies   
$$\mathcal{L}_{n,m}(\varphi_{n,m})=0\textrm{ and }\varphi_{n,m}=r^m(1+O(r^2))\textrm{ as }r\to 0_+.$$
Then, the number of zeros of $\varphi_{n,m}$ in $r>0$ corresponds to the number of strictly negative eigenvalues of $\mathcal{L}_{n,m}$. On $[0,r_0]$, we have by Lemma \ref{lemma:asymptoticbehavioratinfinityforLmn}.
\bee
\sup_{0\leq r\leq r_0}\frac{\left|\mu_n^{-m}\varphi_{n,m}(r) - \nu_m\left(\frac{r}{\mu_n}\right)\right|}{\left|\nu_m\left(\frac{r}{\mu_n}\right)\right|}  &\lesssim& r_0^2
\eee
and $\nu_m(r)>0$ for all $r>0$, and hence $\varphi_{n,m}$ cannot vanish on $[0,r_0]$:
$$\varphi_{n,m}(r)>0\textrm{ on }[0,r_0].$$

Next, we investigate the sign of $\varphi_{n,m}'(r_0)$. Recall \eqref{beahviorunum}:
$$\nu_m(r)\sim \frac{c_{m,-}}{r^{\frac{1-\sqrt{\Delta_m}}{2}}}\textrm{ as }r\to +\infty \ \ c_{m.-}>0$$
and hence
$$\nu_m'(r)\sim \frac{c_{m,-}(\sqrt{\Delta_m}-1)}{r^{\frac{3-\sqrt{\Delta_m}}{2}}}\textrm{ as }r\to +\infty.$$
We infer for $n\geq N$ large enough
$$ \varphi_{n,m}(r_0) = \frac{c_{m,-}(1+O(r_0^2))\mu_n^m}{\left(\frac{r_0}{\mu_n}\right)^{\frac{1-\sqrt{\Delta_m}}{2}}}$$
and
$$ \varphi_{n,m}'(r_0) = \frac{c_{m,-}(\sqrt{\Delta_m}-1)(1+O(r_0^2))\mu_n^{m-1}}{\left(\frac{r_0}{\mu_n}\right)^{\frac{3-\sqrt{\Delta_m}}{2}}}.$$
Thus, taking also into account that $\varphi_{n,m}(r)>0$ on $[0,r_0]$, we infer from the identity for $\varphi_{n,m}(r_0)$ that 
$$c_{m,-}>0.$$
Since $\sqrt{\Delta_m}\geq \sqrt{\Delta_1}=\frac{p+3}{p-1}>1$, we conclude: 
\be
\label{cneinvenevone}
\phi_{n,m}(r_0)>0, \ \ \phi_{n,m}'(r_0)>0.
\ee
Since $\mathcal{L}_{n,m}(\varphi_{n,m})=0$, we have
\be
\label{vneoneneo}
(r^2\rho \varphi_{n,m}')' = r^2\rho \left[\frac{2}{p-1}+ \frac{(m(m+1)-pr^2\Phi_n^{p-1})}{r^2}\right]\varphi_{n,m}
\ee
which together with \eqref{cneinvenevone}, \fref{borne potentiel harmonique} and the fact that $m\geq 2$, and an elementary continuity argument ensures $$\phi_{m,n}'(r)>0, \ \ \phi_{n,m}(r)\geq \phi_{n,m}(r_0)>0\ \ \mbox{for}\ \ r\ge r_0.$$ Hence $\phi_{n,m}$ does not vanish on $(0,+\infty)$ and using \eqref{vneoneneo}: $$r^2\phi_{n,m}'\rho(r)\geq r_0^2\phi_{n,m}'\rho(r_0)=c_0>0$$ which implies $$\int_{r_0}^{+\infty}(\phi'_{n,m})^2\rho r^2dr\gtrsim \int_{r_0}^{+\infty} \frac{dr}{r^2\rho}=+\infty$$ and hence $\phi_{n,m}$ is not eigenvector. We finally conclude that for $m=2$ and all $n\geq N$ large enough, $\mathcal{L}_{n,2}$ has a spectral gap and there exists a constant $c_n>0$ such that we have for all $u\in H^1_{\rho}$
$$(\mathcal{L}_{n,2}(u),u)_{\rho}\geq c_n\|u\|^2_{L^2_\rho}.$$
Since we have for all $m\geq 2$
$$(\mathcal{L}_{n,m}(u),u)_{\rho}\geq (\mathcal{L}_{n,2}(u),u)_{\rho},$$
we infer for all $m\geq 2$ and for all $u\in H^1_{\rho}$
\bea\label{eq:coerciveestimatemgeq2}
(\mathcal{L}_{n,m}(u),u)_{\rho}\geq c_n\|u\|^2_{L^2_\rho}.\\
\nonumber
\eea

\noindent{\bf step 3.} The case $m=0$. We now focus onto $\mathcal{L}_{n,0}$ which is the most delicate case, and we claim that $\mathcal L_{n,0}$ has exactly $n+1$ strictly negative eigenvalues, and that 0 is not in the spectrum. The key is to combine the uniform bounds \eqref{behavioruinfity} with the explicit knowledge of the limiting outer spectrum, Lemma \ref{lemma:spectrumLinfty}, as nicely suggested at the formal level in \cite{bizon}.\\
Let $\varphi_{n,0}$ be the solution to \eqref{defphinoneveev} so that the number of strictly negative eigenvalues of $\mathcal{L}_{n,0}$ coincides with the numbers of zeroes of $\varphi_{n,0}$. We count the number of zeros of $\varphi_{n,0}$ by comparing them with the  number of zeros of $\Lambda\Phi_n$.\\

\noindent{\em Lower bound}. First, since $\Lambda\Phi_n$ is an eigenvector of $\mathcal{L}_{n,0}$ corresponding to the eigenvalue $-2$ and since $\Lambda \Phi_n$ vanishes $n$ times from Proposition \ref{prop:constructionPhin}, we infer from Sturm Liouville
$$\#\textrm{Spec}(\mathcal{L}_{n,0}+2)\cap(-\infty,0] = n+1.$$
In particular, since the number of strictly negative eigenvalues of $\mathcal{L}_{n,0}$ coincides with the number of zeroes of $\varphi_{n,0}$, we infer
$$\#\{r\geq 0\textrm{ such that }\varphi_{n,0}(r)=0\}\geq n+1.$$
\noindent{\em Upper bound}. Recall \eqref{uniformroximity}:
\bee
\sup_{0\leq r\leq r_0}\left(1+\frac{r}{\mu_n}\right)^{\frac{1}{2}}\left|\varphi_{n,0}(r) - \frac{p-1}{2}\Lambda Q\left(\frac{r}{\mu_n}\right)\right| &\lesssim& r_0^2.
\eee
Also, we have $\Lambda Q(0)\neq 0$ and from \eqref{tobeprovedlambdaq}:
$$\left(\frac{r_0}{\mu_n}\right)^{\frac{1}{2}}\left|\Lambda Q\left(\frac{r_0}{\mu_n}\right)\right|\geq c>0$$
for some constant $c>0$ independent of $n$. Hence $\varphi_{n,0}$ and  $\Lambda Q$  vanish the same number of times on $[0,r_0]$. Since on the other hand $\Lambda Q$ and $\Lambda \Phi_n$ vanish the same number of times on $[0,r_0]$ from \eqref{innerzeroes},  $\varphi_{n,0}$ and  $\Lambda\Phi_n$  vanish the same number of times of $[0,r_0]$.\\
Let now $r_{n,0}$ to be the last zero of $\Lambda \Phi_n$ before $r_0$. In view of Corollary \ref{cor:consequenceforPhinofproponexistence}, we have
$$e^{-\frac{2\pi}{\omega}}r_0\leq r_{n,0}\leq r_0.$$
Let us now consider the operators \eqref{nceinveoneoninoe}:
\bee
&&A_n[r_{n,0}](f)=\mathcal{L}_{n,0}(f)\textrm{ on }r>r_{0,n},\,\,\,\, f(r_{n,0})=0,\\
&& A_\infty[r_{n,0}](f)=\mathcal{L}_\infty(f)\textrm{ on }r>r_{n,0},\,\,\,\, f(r_{n,0})=0,
\eee
then $$\mathcal{L}_{n,0}(\Lambda\Phi_n)=-2\Lambda\Phi_n\textrm{ and }\Lambda\Phi_n(r_{n,0})=0,$$ implies
$$A_n[r_{n,0}](\Lambda\Phi_n)=-2\Lambda\Phi_n.$$
In particular, $-2$ belongs to the spectrum of $A_n[r_{n,0}]$. In view of Lemma \ref{lemma:perturbationspectrumAn}, we deduce for $n\geq N$ large enough that the exists an eigenvalue $\l_0$ of $A_\infty[r_{n,0}]$ such that $\l_0 = -2+o(1)$. On the other hand, in view of Lemma \ref{lemma:spectrumLinfty}, the solutions to 
$$\mathcal{L}_\infty(f)=\l f$$
with $f\in H^1_{\rho}$ are completely explicit and behave for $r\to 0$ as 
$$f\sim \frac{1}{r^\frac{1}{2}}\cos(\omega\log(r)-\Phi(\l))$$
with 
$$\Phi(\l) = \arg\left(\frac{2^{\frac{i\omega}{2}}\Gamma(i\omega)}{\Gamma\left(\frac{1}{p-1}-\frac{\l}{2}-\frac{1}{4}+\frac{i\omega}{2}\right)}\right).$$
In order for $f$ to be an eigenfunction of $A_\infty[r_{n,0}]$, we need $f(r_{n,0})=0$ and hence there should exists $k\in\mathbb{Z}$ such that
$$\omega\log(r_{n,0})-\Phi(\l)\sim \frac{\pi}{2}+k\pi.$$
Recall that $\l_0=-2+o(1)$ is an eigenvalue of $A_\infty[r_{n,0}]$, and let $\l_1>\l_0$ be the next eigenvalue of $A_\infty[r_{n,0}]$. Then, there exists $k_0\in\mathbb{R}$ such that 
$$\omega\log(r_{n,0})-\Phi(\l_0)\sim \frac{\pi}{2}+k_0\pi, \,\,\,\,\omega\log(r_{n,0})-\Phi(\l_1)\sim \frac{\pi}{2}+(k_0-1)\pi$$
and hence
\bea\label{eq:equationquantizationlambda1}
\Phi(\l_1) = \Phi(-2) + \pi +o(1).
\eea
Now, by numerical check, we have\footnote{Notice that $\Phi(\l)$ has a well defined limit as $p\to +\infty$ given by
$$\Phi_\infty(\l) = \arg\left(\frac{2^{\frac{i}{4}}\Gamma(\frac{i}{2})}{\Gamma\left(-\frac{\l}{2}-\frac{1}{4}+\frac{i}{4}\right)}\right).$$
Our numerics are carried out using Matlab and indicate that $\Phi_p(\l)$ is increasing on $[-2,0.5]$ for all $p\geq 5$ so that the maximum on $[-2,0.5]$ is achieved at $\l=0.5$. Also, this maximum appears to be a growing function of $p$ so that the maximum in $p$ is given by $\Phi_\infty(0.5)-\Phi_\infty(-2)-\pi\sim -0.5945$. See \cite{bizon} for a similar numerical computation.}
$$\sup_{5\leq p<+\infty}\sup_{-2\leq \l\leq 0.5}(\Phi(\l)-\Phi(-2)-\pi)\sim -0.5945 <0,$$
and hence, the solution $\l_1$ to \eqref{eq:equationquantizationlambda1} satisfies
$$\inf_{5\leq p<+\infty}\l_1\geq 0.5>0.$$
We infer that $A_\infty[r_{n,0}]$ has no eigenvalue between $\l_0=-2+o(1)$ and $\l_1\geq 0.5$. Hence, using again Lemma \ref{lemma:perturbationspectrumAn}, $A_n[r_{n,0}]$ has no eigenvalue between $-2$ and $\l_1+o(1)\geq 0.25$. Thus, we have
$$\#\textrm{Spec}(A_n[r_{n,0}])\cap(-\infty,0] = \#\textrm{Spec}(A_n[r_{n,0}]+2)\cap(-\infty,0].$$
On the other hand, we have
$$\#\textrm{Spec}(A_n[r_{n,0}]+2)\cap(-\infty,0] = \#\{r>r_{n,0}\textrm{ such that }\Lambda\Phi_n(r)=0\}+1$$
since $\Lambda\Phi_n$ is in the kernel of $A_n[r_{n,0}]+2$, 
and hence
$$\#\textrm{Spec}(A_n)\cap(-\infty,0] =  \#\{r>r_{n,0}\textrm{ such that }\Lambda\Phi_n(r)=0\}+1.$$
Also, since $\varphi_{n,0}$ can not be an eigenvector of $A_n$\footnote{Indeed, $\varphi_{n,0}$ would be an eigenvector for the eigenvalue 0, but 0 is not in the spectrum of $A_n$ as seen above.}, we have
$$\#\textrm{Spec}(A_n[r_{n,0}])\cap(-\infty,0] = \#\{r>r_{n,0}\textrm{ such that }\varphi_{n,0}(r)=0\}.$$
We infer
$$\#\{r>r_{n,0}\textrm{ such that }\varphi_{n,0}(r)=0\}=\#\{r>r_{n,0}\textrm{ such that }\Lambda\Phi_n(r)=0\}+1.$$
But since $r_{n,0}$ has been chosen to be the last zero of $\Lambda \Phi_n$ before $r_0$, we have
$$\#\{r>r_{n,0}\textrm{ such that }\Lambda\Phi_n(r)=0\}=\#\{r>r_0\textrm{ such that }\Lambda\Phi_n(r)=0\}$$
and hence
$$\#\{r>r_{n,0}\textrm{ such that }\varphi_{n,0}(r)=0\}=\#\{r>r_0\textrm{ such that }\Lambda\Phi_n(r)=0\}+1.$$

Next, together with the fact that  $\varphi_{n,0}$ and  $\Lambda\Phi_n$  vanish the same number of times of $[0,r_0]$, we infer
\bee
&&\#\{r>0\textrm{ such that }\varphi_{n,0}(r)=0\} \\
&\leq& \#\{0\leq r\leq r_0\textrm{ such that }\varphi_{n,0}(r)=0\}+\#\{r>r_{n,0}\textrm{ such that }\varphi_{n,0}(r)=0\}\\
&=& \#\{0\leq r\leq r_0\textrm{ such that }\Lambda\Phi_n(r)=0\}+\#\{r>r_0\textrm{ such that }\Lambda\Phi_n(r)=0\}+1\\
&=& \#\{r>0\textrm{ such that }\Lambda\Phi_n(r)=0\}+1\\
&=& n+1
\eee
and since 
$$\#\{r\geq 0\textrm{ such that }\varphi_{n,0}(r)=0\}\geq n+1.$$
\noindent{\em $\phi_{n,0}$ is not an eigenstate}. We conclude that 
$$\#\{r\geq 0\textrm{ such that }\varphi_{n,0}(r)=0\}= n+1.$$
Assume now by contradiction that $\varphi_{n,0}$ is in the kernel of $\mathcal{L}_{n,0}$. Recall that $r_{0,n}<r_0$ is the last 0 of $\Lambda\Phi_n$ and let $r_{1,n}<r_0$ be the last 0 of $\varphi_{n,0}$. In particular, we have from Lemma \ref{lemma:comparisionforrleqr0toodeH}:
$$e^{-\frac{2\pi}{\omega}}r_0\leq r_{0,n}, r_{1,n}\leq r_0\textrm{ and }r_{1,n}=r_{0,n}+O(r_0^3).$$
Also, since $\varphi_{n,0}$ is in the kernel of $\mathcal{L}_{n,0}$ and $\varphi_{n,0}(r_{1,n})=0$, we infer that 0 is in the spectrum of $A_n[r_{1,n}]$, and hence applying Lemma \ref{lemma:perturbationspectrumAn} twice as well as Lemma \ref{lemma:perturbationspectrumAinfty}, we obtain that 
\bee
\textrm{dist}(\textrm{Spec}(A_n[r_{0,n}]), 0) \lesssim r_0^{\frac{3}{2}}+o(1)
\eee
as $n\to +\infty$. In particular, we have for $r_0>0$ small enough and $n\geq N$ large enough
\bee
\textrm{dist}(\textrm{Spec}(A_n[r_{0,n}]), 0) \leq 0.2.
\eee
On the other hand, we have proved above that $A_n[r_{n,0}]$ has no eigenvalue between $-2$ and $\l_1+o(1)\geq 0.25$ so that
\bee
\textrm{dist}(\textrm{Spec}(A_n[r_{0,n}]), 0) \geq 0.25
\eee
which is a contradiction. Hence $\varphi_{0,n}$ is not in the kernel of $\mathcal{L}_{n,0}$.\\

\noindent{\em Conclusion}. We conclude that $\mathcal{L}_{n,0}$ has exactly $n+1$ strictly negative eigenvalues. On the other hand, since $\Lambda\Phi_n$ is an eigenvector of $\mathcal{L}_{n,0}$ corresponding to the eigenvalue $-2$ and since $\Lambda \Phi_n$ vanishes $n$ times, we infer
$$\#\textrm{Spec}(\mathcal{L}_{n,0}+2)\cap(-\infty,0] = n+1,$$
and hence $\mathcal{L}_{n,0}$ has exactly $n+1$ negative eigenvalues and the largest negative eigenvalue is $-2$. We denote these eigenvalues by
$$
-\mu_{n+1,n}<\dots<-\mu_{2,n}<-\mu_{1,n}=-2.
$$ 
By Sturm Liouville, these eigenvalues are simple and associated to eigenvectors 
$$\psi_{j,n}, \ \ \|\psi_{j,n}\|_{L^2_\rho}=1, \ \ \psi_{1,n}=\frac{\Lambda \Phi_n}{\|\Lambda \Phi_n\|_{\rho}}.$$ 
Also, there holds for some constant $c_n>0$ and for all $u\in H^1_{\rho}$
\bea\label{eq:coerciveestimatem=0} 
(\mathcal{L}_{n,0}(u), u)_\rho\geq c_n\|u\|_{L^2_\rho}^2-\frac{1}{c_n}\left[\sum_{j=1}^{n+1}(u,\psi_{j,n})^2_{\rho}\right].
\eea
The behavior as $r\to +\infty$ of the eigenstates \eqref{estgrowth} follows from the asymptotic in Lemma \ref{lemma:spectrumLinfty} and a standard ODE argument using the variation of constants formula, this is left to the reader.\\

\noindent{\bf step 4} Conclusion.  We decompose $u\in H^1_{\rho}$ as
$$u=\sum_{m=0}^{+\infty}\sum_{k=-m}^m u_{m,k}Y^{(m,k)}$$
where $u_{m,k}$ are radial functions satisfying
$$\|u\|_{\rho}^2=\sum_{m=0}^{+\infty}\sum_{k=-m}^m \|u_{m,k}\|^2_{\rho}.$$
We have
\bee
(\mathcal{L}_n(u),u)_{\rho} &=&  \sum_{m=0}^{+\infty}\sum_{k=-m}^m (\mathcal{L}_{n,m}(u_{m,k}),u_{m,k})_{\rho}.
\eee
Together with \eqref{eq:coerciveestimatem=1}, \eqref{eq:coerciveestimatemgeq2} and \eqref{eq:coerciveestimatem=0}, we infer for all $u\in H^1_{\rho}$
\bee
(\mathcal{L}_n(u),u)_{\rho} &=&   (\mathcal{L}_{n,0}(u_{0,0}),u_{0,0})_{\rho} +\sum_{k=-1}^1 (\mathcal{L}_{n,1}(u_{1,k}),u_{1,k})_{\rho} + \sum_{m=2}^{+\infty}\sum_{k=-m}^m (\mathcal{L}_{n,m}(u_{m,k}),u_{m,k})_{\rho}\\
&\geq& c_n\|u\|_{\rho}^2 -\frac{1}{c_n}\left[\sum_{j=1}^{n+1}(u_{0,0},\psi_{j,n})^2_{\rho}+\sum_{k=1}^3(u_{1,k},\Phi_n')^2_{\rho}\right].
\eee
Since $\psi_{j,n}$ are all radial, we have
$$(u_{0,0},\psi_{j,n})_{\rho} = (u,\psi_{j,n})_{\rho}.$$
Also, since
$$\nabla\Phi_n(x) = \Phi_n'(r)\frac{x}{r} = \Phi_n'(r)(Y^{(1,-1)}, Y^{(1,1)}, Y^{(1,0)}),$$ 
we infer
\bee
\sum_{k=1}^3(u_{1,k},\Phi_n')^2_{\rho} &=& \sum_{k=1}^3(u,\pr_k\Phi_n)^2_{\rho}.
\eee
Finally, there holds for some constant $c_n>0$ and for all $u\in H^1_{\rho}$
\bee
(\mathcal{L}_n u, u)_\rho\geq c_n\|u\|_{H^1_\rho}^2-\frac{1}{c_n}\left[\sum_{j=0}^n(u,\psi_{j,n})^2_{\rho}+\sum_{k=1}^3(u,\pr_k\Phi_n)^2_{\rho}\right].
\eee
This concludes the proof of Proposition \ref{prop:spectral}.

%%%%%%%%%%%%%%%%%%%%%%%%%%%%%%%%%%%%%
%%%%%%%%%%%%%%%%%%%%%%%%%%%%%%%%%%%%%

\section{Dynamical control of the flow}
\label{sectiondynamical}
%%%%%%%%%%%%%%%%%%%%%%%%%%%%%%%%%%%%%
%%%%%%%%%%%%%%%%%%%%%%%%%%%%%%%%%%%%%

We now turn to the question of the stability of the self similar solution, and more precisely the construction of a manifold of {\it finite energy} initial data such that the corresponding solution to \eqref{eq:heat} blows up in finite time with $\Phi_n$ profile in the self similar regime described by Theorem \ref{thmmain}. $n$ is now fixed.

%%%%%%%%%%%%%%%%%%%%%%%%%%%%%%%%%%%%%

\subsection{Setting of the bootstrap}

%%%%%%%%%%%%%%%%%%%%%%%%%%%%%%%%%%%%%

We set up in this section the bootstrap analysis of the flow for a suitable set of finite energy initial data. The solution will be decomposed in a suitable way with standard technique, see \cite{mamerle, meraphannals}.

\noindent { \em Geometrical decomposition}. We start by showing the existence of the suitable decomposition. Recall the spectral Proposition \ref{prop:spectral}. To ease notations we now omit the $n$ subscript and write $\psi_j$, $\mu_j$ and $\lambda_j$ instead.

Define the $L^{\infty}$ tube around the renormalized versions of $\Phi_n$:
$$
X_{\delta}=\left\{ u=\frac{1}{\lambda^{\frac{2}{p-1}}}(\Phi_n+v)\left(\frac{x-y}{\lambda} \right), \ y\in \mathbb R^d, \ \lambda>0, \ \| v \|_{L^{\infty}}< \delta \right\}
$$

\begin{lemma}[Geometrical decomposition] \label{lemma:decompositionofinitialdata}
There exists $\delta>0$ and $C>0$ such that any $u\in X_{\delta}$ has a unique decomposition
$$
u=\frac{1}{\lambda^{\frac{2}{p-1}}}(\Phi_n+\sum_{j=2}^{n+1}a_j\psi_j+\varepsilon)\left(\frac{x-\bar x}{\lambda} \right),
$$
where $\varepsilon$ satisfies the orthogonality conditions
$$
(\e ,\psi_j)_\rho=(\e ,\pa_k\Phi_n)_\rho=0, \ \ 1\leq j\leq n+1, \ \ 1\leq k\leq 3,
$$
the parameters $\lambda$, $\bar x$ and $a_j$ being Fr\'echet differentiable on $X_{\delta}$, and with
\be \label{eq:bound continuite decomposition}
\| \e \|_{L^{\infty}}+\sum |a_j| \leq C.
\ee
\end{lemma}

\begin{proof}
It is a classical consequence of the implicit function theorem. 

\noindent\textbf{step 1} Decomposition near $\lambda=1$, $\bar x=0$. We introduce the smooth maps
$$F(v, \mu, x, b_1,\ldots, b_n)=\mu^{\frac{2}{p-1}}(\Phi_n+v)(\mu y+x) -\Phi_n - \sum_{j=2}^{n+1} b_j\psi_j$$
and
$$G = ((F, \Lambda\Phi_n), (F,\pr_1\Phi_n), (F,\pr_2\Phi_n), (F,\pr_3\Phi_n), (F, \psi_2),\ldots, (F,\psi_{n+1})).$$ 
We immediately check that $G(\Phi_n, 1, 0,\ldots, 0)=0$ and that
$$\frac{\pr G}{\pr ( \mu, x, b_2,\ldots, b_{n+1})}_{|_{(\Phi_n, 1, 0, \ldots, 0)}}$$
is invertible. In view of the implicit function theorem, for $\kappa>0$ small enough, for any 
$$
\| v \|_{L^{\infty}}\leq \kappa
$$
there exists $(\mu , z, a_2, \ldots, a_{n+1})$ and
$$\ep =F(v, \mu, z, a_2,\ldots, a_{n+1})$$
such that
$$u =\Phi_n+v= \frac{1}{\mu ^{\frac{2}{p-1}}}\left(\Phi_n+\sum_{j=2}^{n+1}a_j\psi_j+\ep \right)\left(\frac{x-z}{\mu}\right),$$
$$
(\e ,\psi_j)=(\e ,\pa_k\Phi_n)=0, \ \ 1\leq j\leq n, \ \ 1\leq k\leq 3,
$$
and there exist two universal constants $K,\tilde K>0$ such that
$$
\| \e \|_{L^{\infty}}+\sum_{j=2}^{n+1} |a_j|+|\mu-1|+|z|\leq K\| v \|_{L^{\infty}}
$$
and such that the decomposition is unique under the bound
\be \label{eq:bound unicite decomposition}
\| \e \|_{L^{\infty}}+\sum_{j=2}^{n+1} |a_j|+|\mu-1|+|z|\leq \tilde K.
\ee

\noindent \textbf{step 2} Decomposition near any $\lambda,\bar x$. For any $\delta>0$, we take $C=C(\delta) := K \delta$. Let $u\in X_{\delta}$ then for some $\lambda'>0$ and $y$ one has 
$$
u(x)=\frac{1}{\lambda^{'\frac{2}{p-1}}}(\Phi_n+v)\left(\frac{x-y}{\lambda'} \right), \ \| v \|_{L^{\infty}}< \delta .
$$
The first step then provides the decomposition claimed in the lemma for $\delta$ small enough via the formulas $\lambda =\lambda' \mu (v)$, $\bar x=y-\lambda' z(v)$, $a_j=a_j(v)$ and $\e=\e(v)$. We will show in the next step that the decomposition is unique, implying that the parameters are Fr\'echet differentiable on $X_{\delta}$ for those of step 1 are.

\noindent \textbf{step 3} Uniqueness of the decomposition. First, from a continuity argument, for any $\epsilon>0$, there exists $\delta>0$ such that if
$$
(\Phi_n +v)(x)=\frac{1}{\mu^{\frac{2}{p-1}}} (\Phi_n+v')\left( \frac{x-y}{\mu}\right), \ \| v \|_{L^{\infty}} + \| v' \|_{L^{\infty}} \leq \delta
$$
then
$$
|\mu-1|+|y|\leq \epsilon . 
$$
Now recall that $C= K \delta$ and assume that we are given a second decomposition for $u\in X_{\delta}$. In view of step 2, performing a change of variable, this amount to say that $\Phi_n+v$ admits another decomposition:
$$
(\Phi_n+v)(x)=\frac{1}{\bar \mu^{\frac{2}{p-1}}} (\Phi_n+\sum_{j=2}^{n+1}\bar a_j \psi_j +\bar \varepsilon)\left( \frac{x-\bar z}{\bar \mu}\right)
$$
and the bound \fref{eq:bound continuite decomposition} gives
$$
\sum_{j=2}^{n+1} |\bar a_j| +\| \bar \varepsilon \|_{L^{\infty}}\leq K\delta .
$$
Using the above continuity estimate, one obtains that for $\delta $ small enough
$$
|\bar z|+|\bar \mu-1| \ll \tilde K.
$$
Therefore, for $\delta$ small enough the second decomposition associated with $\bar \mu$, $\bar z$, $\bar a_j$ and $\bar \e$ satisfies \fref{eq:bound unicite decomposition}, and is therefore the one given by step 2 by uniqueness.
\end{proof}

\noindent { \em Description of the initial datum}. We will now focus on solutions of \fref{eq:heat} that are a suitable perturbation of $\Phi_n$ at initial time:
\bea\label{geomdeopcopmtiinit}
u_0=\frac{1}{\l_0^{\frac{2}{p-1}}}(\Phi_n+v_0)\left(\frac{x}{\l_0}\right)
\eea
with
\bea\label{orhotpsininit}
v_0=\sum_{j=2}^{n+1} a_j\psi_j+\ep_0, \ \ (\e_0,\psi_j)_\rho=(\e_0,\pa_k\Phi_n)_\rho=0, \ \ 1\leq j\leq n+1, \ \ 1\leq k\leq 3.
\eea
For $s_0\gg 1$ and $\mu,K_0>0$ three constants to be defined later on, the parameters $\lambda_0$, $a_j$ and the profile $\e_0$ satisfy the bounds
\begin{itemize}
\item rescaled solution: 
\be
\label{scalingsmall}
\l_0 = e^{-s_0};
\ee
\item initial control of the unstable modes:
\be
\label{nvnono}
\sum_{j=2}^{n+1} |a_j|^2\leq e^{-2\mu s_0};
\ee
\item smallness of suitable initial norms:
\be
\label{inititalboot}
 \|\e_0\|_{H^2_\rho}+\|\Delta v_0\|_{L^2}+\|w_0\|_{\dot{H}^{s_c}}\leq K_0 e^{-\mu s_0} ;
\ee
\end{itemize}
where $w_0$ is given by
$$w_0=\left(1-\chi_{\frac{1}{\l_0}}\right)\Phi_n+v_0.$$
Note that in view of the $L^{\infty}$ bound \fref{poitwisebound}, the decomposition \fref{geomdeopcopmtiinit} is precisely the one given by Lemma \ref{lemma:decompositionofinitialdata}.\\

\noindent{\em Renormalized flow}. As long as the solution $u(t)$ starting from \fref{geomdeopcopmtiinit} belongs to $X_{\delta}$, Lemma \ref{lemma:decompositionofinitialdata} applies and it can be written
\be
\label{geomdeopcopmti}
u(t,x)=\frac{1}{\l(t)^{\frac{2}{p-1}}}(\Phi_n+\psi+\e)(s,z), \ \ y=\frac{x-x(t)}{\l(t)}
\ee 
with 
\be
\label{orhotpsin}
\psi=\sum_{j=2}^{n+1} a_j\psi_j, \ \ (\e,\psi_j)_\rho=(\e,\pa_k\Phi_n)_\rho=0, \ \ 1\leq j\leq n+1, \ \ 1\leq k\leq 3 .
\ee 
Moreover, as the parameters are Fr\'echet differentiable in $L^{\infty}$, and as $u\in C^1((0,T),L^{\infty})$ from parabolic regularizing effects, the above decomposition is differentiable with respect to time. We also introduce a further decomposition 
\be
\label{definitionv}
v=\psi+\e, \ \ \Phi_n+v=\chi_{\frac 1\lambda}\Phi_n+w.
\ee 
Consider the renormalized time 
$$
s(t)=\int_0^t\frac{d\tau}{\l^2(\tau)}+s_0 .
$$
Injecting \eqref{geomdeopcopmti} into \eqref{eq:heat} yields the renormalized equation
\be
\label{equatione}
\pa_s\e+\mathcal{L}_n\e=F-\Mod
\ee
with the modulation term
\be
\label{modulation}
\Mod=\sum_{j=2}^{n+1}\left[(a_j)_s-\mu_ja_j\right]\psi_j-\left(\lsl+1\right)(\Lambda\Phi_n+ \Lambda \psi)-\xsl\cdot(\nabla \Phi_n+\nabla \psi)
\ee
and the force terms
\bea
\label{erreurf}
&&F=L(\e)+\NL, \ \ L(\e)=\left(\lsl+1\right)\Lambda\e+\xsl\cdot\nabla \e\\
\label{nolineaire}
&&\NL=g(\e+\psi), \ \ g(v)=(\Phi_n+v)^p-\Phi_n^p-p\Phi_n^{p-1}v.
\eea

We claim the following bootstrap proposition.

\begin{proposition}[Bootstrap]
\label{bootstrap}
There exist universal constants $0<\mu, \eta\ll 1$, $K\gg 1$ such that for all $s_0\geq s_0(K,\mu, \eta)\gg 1$ large enough the following holds. For any $\lambda_0$ and $\e_0$ satisfying \fref{scalingsmall}, \eqref{orhotpsininit} and
\be \label{eq:bound initiale varepsilon}
\| (1-\chi_{\frac{1}{\l_0}})\Phi_n+\e_0 \|_{\dot H^{s_c}}+\| \e_0 \|_{H^2_{\rho}}+\| \Delta \e_0 \|_{L^2} \leq e^{-2\mu s_0},
\ee
there exist $(a_2(0),\dots,a_{n+1}(0))$ satisfying \eqref{nvnono} such that the solution starting from $u_0$ given by \fref{geomdeopcopmtiinit}, decomposed according to \eqref{geomdeopcopmti} satisfies for all $s\geq s_0$:
\begin{itemize}
\item control of the scaling:
\be
\label{controlsclaing}
0<\l(s)< e^{-\mu s};
\ee
\item control of the unstable modes:
\be
\label{controlunstable}
\sum_{j=2}^{n+1}|a_j|^2\leq e^{-2\mu s};
\ee
\item control of the exponentially weighted norm: 
\be
\label{poitwiseboundhtwo}
\|\e\|_{H^2_\rho}< Ke^{-\mu s};
\ee
\item control of a Sobolev norm above scaling: 
\be
\label{controlsobolev}
\|\Delta v\|_{L^2}< Ke^{-\mu s};
\ee
\item control of the critical norm:
\be
\label{sobolevciritical}
\|w\|_{\dot{H}^{s_c}}<\eta.
\ee 
\end{itemize}
\end{proposition}

Proposition \ref{bootstrap} is the heart of the analysis, and the corresponding solutions are easily shown to satisfy the conclusions of Theorem \ref{thmmain}. The strategy of the proof follows \cite{martelmulti,MRR}: we prove Proposition \ref{bootstrap} by contradiction using a topological argument \`a la Brouwer: given $(\e_0,\l_0)$ satisfying  \eqref{scalingsmall}, \eqref{eq:bound initiale varepsilon} and \eqref{orhotpsininit}, we assume that for all $(a_2(0),\dots,a_{n+1}(0))$ satisfying \eqref{nvnono}, the exit time 
\be
\label{defexittime}
s^*=\sup\{s\geq s_0\ \ \mbox{such that}\ \ \eqref{controlsclaing}, \eqref{controlunstable},\eqref{poitwiseboundhtwo},\eqref{controlsobolev}, \eqref{sobolevciritical} \ \text{holds} \ \text{on} \ [s_0,s)\}
\ee
is finite 
\be
\label{assumptioncontradiction}
s^*<+\infty
\ee
and look for a contradiction for $0<\mu,\eta,\frac{1}{K}$ small enough and $s_0\geq s_0(K,\mu)$ large enough. From now on, we therefore study the flow on $[s_0,s^*]$ where \eqref{controlsclaing}, \eqref{controlunstable}, \eqref{poitwiseboundhtwo}, \eqref{controlsobolev} and \eqref{sobolevciritical} hold. Using a bootstrap method we show that the bounds \eqref{controlsclaing}, \eqref{poitwiseboundhtwo}, \eqref{controlsobolev} and \eqref{sobolevciritical} can be improved, implying that at time $s^*$ necessarily the unstable modes have grown and \fref{controlunstable} is violated. Since $0$ is a linear repulsive equilibrium for these modes, this would contradict Brouwer fixed point theorem.\\

\noindent From the asymptotic \fref{estgrowth} of $\psi_j$ for $2\leq j\leq n+1$, \fref{nvnono} and \fref{eq:bound initiale varepsilon}, one can fix the constant $K_0$ independently of $(s_0,\mu,)$ such that \fref{inititalboot} holds. Also, note that the bootstrap bounds \eqref{controlunstable}, \eqref{poitwiseboundhtwo}, \eqref{controlsobolev} and \eqref{sobolevciritical} imply the $L^{\infty}$ bound \fref{poitwisebound}, and therefore the decomposition used in the Proposition is well defined since Lemma \ref{lemma:decompositionofinitialdata} applies.

%%%%%%%%%%%%%%%%%%%%%%%%%%%%%%%%%%%%%

\subsection{$L^\infty$ bound}

%%%%%%%%%%%%%%%%%%%%%%%%%%%%%%%%%%%%%

We start with the derivations of {\it unweighted} $L^\infty$ and Sobolev bounds on $v,w$ which will be essential to control nonlinear terms in the sequel and follow from \eqref{controlsobolev}, \eqref{sobolevciritical}.

\begin{lemma}[$L^\infty$ smallness]
There holds 
\be
\label{poitwisebound}
\|v\|_{L^\infty}+\|w\|_{L^\infty}\leq e^{-c\mu s}\leq \eta\ll1
\ee
for some universal constants $c>0$, $0<\eta\ll 1$.
\end{lemma}

\begin{proof} We compute from \eqref{definitionv}: 
\be
\label{nvekononvevo}
w=(1-\chi_{\frac 1\lambda})\Phi_n+v.
\ee 
The self similar decay \eqref{behavioruinfity} and \eqref{controlsobolev} yield: $$\|w\|_{\dot{H}^2}\lesssim \|v\|_{\dot{H}^2}+\|(1-\chi_{\frac 1\lambda})\Phi_n\|_{\dot{H}^2}\lesssim K\left[ e^{-\mu s}+\l(s)^{2-s_c}\right]\leq e^{-c\mu s}.$$ Hence by interpolation using $s_c=\frac 32-\frac{2}{p-1}<\frac 32<2$:
$$\|w\|_{L^\infty}\lesssim \|\hat{w}\|_{L^1}\lesssim \|w\|^{1-\alpha}_{\dot{H}^{s_c}}\|w\|_{\dot{H}^2}^\alpha, \  \ \alpha=\frac{\frac 32-s_c}{2-s_c}$$ which together with \eqref{sobolevciritical} ensures: $$\|w\|_{L^\infty}\lesssim e^{-c\mu s}.$$ The decay \eqref{behavioruinfity} and \eqref{controlsclaing}, \eqref{nvekononvevo} yield the $L^\infty$ smallness for $v$ and conclude the proof.
\end{proof}
%%%%%%%%%%%%%%%%%%%%%%%%%%%%%%%%%%%%%

\subsection{Modulation equations}

%%%%%%%%%%%%%%%%%%%%%%%%%%%%%%%%%%%%%

We now compute the modulation equations which describe the time evolution of the parameters. They are computed in the self-similar zone, and involve the $\rho$ weighted norm.

\begin{lemma}[Modulation equations]
\label{lemmamodulation}
There holds the bounds
\be
\label{modulationequationbound}
\left|\lsl+1\right|+\left|\xsl\right|+\sum_{j=2}^{n+1}|(a_j)_s-\mu_ja_j|\lesssim \|\e\|_{H^1_\rho}^2+\|\Delta v\|_{L^2}^2+\sum_{j=2}^{n+1}|a_j|^2.
\ee
\end{lemma}

\begin{proof} This lemma is a classical consequence of the choice of orthogonality conditions \eqref{orhotpsin}, but the control of the nonlinear term relies in an essential way on the $L^\infty$ smallness \eqref{poitwisebound}.\\

\noindent{\bf step 1} Law for $a_j$. Take the $L^2_\rho$ scalar product of \eqref{equatione} with $\psi_j$ for $2\leq j\leq n+1$, then using \eqref{orhotpsin} and the orthogonality 
\be \label{orthogonalite psii}
(\psi_j,\psi_k)_\rho=\delta_{jk}, \ \ \psi_1=\frac{\Lambda \Phi_n}{\|\Lambda \Phi_n\|_{L^2_\rho}},
\ee
we obtain
$$(a_j)_s-\mu_ja_j=\left(\lsl+1\right)(\Lambda \psi,\psi_j)_\rho+(F,\psi_j)_\rho.$$ First, from \fref{controlunstable} one has
$$
|(\Lambda \psi,\psi_j)_\rho |\lesssim e^{-\mu s}\ll \eta .
$$
We now estimate the $F$-term given by \fref{erreurf} . We use the bound from $p>5$: $$\left||1+z|^{p}-1-pz^{p-1}\right|\lesssim |z|^p+|z|^2$$ to estimate from the $L^\infty$ bound \eqref{poitwisebound}:
\be
\label{poitwisenonlinearterm}
|\NL|\lesssim |\e+\psi|^p+\Phi_n^{p-2}(\e+\psi)^2\lesssim (\e+\psi)^2=v^2.
\ee 
We estimate from the Hardy inequality \eqref{estpoidsglobal}:
\be
\label{contorlnornmesloc}
\int \frac{|\nabla v|^2}{1+|y|^2}+\frac{|v|^2}{1+|y|^4}\lesssim \int |\Delta v|^2+\|v\|_{H^1_\rho}^2\lesssim \int |\Delta v|^2+\|\e\|_{H^1_\rho}^2+\sum_{j=2}^{n+1}|a_j|^2
\ee
and hence using the polynomial bound \eqref{estgrowth}:
\bee
|(\NL,\psi_j)_{\rho}|&\lesssim & \int v^2|\psi_j|\rho\lesssim \int \frac{|v|^2}{1+|y|^4}\lesssim \int |\Delta v|^2+\|v\|_{H^1_\rho}^2\\
& \lesssim & \|\e\|^2_{H^1_\rho}+\|\Delta v\|_{L^2}^2+\sum_{j=2}^{n+1}|a_j|^2.
\eee Next, we integrate by parts and use Cauchy Schwarz and \eqref{estgrowth} to estimate:
\bee
\left|\left(\left(\lsl+1\right)\Lambda\e+\xsl\cdot\nabla \e,\psi_j\right)_\rho\right|\lesssim \left[\left|\lsl+1\right|+\left|\xsl\right|\right]\|\e\|_{L^2_\rho}
\eee
and hence the first bound $$\left|(a_j)_s-\mu_ja_j\right|\lesssim \left(\left|\lsl+1\right|+\left|\xsl\right|\right)\eta+\|\e\|^2_{ H^1_\rho}+\sum_{j=2}^{n+1}|a_j|^2+\|\Delta v\|_{L^2}^2.$$

\noindent{\bf step 2} Law for scaling and translation. We scalarize \eqref{equatione} with $\psi_1=\frac{\Lambda \Phi_n}{\|\Lambda \Phi_n\|_{L^2_\rho}}$ and $\frac{\pa_k\Phi_n}{\|\pa_k\Phi_n\|_{L^2}}$ and obtain in a completely similar way
 $$\left|\lsl+1\right|+\left|\xsl\right|\lesssim \left(\left|\lsl+1\right|+\left|\xsl\right|\right)\eta+\|\e\|^2_{ H^1_\rho}+\sum_{j=2}^{n+1}|a_j|^2+\|\Delta v\|_{L^2}^2.$$ Summing the above estimates and using the smallness of $\eta$  yields \eqref{modulationequationbound}.
\end{proof}

%%%%%%%%%%%%%%%%%%%%%%%%%%%%%%%%%%%%%

\subsection{Energy estimates with exponential weights}

%%%%%%%%%%%%%%%%%%%%%%%%%%%%%%%%%%%%%

We now turn to the proof of exponential decay which is an elementary consequence of the spectral gap estimate \eqref{coerciviteapoids}, the dissipative structure of the flow {\it and} the $L^\infty$ bound \eqref{poitwisebound} to control the non linear term.

\begin{lemma}[Lyapounov control of exponentially weighed norms]
There holds the differential bound
\be
\label{differentialcontrolnorm}
\frac{d}{ds}\|\e\|_{L^2_\rho}^2+c_n\|\e\|_{H^1_\rho}^2 \lesssim \sum_{j=2}^{n+1}|a_j|^4+\|\Delta v\|_{L^2}^4+ \|v\|^2_{L^{\infty}}\left[\|\Delta v\|_{L^2}^2+\sum_{j=2}^{n+1}|a_j|^2\right],
\ee
\bea
\label{differentialcontrolnormbis}
\frac{d}{ds}\|\mathcal{L}_n\e\|_{L^2_\rho}^2+c_n\|\mathcal{L}_n\e\|_{H^1_\rho}^2&\lesssim& \|\e\|_{H^1_\rho}^2+ \sum_{j=2}^{n+1}|a_j|^4+\|\Delta v\|_{L^2}^4\\
\nonumber& + & \|v\|^2_{L^{\infty}}\left[\|\Delta v\|_{L^2}^2+\|\e\|_{H^1_\rho}^2+\sum_{j=2}^{n+1}|a_j|^2\right],
\eea
with $c_n>0$ given by \eqref{coerciviteapoids}.
\end{lemma}

\begin{proof} 
{\bf step 1} $L^2$ weighted bound. We compute from \eqref{equatione}:
\be
\label{vnbeovneoneo}
\frac 12\frac{d}{ds}\|\e\|_{L^2_\rho}^2=(\e,\pa_s\e)_\rho=-(\mathcal{L}_n\e,\e)_\rho+(F-\Mod,\e)_\rho.
\ee
From \eqref{modulation}, \eqref{modulationequationbound}:
\bee
|(\e,\Mod)_\rho|\lesssim \|\e\|_{L^2_\rho}\|\Mod\|_{L^2_\rho} & \lesssim & \|\e\|_{L^2_\rho}\left(\|\e\|_{H^1_\rho}^2+\sum_{j=2}^{n+1}|a_j|^2+\|\Delta v\|_{L^2}^2\right) \\
& \lesssim & \delta \|\e\|_{L^2_\rho}^2+C_{\delta} \left(\|\e\|_{H^1_\rho}^4+\sum_{j=2}^{n+1}|a_j|^4+\|\Delta v\|_{L^2}^4\right)
\eee
for any $\delta>0$. Integrating by parts and using \eqref{weightedesimate}, we estimate
\be
\label{comutareurwirhgt}
|(\e,\Lambda\e)_\rho|+|(\nabla \e,\e)_\rho|\lesssim \int (1+|y|^2)\e^2\rho dy\lesssim\|\e\|_{H^1_\rho}^2
\ee
from which using \eqref{modulationequationbound}:
$$
\left|\left(L(\e),\e\right)_\rho\right|\lesssim \|\e\|^2_{H^1_\rho}\left(\|\e\|_{L^2_\rho}^2+\sum_{j=2}^{n+1}|a_j|^2+\|\Delta v\|_{L^2}^2\right).$$
Finally using \eqref{poitwisenonlinearterm}, \eqref{contorlnornmesloc}:
\bee
|(\NL,\e)_{\rho}|&\lesssim & \int |\e|v^2\rho dy \leq\delta\int|\e|^2\rho+C_\delta \int |v|^4\rho dy\\
& \leq & \delta \int |\e|^2\rho+C_\delta\|v\|^2_{L^{\infty}}\int \frac{|v|^2}{1+|y|^4}dy\\
& \leq & \delta \|\e\|_{L^2_\rho}^2+C_\delta\|v\|^2_{L^{\infty}}\left[\int |\Delta v|^2+\|\e\|_{H^1_\rho}^2+\sum_{j=2}^{n+1}|a_j|^2\right].
\eee

Injecting the collection of above bounds into \eqref{vnbeovneoneo} and using the spectral gap estimate \eqref{coerciviteapoids} with the choice of orthogonality conditions \eqref{orhotpsin} yields
\bee
\frac{d}{ds} \| \e \|^2 & \leq &  -2c_n \| \e \|_{H^1_\rho} \left(1-C(\| \e \|_{H^1_\rho}^2-\sum_{j=2}^{n+1} |a_j|^2-\|†\Delta v\|_{L^2}^2)-C\delta-C_{\delta} \| \e†\|_{H^1_\rho} \right)\\
&&+C_{\delta}\| v\|_{L^{\infty}}\left[ \int \Delta v^2+\|\e\|_{H^1_\rho}^2+\sum_{j=2}^{n+1} |a_j|^2 \right]
\eee
which using the bootstrap bounds \fref{controlunstable}, \fref{poitwiseboundhtwo} and \fref{controlsobolev} gives \eqref{differentialcontrolnorm} for $s_0$ large enough and $\delta$ small enough.\\

\noindent{\bf step 2} $H^2$ weighted bound. Let $$\e_2=\mathcal{L}_n\e,$$ then $\e_2$ satisfies the orthogonality conditions \eqref{orhotpsin}: 
\be
\label{orhtongoanoa}
(\e_2,\psi_j)=(\e_2,\pa_k\Phi_n)=0, \ \ 1\leq j\leq n+1, \ \ 1\leq k\leq 3,
\ee and the equation from \eqref{equatione}:
$$\pa_s\e_2+\mathcal{L}_n\e_2=\mathcal{L}_n(F-\Mod).$$ Hence:
\be
\label{cneknenonee}
\frac{1}{2}\frac{d}{ds}\|\e_2\|_{L^2_\rho}^2=-(\mathcal{L}_n\e_2,\e_2)_\rho+(\mathcal{L}_n(F-\Mod),\e_2)_\rho.
\ee 
We estimate from \eqref{modulationequationbound}:
$$\|\mathcal{L}_n\Mod\|_{L^2_\rho}\lesssim \left|\lsl-1\right|+\left|\xsl\right|+\sum_{j=2}^{n+1}|(a_j)_s-a_j|\lesssim \|\e\|_{H^1_\rho}^2+\sum_{j=2}^{n+1}|a_j|^2+\|\Delta v\|_{L^2}^2.$$We now use the commutator relation 
$$[\Delta, \Lambda]=2\Delta$$ 
to compute 
$$[\mathcal{L}_n,\Lambda]=[-\Delta+\Lambda-p\Phi_n^{p-1},\Lambda]=-2\Delta+p(p-1)\Phi_n^{p-2}r\pr_r\Phi_n=2(\mathcal L_n-\Lambda +p\Phi_n^{p-1})+p(p-1)\Phi_n^{p-2}r\pr_r\Phi_n$$
from which using \eqref{comutareurwirhgt}, \eqref{weightedesimate}:
\bee
|(\e_2,\mathcal{L}_n\Lambda \e)_\rho|&=& \left|(\e_2,[\mathcal{L}_n,\Lambda]\e)_\rho+(\e_2,\Lambda \e_2)_\rho\right|\\
& \lesssim &\|\e_2\|_{H^1_\rho}^2+|(\e_2,\Lambda \e)_\rho|+|(\e_2, \Phi_n^{p-1}\e)_\rho|+|(\e_2, \Phi_n^{p-2}\Lambda\Phi_n\e)_\rho|\\
&\lesssim& \|\e_2\|_{H^1_\rho}^2+\|\e\|_{H^1_\rho}^2
\eee 
and similarly
$$|(\e_2,\mathcal{L}_n\pa_k \e)_\rho|\lesssim \|\e_2\|_{H^1_\rho}^2+\|\e\|_{H^1_\rho}^2.$$ 
Hence from \eqref{modulationequationbound}:
\bee
|(\e_2,\mathcal{L}_nL(\e) )_\rho|\lesssim (\|\e_2\|_{H^1_\rho}^2+\|\e\|_{H^1_\rho}^2)\left(\|\e\|_\rho^2+\sum_{j=2}^{n+1}|a_j|^2+\|\Delta v\|_{L^2}^2\right).
\eee
It remains to estimate the nonlinear term. We first integrate by parts since $\mathcal{L}_n$ is self adjoint for $(\cdot,\cdot)_\rho$ to estimate using the notation \eqref{nolineaire}:
\bee
|(\mathcal{L}_n\NL,\e_2)_\rho| &=& \left|(\nabla \NL,\nabla \e_2)_\rho+\left(\frac{2}{p-1}\NL-p\Phi_n^{p-1}\NL, \e_2\right)_\rho\right|\\
&\lesssim& \left|(\nabla g(v),\nabla \e_2)_\rho\right|+\left|\left(\frac{2}{p-1}g(v)-p\Phi_n^{p-1}g(v), \e_2\right)_\rho\right|.
\eee
We now compute explicitly  
\bea
\label{computationderivee}
\nabla g(v)&=& p\nabla v\left[(\Phi_n+v)^{p-1}-\Phi_n^{p-1}\right]\\
\nonumber & +& p\nabla \Phi_n\left[(\Phi_n+v)^{p-1}-\Phi_n^{p-1}-(p-1)\Phi_n^{p-2}v\right].
\eea
 We estimate by homogeneity with the $L^\infty$ bound \eqref{poitwisebound}: $$|g(v)|\lesssim |v|^2,\,\,\,\, |\nabla g(v)|\lesssim |\nabla v||v|+|v|^2$$ and hence the bound using \eqref{poitwisebound} again:
\bee
&& |(\nabla g(v),\nabla \e_2)_\rho|+\left|\left(\frac{2}{p-1}g(v)-p\Phi_n^{p-1}g(v), \e_2\right)_\rho\right|\\
&\lesssim& \int\left[|v||\nabla(v)|+|v|^2\right]|\nabla\e_2|\rho dy  +\int |\ep_2||v|^2\rho dy\\
&\leq& \delta\|\e_2\|_{H^1_\rho}^2+C_\delta\left[\int|v|^2|\nabla v|^2\rho dy+\int|v|^4\rho dy\right]
\\
&\leq & \delta\|\nabla\e_2\|_{L^2_\rho}^2+C_\delta\|v\|_{L^\infty}^2\left[\int \frac{|\nabla v|^2}{1+|y|^2} dy+\int\frac{|v|^2}{1+|y|^4}dy\right]\\
& \leq & \delta\|\nabla\e_2\|_{L^2_\rho}^2+C_\delta\|v\|_{L^\infty}^2\left[\|\e\|_{H^1_\rho}^2+\sum_{j=2}^{n+1}|a_j|^2+\|\Delta v\|_{L^2}^2\right].
\eee
The collection of above bounds together with the spectral gap estimate \eqref{coerciviteapoids} and the orthogonality conditions \eqref{orhtongoanoa} injected into \eqref{cneknenonee} yields \eqref{differentialcontrolnormbis}.
\end{proof}

\begin{remark} The proof of \eqref{differentialcontrolnorm} is elementary but {\emph requires in an essential way} the $L^\infty$ smallness bound\footnote{or anything above or equal scaling in terms of regularity.} \eqref{poitwisebound}, and in particular the sole control of the $H^1_\rho$ norm cannot suffice to control the nonlinear term $\int|\e|^{p+1}\rho$ due to both the energy super critical nature of the problem and the exponential weight.
\end{remark}

%%%%%%%%%%%%%%%%%%%%%%%%%%%%%%%%%%%%%

\subsection{Outer global $\dot{H}^2$ bound}

%%%%%%%%%%%%%%%%%%%%%%%%%%%%%%%%%%%%%

We recall $$v=\e+\psi$$ and now aim at propagating an {\it unweighted global} $\dot{H}^2$ decay estimate for $v$. We have 
$$\pa_sv-\Delta v -\lsl \Lambda v -\xsl\cdot\nabla v=G$$ with 
\bee
 G&=& \left[\left(\lsl+1\right)\Lambda\Phi_n+\xsl\cdot\nabla \Phi_n\right]+\widehat{NL}, \ \  \widehat{NL}=(\Phi_n+v)^p-\Phi_n^p.
 \eee
 
\begin{lemma}[Global $\dot{H}^2$ bound]
\label{lemmehtwoweight}
There holds the Lyapounov type monotonicity formula
\be
\label{decayweight}
\frac{d}{ds}\left[\frac{1}{\l^{4-\delta-2s_c}}\int|\Delta v|^2dy\right]+\frac{1}{\l^{4-\delta-2s_c}}\int |\nabla \Delta v|^2dy \lesssim  \frac{1}{\l^{4-2s_c-\delta}}\left[\|\e\|_{H^2_\rho}^2+\sum_{j=2}^{n+1}|a_j|^2\right]
\ee
for some universal constant $0<\delta\ll1$.
\end{lemma}

\begin{proof} We compute the $\dot{H}^2$ energy identity:
\bee
\frac12 \frac{d}{ds}\int |\Delta v|^2dy&=&\int \Delta v\Delta\left[\Delta v+\lsl\Lambda v+\xsl\cdot\nabla v+G\right]dy\\
& = & -\int |\nabla \Delta v|^2dy+\int \Delta v\Delta\left[\lsl\Lambda v+\xsl\cdot\nabla v+G\right]dy
\eee
 and estimate all terms.\\
 
\noindent{\bf step 1} Parameters terms. For any $\mu>0$, let $v_\mu=\frac{1}{\mu^{\frac{2}{p-1}}}v\left(\frac{y}{\mu}\right)$, then: $$\int|\Delta v_\mu|^2dy=\frac{1}{\mu^{4-2s_c}}\int|\Delta v|^2dy$$ and hence differentiating and evaluating at $\mu=1$: $$-2\int \Delta v \Delta (\Lambda v)dy=-(4-2s_c)\int|\Delta v |^2dy.$$ Hence
$$\lsl\int \Delta v\Delta(\Lambda v)=  (2-s_c)\lsl\int|\Delta v|^2dy.
$$
Also, integrating by parts:
$$
\int\Delta v\Delta\left(\xsl\cdot\nabla v\right)dy=0.
$$

\noindent{\bf step 2} $G$ terms. Thanks to the decay of the self similar solution from \eqref{behavioruinfity}: 
$$\int |\Delta\Lambda\Phi_n|^2dy + \int |\Delta\nabla\Phi_n|^2dy<+\infty,$$ 
we estimate in brute force using \eqref{modulationequationbound} the terms induced by the self similar solution:
\bee
&&\left|\int \Delta v\Delta\left\{\left[\left(\lsl+1\right)\Lambda \Phi_n+\xsl\cdot\nabla\Phi_n\right]\right\}\right|\\
&\lesssim & \left[\left|\lsl+1\right|+\left|\xsl\right|\right]\|\Delta v\|_{L^2}\leq  \delta \|\Delta v\|_{L^2}^2+C_\delta\left(\|\e\|_{H^1_\rho}^2+\|\Delta v\|_{L^2}^2+\sum_{j=2}^{n+1}|a_j|^2\right)^2\\
& \leq &  \delta \|\Delta v\|_{L^2}^2+C_\delta\left(\|\e\|_{H^1_\rho}^2+\sum_{j=2}^{n+1}|a_j|^2\right).
\eee
It remains to estimate the nonlinear term. We estimate by homogeneity:
\bee
|\Delta  \widehat{NL}|& = & \Bigg|p\Delta \Phi_n\left[(\Phi_n+v)^{p-1}-\Phi_n^{p-1}\right]+p(\Phi_n+v)^{p-1}\Delta v\\
& + & p(p-1)|\nabla \Phi_n|^2\left[(\Phi_n+v)^{p-2}-\Phi_n^{p-2}\right]+p(p-1)|\nabla v|^2(\Phi_n+v)^{p-1}\\
&+& 2p(p-1)(\Phi_n+v)^{p-2}\nabla \Phi_n\cdot\nabla v\Bigg|\\
& \lesssim & |\Delta \Phi_n|(|v|^{p-1}+|\Phi_n|^{p-2}|v|)+|\Delta v|(|v|^{p-1}+|\Phi_n|^{p-1})\\
& + &  |\nabla \Phi_n|^2(|v|^{p-2}+|\Phi_n|^{p-3}|v|)+|\nabla v|^2(|v|^{p-1}+|\Phi_n|^{p-1})+|\nabla v||\nabla \Phi_n|(|\Phi_n|^{p-2}+|v|^{p-2})
\eee
and hence using the self similar decay of $\Phi_n$ and the $L^\infty$ smallness \eqref{poitwisebound}:
\bee
|\Delta  \widehat{NL}| & \lesssim &\left[\frac{|\Delta v|}{1+|y|^2}+\frac{|\nabla v|}{1+|y|^3}+\frac{|v|}{1+|y|^4}\right]+\eta\left[|\Delta v|+\frac{|\nabla v|}{1+|y|}+\frac{|v|}{1+|y|^2}\right]\\
& + &|\nabla v|^2(|v|^{p-1}+|\Phi_n|^{p-1}).
\eee
The linear term is estimated using \eqref{estpoidsglobal}:
\bee
\int \left|\frac{|\Delta v|}{1+|y|^2}+\frac{|\nabla v|}{1+|y|^3}+\frac{|v|}{1+|y|^4}\right|^2&\lesssim& \frac{1}{A^4}\int_{|y|\geq A}|\Delta v|^2+C_A\|v\|_{H^2_\rho}^2\\
& \leq & \delta \int |\Delta v|^2 +C_\delta\left(\|\e\|_{H^2_\rho}^2+\sum_{j=2}^{n+1}|a_j|^2\right)
\eee
and using \eqref{estpoidsglobal} again:
$$
\int\left|\eta\left[|\Delta v|+\frac{|\nabla v|}{1+|y|}+\frac{|v|}{1+|y|^2}\right]\right|^2\lesssim \eta\|\Delta v\|_{L^2}^2+\|\e\|_{H^2_\rho}^2+\sum_{j=2}^{n+1}|a_j|^2.$$
To estimate the nonlinear term, we let $$q_c=\frac{3(p-1)}{2}\ \ \mbox{so that}\ \ \dot{H}^{s_c}\subset L^{q_c}.$$  We estimate using \eqref{poitwisebound} with $6(p-2)>q_c$ and Sobolev:
\bee
&&\int|\nabla v|^4(|v|^{2(p-2)}+|\Phi_n|^{2(p-2)})\lesssim \|\nabla v\|_{L^6}^4\left[\|v\|_{L^{6(p-2)}}^{2(p-2)}+\|\Phi_n\|_{L^{6(p-2)}}^{2(p-2)}\right]\\
&\lesssim &\|\Delta v\|_{L^2}^4\left[\|\Phi_n\|^{2(p-2)}_{L^{6(p-2)}}+\|w\|^{2(p-2)}_{L^{6(p-2)}}\right]\lesssim  \|\Delta v\|_{L^2}^4\left[1+\|w\|^{\frac{p-1}{2}}_{\dot{H}^{s_c}}\right]\leq \delta \|\Delta v\|_{L^2}^2.
\eee
We have therefore obtained
$$\int |\Delta\widehat{NL}|^2\leq \delta \|\Delta v\|_{L^2}^2+C_\delta\left(\|\e\|_{H^2_\rho}^2+\sum_{j=2}^{n+1} |a_j|^2\right).$$
The collection of above bounds and \fref{modulationequationbound} yields \eqref{decayweight}.
\end{proof}

%%%%%%%%%%%%%%%%%%%%%%%%%%%%%%%%%%%%%

\subsection{Control of the critical norm}

%%%%%%%%%%%%%%%%%%%%%%%%%%%%%%%%%%%%%

We now claim the control of the critical norm of $w$ (defined by \fref{definitionv}).

\begin{lemma}[Control of the critical norm]
\label{lemmehtwoweightbis}
There holds the Lyapounov type control
\be
 \label{decaycriticalnorm}
\frac{d}{ds}\int |\nabla^{s_c}w|^2dy+\int|\nabla^{s_c+1}w|^2dy \lesssim \|\e\|_{H^2_\rho}^2+\sum_{j=2}^{n+1}|a_j|^2+  \l^{\delta(2-s_c)}+\|\Delta v\|^{\delta}_{L^2}.
\ee
for some small enough universal constant $0<\delta=\delta(p)\ll1$.
\end{lemma}

\begin{proof} Let 
\be
\label{defwidetilde}
\widetilde{\Phi_n}=\chi_{\frac1\lambda}\Phi_n,
\ee we compute the evolution equation of $w$:
\be
\label{wequation}
\pa_sw-\Delta w=\lsl\Lambda w+\xsl\cdot\nabla w+\widetilde{G}
\ee
with
\bee
\widetilde{G}  &=&  \left(\lsl+1\right)\chi_{\frac1\lambda}\Lambda \Phi_n+\xsl\cdot\nabla \widetilde{\Phi}_n +  2\nabla\chi_{\frac1\lambda}\cdot\nabla \Phi_n+\Delta \chi_{\frac1\lambda} \Phi_n -(\chi_{\frac1\lambda}-\chi^p_{\frac1\lambda})\Phi^p_n+\widetilde{\NL}, \\
\widetilde{\NL} &=&  (\widetilde{\Phi_n}+w)^p-(\widetilde{\Phi_n})^p.
\eee
Observe from the space localization of the cut, from the decay of the self similar solution, and from \eqref{controlsobolev} and \eqref{sobolevciritical}:
\be
\label{controlsoboevw}
\forall s_c\leq s\leq 2, \ \ \|w\|_{\dot{H}^s}\lesssim \eta.
\ee
We compute:
\bee
\frac12 \frac{d}{ds}\int |\nabla^{s_c} w|^2dy&=&\int \nabla^{s_c}w \cdot\nabla^{s_c}\left[\Delta w+\lsl\Lambda w+\xsl\cdot\nabla w+\widetilde{G}\right]dy\\
& = & -\int |\nabla^{s_c+1} w|^2+\int \nabla^{s_c}w\cdot\nabla^{s_c}\left[\lsl\Lambda w+\xsl\cdot\nabla w+\widetilde{G}\right]dy
\eee
 and estimate all terms.\\
 
\noindent{\bf step 1} Parameters terms. For any $\mu>0$, let $w_\mu=\frac{1}{\mu^{\frac{2}{p-1}}}w\left(\frac{y}{\mu}\right)$, then : $$\int|\nabla^{s_c} w_\mu|^2dy=\int|\nabla^{s_c} w|^2dy$$ and hence differentiating at $\mu=1$: $$-2\int\nabla^{s_c} w\cdot \nabla^{s_c} (\Lambda w)dy=0.$$ Integrating by parts:
$$
\int\nabla^{s_c}w\cdot\nabla^{s_c}\left(\xsl\cdot\nabla w\right)dy=0.
$$

\noindent{\bf step 2} $\widetilde{G}$ terms. The decay of the self similar solution and the space localization of the cut ensure using $1<s_c<2$:
\bee
&&\left\|2\nabla\chi_{\frac1\lambda}\cdot\nabla \Phi_n+\Delta \chi_{\frac1\lambda} \Phi_n\right\|_{\dot{H}^{s_c}}\\
& \lesssim & \left\|2\nabla\chi_{\frac1\lambda}\cdot\nabla \Phi_n+\Delta \chi_{\frac1\lambda} \Phi_n\right\|_{\dot{H}^1}^{2-s_c} \left\|2\nabla\chi_{\frac1\lambda}\cdot\nabla \Phi_n+\Delta \chi_{\frac1\lambda} \Phi_n\right\|_{\dot{H}^2}^{s_c-1}\\
& \lesssim & \left(\frac{\l^2}{\l^{s_c-1}}\right)^{2-s_c}\left(\frac{\l^2}{\l^{s_c-2}}\right)^{s_c-1}\lesssim \l^{2},
\eee
and similarly 
\bee
&&\left\|\left(\chi_{\frac1\lambda}-\chi_{\frac1\lambda}^p\right)\Phi_n^p\right\|_{\dot{H}^{s_c}}\lesssim \left\|\left(\chi_{\frac1\lambda}-\chi_{\frac1\lambda}^p\right)\Phi_n^p\right\|_{\dot{H}^{1}}^{2-s_c}\left\|\left(\chi_{\frac1\lambda}-\chi_{\frac1\lambda}^p\right)\Phi_n^p\right\|_{\dot{H}^{2}}^{s_c-1}\\
& \lesssim & (\l^{3-s_c})^{2-s_c}(\l^{4-s_c})^{s_c-1}\lesssim \l^2.
\eee
Using \eqref{modulationequationbound}:
\bee
\left\|\left(\lsl+1\right)\chi_{\frac{1}{\l}}\Lambda \Phi_n+\xsl\cdot\nabla(\chi_\frac 1\lambda\Phi_n)\right\|_{\dot{H}^{s_c}}\\
\lesssim \left|\xsl\right|+\left|\lsl+1\right|\lesssim \|\e\|_{L^2_\rho}^2+\sum_{j=2}^{n+1}|a_j|^2+\|\Delta v\|_{L^2}^2.
\eee 
We now turn to the control of the nonlinear term and claim the bound:
\be
\label{controlnonlinearterm}
\|\nabla^{s_c}\widetilde{NL}\|_{L^2}\lesssim \|\nabla^{s_c+\alpha} w\|_{L^2}
\ee
for some small enough universal constant $0<\alpha=\alpha(p)\ll1$. Assume \eqref{controlnonlinearterm}, we then interpolate with $\delta=\frac{\alpha}{2-s_c}$ and use \eqref{nvekononvevo}, \eqref{sobolevciritical} and the decay of the self similar solution to estimate: 
$$ \|\nabla^{s_c+\alpha} w\|_{L^2}\lesssim  \|\nabla^{s_c} w\|^{1-\delta}_{L^2} \|\Delta w\|^\delta_{L^2}\lesssim \l^{\delta(2-s_c)}+\|\Delta v\|^{\delta}_{L^2},$$ 
and the collection of above bounds yields \eqref{decaycriticalnorm}.\\

\noindent{\it Proof of \eqref{controlnonlinearterm}}. We compute
\bee
\nabla \widetilde{NL}&=&p\nabla(\widetilde{\Phi_n}+w)(\widetilde{\Phi_n}+w)^{p-1}-p\nabla \widetilde{\Phi_n}\widetilde{\Phi_n}^{p-1}\\
& = & p\nabla\widetilde{\Phi_n}\left[(\widetilde{\Phi_n}+w)^{p-1}-\widetilde{\Phi_n}^{p-1}\right]+p\nabla w(\widetilde{\Phi_n}+w)^{p-1}\\
&=& pg_1(w)\nabla(\widetilde{\Phi_n}+w)+p\widetilde{\Phi_n}^{p-1}\nabla w
\eee
with $$g_1(w)=(\widetilde{\Phi_n}+w)^{p-1}-\widetilde{\Phi_n}^{p-1}.$$ 
Hence letting 
$$s_c=1+\nu, \ \ 0<\nu=\frac12-\frac{2}{p-1}<\frac 12,$$ 
we estimate:
\be
\label{esnioenoveno}
\|\nabla^{s_c}\widetilde{\NL}\|_{L^2}\lesssim\left\|\nabla^\nu \left[g_1(w)\nabla(\widetilde{\Phi_n}+w)\right]\right\|_{L^2}+\left\|\nabla^\nu\left(\widetilde{\Phi_n}^{p-1}\nabla w\right)\right\|_{L^2}.
\ee
For the first term, we use the following commutator estimate proved in Appendix \ref{sec:appendixcommutatorestimate}: let $$0<\nu<1, \ \ 1<p_1,p_2, p_3, p_4<+\infty,\ \ \frac{1}{2}=\frac1{p_1}+\frac 1{p_2}=\frac1{p_3}+\frac1{p_4}$$ then 
\be
\label{noneoneoneevoe}
\|\nabla^\nu(uv)\|_{L^2}\lesssim \|u\|_{\dot{B}^\nu_{p_1, 2}}\|v\|_{L^{p_2}}+\|u\|_{L^{p_4}}\|v\|_{\dot{B}^\nu_{p_3, 2}},
\ee 
where we use here the standard space formulation of Besov norms for $0<s<1$ and $1\leq p<+\infty$\footnote{see for example \cite{cazenavebook}.}:
\bea\label{eq:standarddefinitionbesovnorms}
\|u\|_{\dot{B}^s_{p, 2}} &\sim&\left(\int_0^{+\infty}\left(\frac{\sup_{|y|\leq t}\|u(\cdot-y)-u(\cdot)\|_{L^{p}}}{t^s}\right)^2\frac{dt}t\right)^{\frac 12}.
\eea
We pick a small enough $0<\alpha\ll1$ to be chosen later and 
\bee
&&\frac{1}{p_1}=\frac{1}{3}+\frac\alpha 3, \ \ \frac{1}{p_2}=\frac16-\frac{\alpha}{3}\\
&& \frac{1}{p_3}=\frac{1+\alpha+\nu}{3}, \ \ \frac{1}{p_4}=\frac{1-2(\alpha+\nu)}{6} .
\eee 
Observe that $$-\nu+\frac{3}{p_2}=\frac{3}{p_4}$$ and hence from \eqref{noneoneoneevoe}, the embedding of $\dot{H}^{s, p}$ in $\dot{B}^s_{p, 2}$, and Sobolev\footnote{using $\frac 32-\frac{3}{p_3}=\frac12-(\alpha+\nu)>0$.}:
\bee
&&\left\|\nabla^\nu \left[g_1(w)\nabla(\widetilde{\Phi_n}+w)\right]\right\|_{L^2}\\
&\lesssim&  \|\nabla(\widetilde{\Phi_n}+w)\|_{L^{p_1}}\|g_1(w)\|_{\dot{B}^\nu_{p_2, 2}}+ \|\nabla(\widetilde{\Phi_n}+w)\|_{\dot{B}^\nu_{p_3, 2}}\|g_1(w)\|_{L^{p_4}}\\
& \lesssim & \|\nabla^{1+\frac32-\frac{3}{p_1}}(\widetilde{\Phi_n}+w)\|_{L^2}\|g_1(w)\|_{\dot{B}^\nu_{p_2, 2}}+ \|\nabla^{1+\nu+\frac 32-\frac{3}{p_3}}(\widetilde{\Phi_n}+w)\|_{L^2}\|\nabla^\nu g_1(w)\|_{L^{p_2}}\\
& \lesssim & \|\nabla^{\frac 32-\alpha}(\widetilde{\Phi_n}+w)\|_{L^2}\|g_1(w)\|_{\dot{B}^\nu_{p_2, 2}}.
\eee
Since $s_c=\frac32-\frac{2}{p-1}<\frac 32$, we may pick $0<\alpha\ll1 $ with $\frac 32-\alpha>s_c$ and hence using \eqref{controlsoboevw} and the decay of the self similar solution: $$\|\nabla^{\frac 32-\alpha}(\widetilde{\Phi_n}+w)\|_{L^2}\lesssim 1.$$
Let now $$f(z)=(1+z)^{p-1}-1$$ then $f(0)=0$ and  
$$ |f(z_2)-f(z_1)|=\left|\int_{z_1}^{z_2}f'(\tau)d\tau\right|\lesssim \int_{z_1}^{z_2}(1+|\tau|^{p-2})d\tau\lesssim |z_2-z_1|(1+|z_1|^{p-2}+|z_2|^{p-2})$$ and hence by homogeneity:
$$|g_1(w_2)-g_1(w_1)|\lesssim |w_2-w_1|(|\widetilde{\Phi_n}|^{p-2}+|w_2|^{p-2}+|w_1|^{p-2}).$$ Using the $L^\infty$ bound \eqref{poitwisebound},  \eqref{eq:standarddefinitionbesovnorms}, and Sobolev\footnote{Here we use that $\dot{B}^s_{2, 2}$ embeds in $\dot{B}^t_{p, 2}$ with $s-3/2=t-3/p$ for $p\geq 2$, and $\dot{B}^s_{2, 2}=\dot{H}^s$.}
\bee
\|g_1(w)\|_{\dot{B}^\nu_{p_2, 2}}&\lesssim&\left(\int_0^{+\infty}\left(\frac{\sup_{|y|\leq t}\|g_1(w(\cdot-y))-g_1(w(\cdot))\|_{L^{p_2}}}{t^{\nu}}\right)^2\frac{dt}t\right)^{\frac 12}\\
& \lesssim & \left(\int_0^{+\infty}\left(\frac{\sup_{|y|\leq t}\|w(\cdot-y)-w(\cdot)\|_{L^{p_2}}}{t^{\nu}}\right)^2\frac{dt}t\right)^{\frac 12}\sim \|w\|_{\dot{B}^\nu_{p_2, 2}}\\
& \lesssim & \|\nabla^{\nu+\frac 32-\frac{3}{p_2}}w\|_{L^2}=\|\nabla^{s_c+\alpha} w\|_{L^2}.
\eee
The collection of above bounds yields the control of the first term of \eqref{esnioenoveno}:
$$ \|\nabla^\nu \left[g_1(w)\nabla(\widetilde{\Phi_n}+w)\right]\|_{L^2}\lesssim \|\nabla^{s_c+\alpha} w\|_{L^2}.
$$
For the second term in \eqref{esnioenoveno}, we recall the following estimate proved in \cite{MRR}: let $ 0<\nu<1$ and $\mu>0$ with $\mu+\nu<\frac 32$, let $f$ smooth radially symmetric with 
\be
\label{veioboen}
|\pa_r^kf|\lesssim\frac{1}{1+r^{\mu+k}}, \ \ k=0,1,
\ee then there holds the generalized Hardy bound 
\be
\label{hardyboud}
\|\nabla^{\nu}(uf)\|_{L^2}\lesssim \|\nabla^{\nu+\mu}f\|_{L^2}.
\ee 
We then pick again a small enough $0<\alpha\ll1$ and let $$\mu=\alpha, \ \ \mu+\nu=\nu+\alpha=s_c-1+\alpha<\frac 32$$ for $0<\alpha\ll1 $ small enough, and $f=(\chi_{\frac1\lambda} \Phi_n)^{p-1}$ satisfies $$|\pa_r^kf|\lesssim\frac{1}{1+r^{2+k}}\lesssim\frac{1}{1+r^{\mu+k}}.$$ Hence 
$$\|\nabla^\nu\left(\widetilde{\Phi_n}^{p-1}\nabla w\right)\|_{L^2}\lesssim \|\nabla^{\nu+\mu+1}w\|_{L^2}=\|\nabla^{s_c+\alpha}w\|_{L^2}.$$ This concludes the proof of \eqref{controlnonlinearterm}.
\end{proof}

%%%%%%%%%%%%%%%%%%%%%%%%%%%%%%%%%%%%%

\subsection{Conclusion}\label{sec:conclusionproofpropboot}

%%%%%%%%%%%%%%%%%%%%%%%%%%%%%%%%%%%%%

We are now in position to conclude the proof of Proposition \ref{bootstrap} which then easily implies Theorem \ref{thmmain}.\\

\begin{proof}[Proof of Proposition \ref{bootstrap}] We recall that we are arguing by contradiction assuming \eqref{assumptioncontradiction}. We first  show that the bounds \eqref{controlsclaing}, \eqref{poitwiseboundhtwo}, \eqref{controlsobolev} and \eqref{sobolevciritical} can be improved on $[s_0,s^*]$, and then, the existence of the data $(a_j(0))_{2\leq j\leq n+1}$ follows from a classical topological argument \`a la Brouwer.\\

\noindent{\bf step 1} Improved scaling control. We estimate from \eqref{controlunstable}, \eqref{poitwiseboundhtwo}, \eqref{controlsobolev},  \eqref{modulationequationbound}:
\be
\label{mpdulaitoneauito}
\left|\lsl+1\right|\lesssim K^2 e^{-2\mu s}
\ee and hence after integration:
$$\left|\log \left(\frac{\l(s)}{\l_0}\right)+s-s_0\right|\lesssim \int_{s_0}^{+\infty}K^2e^{-2\mu\tau}d\tau\lesssim 1+o(1)$$  for $s_0$ large enough, which together with \eqref{scalingsmall} implies: 
\be
\label{improvedcontrolltworhoone}
\l(s)=\left(\l(s_0)e^{s_0}\right)e^{-s}(1+o(1))\ \ \mbox{and hence}\ \ \frac{e^{-s}}{2}\leq \lambda(s)\leq 2e^{-s}.
\ee

\noindent{\bf step 2} Improved Sobolev bounds.\\
\noindent{\em $L^2_\rho$ bound}. From \eqref{differentialcontrolnorm}, \eqref{controlunstable}, \eqref{controlsobolev}, \eqref{poitwisebound}:
$$\frac{d}{ds}\|\e\|_{L^2_\rho}^2+c_n\|\e\|_{H^1_\rho}^2\lesssim (1+K^4)e^{-4\mu s}+K^2e^{-2\mu s}e^{-2c\mu s}\leq e^{-(2+c)\mu s}$$ for $s\geq s_0$ large enough. From now on, we may fix once and for all the value 
\be
\label{defmu}
\mu=\frac{c_n}{4}
\ee 
and hence
 \be
 \label{nvekonevonove}
 \frac{d}{ds}\|\e\|_{L^2_\rho}^2+4\mu \|\e\|_{H^1_\rho}^2\leq e^{-(2+c)\mu s}
 \ee
which time integration yields using \eqref{inititalboot}:
\bea
\label{improvedcontrolltworhotwo}
\nonumber \|\e(s)\|_{L^2_\rho}^2+2\mu e^{-2\mu s}\int_{s_0}^se^{2\mu \sigma}\|\e\|_{H^1_\rho}^2d\sigma&\leq&  \left(e^{2\mu s_0}\|\e(s_0)\|_{L^2_\rho}^2\right)e^{-2\mu s}+e^{-2\mu s}\int_{s_0}^{s}e^{-\mu c\tau}d\tau\\
& \lesssim & K^2_0 e^{-2\mu s}.
\eea

\noindent{\em $H^2_\rho$ bound}. We estimate from \eqref{differentialcontrolnormbis} like for the proof of \eqref{nvekonevonove}:
$$\frac{d}{ds}\|\mathcal{L}_n\e\|_{L^2_\rho}^2+4\mu\|\mathcal{L}_n\e\|_{H^1_\rho}^2\lesssim \|\e\|_{H^1_\rho}^2+e^{-(2+c)\mu s} $$ whose time integration with the initial bound \eqref{inititalboot} and the bound \eqref{improvedcontrolltworhotwo} ensures:
$$
\|\mathcal L_n\e(s)\|_{L^2_\rho}^2\lesssim K^2_0 e^{-2\mu s}.
$$
We recall $$(\mathcal L_n\e,\e)_\rho=\|\nabla \e\|^2_{L^2_\rho}+\int\left(\frac 2{p-1}-p\Phi_n^{p-1}\right)|\e|^2\rho dy$$ and hence we first estimate from the spectral bound \eqref{coerciviteapoids}, the orthogonality conditions \eqref{orhotpsin}, and Cauchy-Schwarz:
\be
\label{koneonvvkonvoneo}
\|\nabla \e\|_{L^2_\rho}^2\leq (\mathcal L_n\e,\e)_\rho+C\|\e\|_{L^2_\rho}^2\lesssim  \|\mathcal L_n\e\|_{L^2_\rho}^2+\|\e\|_{L^2_\rho}^2\lesssim K^2_0 e^{-2\mu s}.
\ee 
This yields using \eqref{estimationapoinds}:
\be
\label{imprvedhtwobound}
\|\e\|_{H^2_\rho}^2\lesssim \|\mathcal L_n\e\|_{L^2_\rho}^2+\|\e\|^2_{H^1_\rho}
\ee
and hence the improved bound
\be
\label{inoeeioenoen}
\|\e\|_{H^2_\rho}^2 \lesssim K_0^2e^{-2\mu s}.
\ee

\noindent{\it $\dot{H}^2$ bound}. We rewrite \eqref{decayweight} using \eqref{controlunstable}, \eqref{modulationequationbound}, \eqref{inoeeioenoen}
$$\frac{d}{ds}\|\Delta v\|^2_{L^2}+(4-\delta-2s_c)\|\Delta v\|_{L^2}^2\lesssim K^2_0e^{-2\mu s}.$$ By possibly diminishing the value of $c_n$, we may always assume $$4-\delta-2s_c>c_n=4\mu$$
and hence from \eqref{inititalboot}: 
\bea\label{eq:improvedH2dotboundforv}
\|\Delta v\|_{L^2}^2\leq K_0^2e^{-4\mu s}e^{4\mu s_0}e^{-2\mu s_0}+e^{-4\mu s}\int_{s_0}^s K_0^2e^{4\mu \tau}e^{-2\mu \tau}d\tau\lesssim K_0^2e^{-2\mu s}.
\eea

\noindent{\it $\dot{H}^{s_c}$ bound}. We now rewrite \eqref{decaycriticalnorm} using \eqref{controlsclaing}-\eqref{sobolevciritical}:
$$\frac{d}{ds}\int \|\nabla^{s_c}w\|_{L^2}^2\leq  e^{-c \mu s} $$
for some universal constant $c>0$ which time integration using \eqref{inititalboot} ensures:
\be
\label{skjnjknkenekj}
\|\nabla^{s_c}w(s)\|_{L^2}^2\lesssim \|\nabla^{s_c}w(s_0)\|_{L^2}^2+e^{-cs_0}<\frac{\eta}{2}
\ee 
 for $s_0$ large enough.\\

\noindent{\bf step 3} The Brouwer fixed point argument. We conclude from \eqref{improvedcontrolltworhoone}, \eqref{inoeeioenoen}, \eqref{eq:improvedH2dotboundforv}, \eqref{skjnjknkenekj}, the definition \eqref{defexittime} of $s^*$ and a simple continuity argument that the contradiction assumption \eqref{assumptioncontradiction} implies from \eqref{controlunstable}: 
\be
\label{exitondition}
\sum_{j=2}^{n+1}|a_j(s^*)|^2=e^{-2\mu s^*}.
\ee 
Moreover, the vector field is strictly outgoing from \eqref{modulationequationbound}, \eqref{controlunstable}, \eqref{poitwiseboundhtwo}, \eqref{controlsobolev}:
\bee
\frac12 \frac{d}{ds}\sum_{j=2}^{n+1}|a_je^{\mu s}|^2&=&\sum_{j=2}^{n+1}a_je^{2\mu s}((a_j)_s+\mu a_j)=  \sum_{j=2}^{n+1}a_je^{2\mu s}\left[(\mu+\mu_j) a_j+O\left(K^2e^{-2\mu s}\right)\right]\\
& \geq & \mu\sum_{j=2}^{n+1}|a_je^{\mu s}|^2+O\left(K^2e^{-\mu s}\right)
\eee
from which $$\left(\frac{d}{ds}\sum_{j=2}^{n+1}|a_je^{\mu s}|^2\right)(s^*)>\mu+O(K^2e^{-\mu s_0})>0$$ for $s_0$ large enough. We conclude from standard argument that the map $$(a_j(0)e^{\mu s_0})_{2\leq j\leq n+1}\mapsto (a_j(s^*)e^{\mu s^*})_{2\leq j\leq n+1}$$ is continuous in the unit ball of $\Bbb R^{n}$, and the identity on its boundary, a contradiction to Brouwer's theorem. This concludes the proof of Proposition \ref{bootstrap}.
\end{proof}

We are now in position to conclude the proof of Theorem \ref{thmmain}.

\begin{proof}[Proof of Theorem \ref{thmmain}] Let an initial data as in Proposition \ref{bootstrap}, then the corresponding solution $u(s,y)$ admits on $[s_0,+\infty)$ a decomposition \eqref{geomdeopcopmti} with the bounds \eqref{controlunstable}, \eqref{poitwisebound}, \eqref{controlsobolev}, \eqref{sobolevciritical}, \eqref{improvedcontrolltworhoone}.\\

\noindent{\bf step 1} Self similar time blow up. Using \eqref{improvedcontrolltworhoone}, the life space of the solution $u$ is finite 
$$T=\int_{s_0}^{+\infty}\l^2(s)ds\lesssim \int_{s_0}^{+\infty}e^{-2s}ds<+\infty,$$ 
and hence $$T-t=\int_{s}^{+\infty}\l^2(s)ds\sim  e^{-2s}.$$ We may therefore rewrite \eqref{mpdulaitoneauito}: $$|\l \l_t+1|\lesssim (T-t)^{\mu}$$ 
and integrating in time using $\l(T)=0$ yields 
\be
\label{loicaling}
\l(t)=\sqrt{(2+o(1))(T-t)}.
\ee Also from \eqref{modulationequationbound}: 
$$\int_0^T|x_t|=\int_{s_0}^{+\infty}|x_s| ds\lesssim \int_{s_0}^{+\infty} e^{-s-2\mu s}ds<+\infty$$
and \eqref{behaviourxpoitn} is proved.\\

\noindent{\bf step 2} Asymptotic stability above scaling. We now prove \eqref{sobolevone} and \eqref{normioeoe}. We first estimate from \eqref{nvekononvevo} using the self similar decay of $\Phi_n$: \bee
\|w\|_{\dot{H}^2}&\lesssim & \|v\|_{\dot{H}^2}+\|(1-\chi_{\frac 1\lambda})\Phi_n\|_{\dot{H}^2}\lesssim  e^{-2\mu s}+\l^{2-s_c}(s)\\
&\to & 0\ \ \mbox{as}\ \ t\to T.
\eee
Hence from \eqref{sobolevciritical}: $$\forall s_c<\sigma\leq 2, \ \ \lim_{s\to +\infty}\|w(s)\|_{\dot{H}^\sigma}=0$$ which using \eqref{nvekononvevo} and the self similar decay of $\Phi_n$ again implies $$\forall s_c<\sigma\leq 2, \ \ \lim_{s\to +\infty}\|v(s)\|_{\dot{H}^\sigma}=0,$$ this is \eqref{sobolevone}. At the critical level, we have from \eqref{geomdeopcopmti}, \eqref{definitionv} and the sharp self similar decay from Proposition \ref{prop:exteriorsolution}:
$$\|u(t)\|_{\dot{H}^{s_c}}=\|\chi_{\frac{1}{\lambda}}\Phi_n+w\|_{\dot{H}^{s_c}}=c_n(1+o(1))\sqrt{|\log \lambda|}, \ \ c_n\neq 0,$$ and \eqref{loicaling} now yields \eqref{normioeoe}.\\

\noindent{\bf step 3} Boundedness below scaling. We now prove \eqref{sobolevtwo}.\\
\noindent{\it Control of the Dirichlet energy}. Recall the notation \eqref{defwidetilde}
and compute by rescaling using the self similar decay of $\Phi_n$:
$$\l^{2(s_c-1)}\left[\|\nabla \widetilde{\Phi_n}\|_{L^2}^2+\|\widetilde{\Phi_n}\|_{L^{p+1}}^{p+1}\right]\lesssim 1.$$ Hence the dissipation of energy which is translation invariant ensures
\bee
\l^{2(s_c-1)}\|\nabla w\|_{L^2}^2&\lesssim& \l^{2(s_c-1)}\left[\|\nabla (\widetilde{\Phi_n}+w)\|_{L^2}^2+\|\nabla\widetilde{\Phi_n}\|_{L^2}^2\right]\lesssim 1+2E(u)+\frac{2}{p+1}\|u\|_{L^{p+1}}^{p+1}\\
& \lesssim & 1+|E_0|+\l^{2(s_c-1)}\|w\|_{L^{p+1}}^{p+1}.
\eee
We now interpolate using the smallness\footnote{this is the only place in the proof where we use that the critical norm is small, bounded suffices everywhere else.} \eqref{sobolevciritical}
$$\|w\|_{L^{p+1}}^{p+1}\lesssim \|w\|_{\dot{H}^{s_c}}^{p-1}\|\nabla w\|_{L^2}^2\lesssim \eta \|\nabla w\|_{L^2}^2$$ and hence 
\be
\label{diehihuieboud}
\l^{2(s_c-1)}\|\nabla w\|_{L^2}^2\lesssim C(u_0)
\ee 
and $$\|\nabla u\|_{L^2}^2\lesssim \l^{2(s_c-1)}\left[\|\nabla\widetilde{\Phi_n}\|_{L^2}^2+\|\nabla w\|_{L^2}^2\right]\lesssim 1.$$
\noindent{\it Proof of \eqref{sobolevtwo}}. Let now $1\leq\sigma<s_c$, then using \eqref{sobolevciritical}, \eqref{diehihuieboud} and interpolation:
\bee
\|\nabla^\sigma u\|_{L^2} &\lesssim & \l^{s_c-\sigma}\|\nabla^\sigma \widetilde{\Phi_n}\|_{L^2}+\l^{s_c-\sigma}\|\nabla^\sigma w\|_{L^2}\lesssim 1+\l^{s_c-\sigma}\|\nabla w\|^{\frac{s_c-\sigma}{s_c-1}}_{L^2}\|\nabla^{s_c}w\|^{\frac{\sigma-1}{s_c-1}}_{L^2}\\
&\lesssim& 1+\left(\l^{s_c-1}\|\nabla w\|_{L^2}\right)^{\frac{s_c-\sigma}{s_c-1}}\lesssim C(u_0)
\eee
and \eqref{sobolevtwo} is proved. This concludes the proof of Theorem \ref{thmmain}.
\end{proof}

%%%%%%%%%%%%%%%%%%%%%%%%%%%%%%%%%%%%%

\subsection{The Lipschitz dependence}

We now state the Lipschitz aspect of the set of solutions constructed in this paper.

\begin{proposition}[Lipschitz dependence] \label{pr:lipschitz}
Let $s_0\gg 1$, $\e^{(1)}_0$ and $\e^{(2)}_0$ satisfy \fref{orhotpsininit} and \fref{eq:bound initiale varepsilon}, and take $\lambda_0^{(1)}=\lambda^{(2)}_0=e^{-s_0}$. Then the parameters $(a_j^{(1)}(0))_{2\leq j \leq n+1}$ and $(a_j^{(2)}(0))_{2\leq j \leq n+1}$, associated by Proposition \ref{bootstrap} to $(\e^{(1)},\lambda_0^{(1)})$ and $(\e^{(2)},\lambda_0^{(2)})$ respectively, satisfy:
\be \label{lip:lipschitz}
\sum_{j=2}^{n+1} \left|a_j^{(1)}(0)-a_j^{(2)}(0)\right|^2\lesssim \left\| \e^{(1)}_0 -\e^{(2)}_0\right\|^2_{L^2_\rho} .
\ee
\end{proposition}

\begin{proof}
The idea of the proof is classical, see for instance \cite{DKSW}. We study the difference of two solutions, and use the bounds we already derived in the existence result as a priori bounds now. This allows us to control the difference of solutions at a low regularity level which is sufficient to conclude. \\

We use the superscripts $(i)$, $i=1,2$ for all variables associated to the two solutions respectively: $u^{(i)}$ for \fref{geomdeopcopmti}, $v^{(i)}$ for \fref{definitionv}, $\psi^{(i)}$ for \fref{orhotpsin}, $\lambda^{(i)}$ for the scales and $x^{(i)}$ for the central points. The differences are denoted by
$$
\triangle \e := \e^{(1)}-\e^{(2)}, \ \triangle a_j := a^{(1)}_j-a^{(2)}_j, \ \triangle v := v^{(1)}-v^{(2)}.
$$
We compare the two renormalized solutions at the same renormalized time $s$. The time evolution for the difference is given by
\be \label{lip:evolution}
\begin{array}{r c l}
\triangle \e_s+\mathcal L_n \triangle \e &=& \frac{d}{ds}\left[\text{log}\left( \frac{\lambda^{(1)}}{\lambda^{(2)}} \right) \right] \Lambda (\Phi_n+v^{(2)}) +\left(\frac{x_s^{(1)}}{\lambda^{(1)}}-\frac{x^{(2)}_s}{\lambda^{(2)}} \right).\nabla (\Phi_n+v^{(2)}) \\
&&\displaystyle -\sum_{j=2}^{n+1} (\triangle a_{j,s}-\mu_j \triangle a_j ) \psi_j+\left(\frac{\lambda_s^{(1)}}{\lambda^{(1)}}+1 \right)\Lambda \triangle v \\
&& +\frac{x_s^{(1)}}{\lambda^{(1)}}.\nabla \triangle v+\left[(\Phi_n+v^{(1)})^p -(\Phi_n+v^{(2)})^p-p\Phi_n^{p-1}\triangle v \right] .
\end{array}
\ee

\noindent \textbf{step 1} Modulation equations. We claim that
\bea 
\nonumber && \left|\frac{d}{ds} \text{log}\left(\frac{\lambda^{(1)}}{\lambda^{(2)}} \right) \right|+\left|\frac{x_s^{(1)}}{\lambda^{(1)}}-\frac{x_s^{(2)}}{\lambda^{(2)}} \right|+\sum_{j=2}^{n+1} \left|\triangle a_{j,s}-\mu_j \triangle a_j \right| \\
\label{lip:modulation} &\lesssim & e^{-c\mu s}\left( \|\triangle \e \|_{L^2_\rho}+ \sum_{j=2}^{n+1} |\triangle a_j|\right) .
\eea
We now show this estimate. Taking the scalar product of \fref{lip:evolution} with $\psi_1=\frac{\Lambda \Phi_n}{\| \Lambda \Phi_n\|_{L^2_\rho}}$, using the orthogonality conditions \fref{orhotpsin} and \fref{orthogonalite psii} and the fact that $\psi_j$ is radial for $1\leq j \leq n+1$, yields the identity
\be \label{lip:modulation expression}
\begin{array}{r c l}
& \frac{d}{ds}\left[\text{log}\left( \frac{\lambda^{(1)}}{\lambda^{(2)}} \right) \right] (\Lambda (\Phi_n+v^{(2)}),\psi_1)_\rho \\
= & - \left(\left(\frac{x_s^{(1)}}{\lambda^{(1)}}-\frac{x^{(2)}_s}{\lambda^{(2)}} \right).\nabla \e^{(2)},\psi_1 \right)_\rho-\left(\frac{\lambda_s^{(1)}}{\lambda^{(1)}}+1 \right)(\Lambda \triangle v,\psi_1)_\rho-\left( \frac{x_s^{(1)}}{\lambda^{(1)}}.\nabla \triangle \e,\psi_1\right)_\rho  \\
&- \left( (\Phi_n+v^{(1)})^p -(\Phi_n+v^{(2)})^p-p\Phi_n^{p-1}\triangle v ,\psi_1 \right)_\rho
\end{array}
\ee 
and we now estimate each term. The coercivity \fref{weightedesimate} and the bounds \fref{controlunstable} and \fref{poitwiseboundhtwo} yields
$$
(\Lambda (\Phi_n+v^{(2)}),\psi_1)_\rho=1+O(e^{-\mu s}),
$$
$$
\left|\left(\left(\frac{x_s^{(1)}}{\lambda^{(1)}}-\frac{x^{(2)}_s}{\lambda^{(2)}} \right).\nabla \e^{(2)},\psi_1 \right)_\rho\right| \lesssim e^{-\mu s}\left| \frac{x_s^{(1)}}{\lambda^{(1)}}-\frac{x_s^{(2)}}{\lambda^{(2)}} \right| .
$$
The modulation estimate \fref{modulationequationbound}, with \fref{controlunstable}, \fref{poitwiseboundhtwo} and \fref{controlsobolev} and an integration by parts yields
$$
\left|\left(\frac{\lambda_s^{(1)}}{\lambda^{(1)}}+1 \right)(\Lambda \triangle v,\psi_1)_\rho-\left( \frac{x_s^{(1)}}{\lambda^{(1)}}.\nabla \triangle \e,\psi_1\right)_\rho \right|\lesssim e^{-\mu s} \left(\| \triangle \e \|_{L^2_\rho}+\sum_{j=2}^{n+1}|\triangle a_j| \right) .
$$
Eventually, for the difference of the nonlinear terms the nonlinear inequality
$$
\left|(x+y)^p-(x+z)^p-px^{p-1}(y-z) \right|\lesssim (|x|^{p-2}+|y|^{p-2}+|z|^{p-2})(|y|+|z|)|y-z|
$$
for any $x,y,z$ and the bound \fref{poitwisebound} yields the pointwise estimate
\be\label{eq:truepointwisenonlinear}
\left| (\Phi_n+v^{(1)})^p -(\Phi_n+v^{(2)})^p-p\Phi_n^{p-1}\triangle v \right|\lesssim e^{-c\mu s} |\triangle v|,
\ee
which implies
\be \label{lip:pointwise nonlinear}
\left| \left( (\Phi_n+v^{(1)})^p -(\Phi_n+v^{(2)})^p-p\Phi_n^{p-1}\triangle v ,\psi_1 \right)_\rho\right|\lesssim e^{-c\mu s}  \left(\| \triangle \e \|_{L^2_\rho}+\sum_{j=2}^{n+1}|\triangle a_j|\right).
\ee
The collection of the above bounds, when plugged in \fref{lip:modulation expression}, yields
$$
\left| \frac{d}{ds}\left[\text{log}\left( \frac{\lambda^{(1)}}{\lambda^{(2)}} \right) \right] \right| \lesssim e^{-\mu s} \left| \frac{x_s^{(1)}}{\lambda^{(1)}}-\frac{x_s^{(2)}}{\lambda^{(2)}} \right| +e^{-c\mu s} \left(\| \triangle \e \|_{L^2_\rho}+\sum_{j=2}^{n+1}|\triangle a_j|\right).
$$
With the same techniques, taking the scalar product of \fref{lip:evolution} with $\partial^k \Phi_n$, $k=1,2,3$ implies
$$
\left| \frac{x_s^{(1)}}{\lambda^{(1)}}-\frac{x_s^{(2)}}{\lambda^{(2)}} \right| \lesssim e^{-\mu s} \left| \frac{d}{ds}\left[\text{log}\left( \frac{\lambda^{(1)}}{\lambda^{(2)}} \right) \right] \right|  +e^{-c\mu s} \left(\| \triangle \e \|_{L^2_\rho}+\sum_{j=2}^{n+1}|\triangle a_j|\right).
$$
The two above equations, when put together, imply the estimate
$$
\left| \frac{d}{ds}\left[\text{log}\left( \frac{\lambda^{(1)}}{\lambda^{(2)}} \right) \right] \right|+\left| \frac{x_s^{(1)}}{\lambda^{(1)}}-\frac{x_s^{(2)}}{\lambda^{(2)}} \right|\lesssim e^{-c\mu s} \left(\| \triangle \e \|_{L^2_\rho}+\sum_{j=2}^{n+1}|\triangle a_j|\right).
$$
The corresponding estimate for $|\triangle a_{j,s}+\mu_j \triangle a_j|$ follows along the same lines, and therefore \fref{lip:modulation} is proven.\\

\noindent \textbf{step 2} Localized energy estimate. We claim the differential bound
\be \label{lip:energie}
\frac{d}{ds} \| \triangle \e \|_{L^2_\rho}^2+c_n \| \triangle \e \|_{L^2_\rho}^2\lesssim e^{-c\mu s} \sum_{j=2}^{n+1}|\triangle a_j|^2
\ee
which we now prove. From the evolution equation \fref{lip:evolution} and the orthogonality conditions \fref{orhotpsin} one obtains first the identity
\be \label{lip:energie expression}
\begin{array}{r c l}
\frac{d}{ds} \frac 12 \| \triangle \e \|^2_{L^2_\rho}  &=& -(\mathcal L_n \triangle \e, \triangle \e)_\rho+\frac{d}{ds}\left[\text{log}\left( \frac{\lambda^{(1)}}{\lambda^{(2)}} \right) \right]( \Lambda v^{(2)},\triangle \e)_\rho \\
&&+\left(\left(\frac{x_s^{(1)}}{\lambda^{(1)}}-\frac{x^{(2)}_s}{\lambda^{(2)}} \right).\nabla v^{(2)},\triangle \e \right)_\rho  +\left(\frac{\lambda_s^{(1)}}{\lambda^{(1)}}+1 \right)(\Lambda \triangle v,\triangle \e)_\rho \\
&& +\left(\frac{x_s^{(1)}}{\lambda^{(1)}}.\nabla \triangle v,\triangle \e \right)_\rho+\left( (\Phi_n+v^{(1)})^p -(\Phi_n+v^{(2)})^p-p\Phi_n^{p-1}\triangle v,\triangle \e \right)_\rho 
\end{array}
\ee
and we now estimate each term. The spectral gap \fref{coerciviteapoids} and \fref{orhotpsin} imply
$$
-(\mathcal L_n \triangle \e,\triangle \e)_\rho\leq -c_n \| \triangle \e \|_{L^2_\rho}^2.
$$
The modulation estimates \fref{lip:modulation} of step 1 and Cauchy-Schwarz imply
\begin{eqnarray*}
&& \left| \frac{d}{ds}\left[\text{log}\left( \frac{\lambda^{(1)}}{\lambda^{(2)}} \right) \right]( \Lambda v^{(2)},\triangle \e)_\rho+ \left(\left(\frac{x_s^{(1)}}{\lambda^{(1)}}-\frac{x^{(2)}_s}{\lambda^{(2)}} \right).\nabla v^{(2)},\triangle \e \right)_\rho \right| \\
&\lesssim &  \left( \left| \frac{d}{ds} \text{log}\left( \frac{\lambda^{(1)}}{\lambda^{(2)}} \right) \right| \| \Lambda v^{(2)} \|_{L^2_\rho}+\left|\frac{x_s^{(1)}}{\lambda^{(1)}}-\frac{x^{(2)}_s}{\lambda^{(2)}} \right| \|\nabla v^{(2)} \|_{L^2_\rho}\right) \|\triangle \e \|_{L^2_\rho} \\
&\lesssim &  \| v^{(2)} \|_{H^2_\rho} e^{-c\mu s} \left(\| \triangle \e \|_{L^2_\rho}+\sum_{j=2}^{n+1}|\triangle a_j|\right) \|\triangle \e \|_{L^2_\rho} \\
&\lesssim &  e^{-(1+c)\mu s} \left(\| \triangle \e \|_{L^2_\rho}^2+\sum_{j=2}^{n+1}|\triangle a_j|^2\right)
\end{eqnarray*}
where we used \fref{weightedesimate}, \fref{controlunstable} and \fref{poitwiseboundhtwo} to control $v^{(2)}$. Using the modulation estimate \fref{modulationequationbound}, with \fref{controlunstable}, \fref{poitwiseboundhtwo} and \fref{controlsobolev} for $u^{(1)}$, integrating by parts and applying Cauchy-Schwarz and \fref{weightedesimate} yields
\begin{eqnarray*}
&& \left| \left(\frac{\lambda^{(1)}_s}{\lambda^{(1)}}+1 \right)(\Lambda \triangle v,\triangle \e)_\rho+\left(\frac{x_s^{(1)}}{\lambda^{(1)}}.\nabla \triangle v,\triangle \e \right)_\rho \right| \\
&\lesssim & \left( \left| \frac{\lambda^{(1)}_s}{\lambda^{(1)}}+1 \right|+\left| \frac{x_s^{(1)}}{\lambda^{(1)}}\right| \right)\left(|(\Lambda \triangle \psi,\triangle \e)_\rho| +|(\Lambda \triangle \e,\triangle \e)_\rho| +|(\nabla \triangle \psi,\triangle \e)_\rho| +|(\nabla \triangle \e,\triangle \e)_\rho| \right)\\
&\lesssim & e^{-2\mu s} \left(\sum_{j=2}^{n+1}|\triangle a_j|^2+\| \triangle \e\|_{H^1_\rho}^2 \right).
\end{eqnarray*}
Finally, the pointwise estimate \fref{eq:truepointwisenonlinear} and Cauchy-Schwarz imply for the nonlinear term
$$
\left|\left( (\Phi_n+v^{(1)})^p -(\Phi_n+v^{(2)})^p-p\Phi_n^{p-1}\triangle v,\triangle \e \right)_\rho\right| \lesssim e^{-c\mu s}\left(\| \triangle \e \|_{L^2_\rho}^2+\sum_{j=2}^{n+1}|\triangle a_j|^2\right).
$$
We inject all the above bounds in the identity \fref{lip:energie expression}, which for $s_0$ large enough imply the desired estimate \fref{lip:energie} since $0<c\leq 1$.\\

\textbf{step 3} Lipschitz bound by reintegration. We define
\be \label{lip:normes bootstrap}
A:= \underset{s\geq s_0}{\text{sup}} \ \sum_{j=2}^{n+1} |\triangle a_j(s)|e^{\mu s}<+\infty, \ \ \mathcal E:= \underset{s\geq s_0}{\text{sup}} \ \|\Delta\e\|_{L^2_\rho}^2e^{2\mu s}<+\infty,
\ee
which are finite from \fref{controlunstable} and \fref{poitwiseboundhtwo}.

\noindent \emph{Identity for $\triangle a_j$}. Fix $j$ with $2\leq j\leq n+1$. Reintegrating the modulation equation \fref{lip:modulation} yields
\begin{eqnarray}\label{lip:identite triangleaj}
\nonumber\triangle a_j & = & \triangle a_j(0)e^{\mu_j(s-s_0)}+e^{\mu_j s}\int_{s_0}^s e^{-\mu_js'}O(e^{-c\mu s'}(\| \Delta\e \|_{L^2_\rho}+\sum_{j=2}^{n+1}|\triangle a_j|))ds'\\
\nonumber&=& \triangle a_j(0)e^{\mu_j(s-s_0)}+e^{\mu_j s}\int_{s_0}^s O(e^{-(\mu_j+(c+1)\mu) s'}(A+\sqrt \mathcal E))ds' \\
\nonumber&=& \left(\triangle a_j(0)e^{-\mu_j s_0}+\int_{s_0}^{+\infty} O(e^{-(\mu_j+(c+1)\mu) s'}(A+\sqrt \mathcal E))ds'\right)e^{\mu_j s}\\
&&-e^{\mu_j s} \int_{s}^{+\infty} O(e^{-(\mu_j+(c+1)\mu) s'}(A+\sqrt \mathcal E))ds' .
\end{eqnarray}
The integral appearing in this identity is indeed convergent and satisfies:
$$
\left| \int_{s}^{+\infty} O(e^{-(\mu_j+(c+1)\mu) s'}(A+\sqrt \mathcal E))ds' \right|\lesssim e^{-(\mu_j+(c+1)\mu)s}(A+\sqrt \mathcal E).
$$
From \fref{lip:normes bootstrap} one gets $|\triangle a_j|\lesssim e^{-\mu s}$ and from the two above identities one necessarily must have that the parameter in front of the diverging term $e^{\mu_j s}$ is $0$:
$$
\triangle a_j(0)e^{-\mu_j s_0}+\int_{s_0}^{+\infty} O(e^{-(\mu_j+(c+1)\mu) s'}(A+\sqrt \mathcal E))ds'=0
$$
which gives the first bound
\be \label{lip:borne triangle aj0}
|\triangle a_j(0)|\lesssim e^{-(c+1)\mu s_0}(A+\sqrt \mathcal E),
\ee
and going back to the identity \fref{lip:identite triangleaj} one obtains:
$$
|\triangle a_j| \lesssim e^{-((c+1)\mu)s}(A+\sqrt \mathcal E)
$$
which implies from the definition \fref{lip:normes bootstrap} of $A$ the bound
\be \label{lip:borne intermedaire}
A\lesssim e^{-c\mu s_0}\sqrt \mathcal E .
\ee

\noindent \emph{Identity for $\triangle \e$}. We reintegrate the energy bound \fref{lip:energie} to find
$$
\begin{array}{r c l}
\| \triangle \e \|_{L^2_\rho}^2 & \lesssim & \| \triangle \e (0)\|_{L^2_\rho}^2e^{-c_n(s-s_0)}+e^{-c_ns}\int_{s_0}^s e^{c_ns'} \sum_{j=2}^{n+1} |\triangle a_j|^2e^{-\mu cs'}ds' \\
&\lesssim &  \| \triangle \e (0)\|_{L^2_\rho}^2e^{-c_n(s-s_0)}+A^2e^{-(c+2)\mu s}
\end{array}
$$
since $\mu =\frac{c_n}{4}$ from \fref{defmu} and $0<c\ll 1$ can be chosen arbitrarily small. Injecting \fref{lip:borne intermedaire} in the above identity yields
$$
\mathcal E \lesssim \| \triangle \e (0)\|_{L^2_\rho}^2e^{2\mu s_0}
$$
so that \fref{lip:borne intermedaire} can be rewritten as $A\lesssim \| \triangle \e (0)\|_{L^2_\rho}e^{(1-c)\mu s_0}$. We inject these two last bounds in \fref{lip:borne triangle aj0} which finally yields the desired estimate \fref{lip:lipschitz}.
\end{proof}

%%%%%%%%%%%%%%%%%%%%%%%%%%%%%%%%%%%%%
%%%%%%%%%%%%%%%%%%%%%%%%%%%%%%%%%%%%%
%%%%%%%%%%%%%%%%%%%%%%%%%%%%%%%%%%%%%

\begin{appendix}

%%%%%%%%%%%%%%%%%%%%%%%%%%%%%%%%%%%%%
%%%%%%%%%%%%%%%%%%%%%%%%%%%%%%%%%%%%%

%%%%%%%%%%%%%%%%%%%%%%%%%%%%%%%%%%%%%

\section{Coercivity estimates}
\label{appendcoerc}
%%%%%%%%%%%%%%%%%%%%%%%%%%%%%%%%%%%%%

\begin{lemma}[Weighted $L^2$ estimate]\label{L:HARMONICOSCILLATOR}
Let $u,\pa_r u\in L^2_{\rho}(\RR^3)$, then \be
\label{weightedesimate}
\|ru\|_{\rho}\lesssim\|u\|_{H^1_{\rho}}.
\ee
Moreover, 
\be
\label{estimationapoinds}
\|\Delta u\|_{L^2_\rho}^2\lesssim \|-\Delta u+y\cdot\nabla u\|_{L^2_{\rho}}^2+\|u\|_{H^1_\rho}^2.
\ee
\end{lemma}

\begin{proof}
We may assume by density $u\in \mathcal D(\Bbb R^3)$.\\

\noindent{\bf step 1} Proof of \eqref{weightedesimate}. We use $\pa_r\rho=-r\rho$ and integrate by parts to compute:
\bee
&& \int_0^{+\infty}\left(\pa_ru-\frac{1}{2} ru\right)^2\rho r^2dr\\
 &=&  \int_0^{+\infty} (\pa_ru)^2\rho r^2dr+\frac{1}{4}\int_0^{+\infty} r^{2}u^2\rho r^2dr - \int_0^{+\infty} ru\pa_r u\rho r^2 dr\\
& = & \int_0^{+\infty} (\pa_ru)^2\rho r^2dr+\frac{1}{4}\int_0^{+\infty} r^{2}u^2\rho r^2dr-\frac{1}{2}\left[r^3\rho u^2\right]_0^{+\infty}\\
&& +\frac{1}{2}\int_0^{+\infty} u^2(3-r^2)\rho r^2dr\\
& = & \int_0^{+\infty} (\pa_ru)^2\rho r^2dr-\frac{1}{4}\int_0^{+\infty} r^{2}u^2\rho r^2dr +\frac{3}{2}\int_0^{+\infty} u^2\rho r^2dr
\eee
and hence
\bee
\|ru\|_{L^2_\rho}^2=\int_0^{+\infty} r^{2}u^2\rho r^2dr \leq  4\int_0^{+\infty} (\pa_ru)^2\rho r^2dr +6\int_0^{+\infty} u^2\rho r^2dr\lesssim \|u\|^2_{H^1_\rho}
\eee 
which concludes the proof of \eqref{weightedesimate}.\\

\noindent{\bf step 2}. Proof of \eqref{estimationapoinds}. We compute:
$$\|-\Delta u +y\cdot\nabla u\|_{L^2_\rho}^2=\|\Delta u\|^2_{L^2_\rho}+\|y\cdot\nabla u\|_{L^2_\rho}^2-2\int(\Delta u) y\cdot\nabla u\rho dy.$$
To compute the crossed term, let $u_\l(y)=u(\l y),$ then $$\int|\nabla u_\l(y)|^2\rho dy=\frac{1}{\l}\int|\nabla u(y)|^2\rho\left(\frac{y}{\l}\right)dy$$ and hence differentiating in $\l$ and evaluating at $\l=1$:
$$2\int \nabla u\cdot\nabla(y\cdot\nabla u)\rho dy=\int|\nabla u|^2(-\rho-y\cdot\nabla \rho)dy$$
i.e. 
$$2\int y\cdot\nabla u(\rho\Delta u+\nabla u\cdot\nabla \rho)=\int|\nabla u|^2(\rho +y\cdot\nabla \rho)dy$$
which using $\nabla \rho=-y\rho$ becomes:
$$-2\int(\Delta u) y\cdot\nabla u\rho dy=\int|\nabla u|^2\rho |y|^2-2\int|y\cdot\nabla u|^2\rho-\int\rho |\nabla u|^2.$$ Hence:
\bee
\|-\Delta u +y\cdot\nabla u\|_{L^2_\rho}^2&=&\|\Delta u\|^2_{L^2_\rho}+\int \rho(|y|^2|\nabla u|^2-|y\cdot\nabla u|^2)-\int\rho |\nabla u|^2\\
&\geq & \|\Delta u\|^2_{L^2_\rho}-\|\nabla u\|_{L^2_\rho}^2
\eee
which concludes the proof of \eqref{estimationapoinds}.
\end{proof}

We now turn to the proof of Hardy type inequalities. All proofs are more or less standard and we give the argument for the sake of completeness. 

\begin{lemma}[Radial Hardy with best constants]
\label{ahrdyvevoev}
Let $u\in \mathcal C^\infty_c(r>1)$ and 
\be
\label{nondegene}
\gamma\neq -1,
\ee
then 
\be
\label{hardypoidsbis}
\int_1^{+\infty}\frac{(\pa_ru)^2}{r^{\gamma}}dr\geq \left(\frac{\gamma+1}{2}\right)^2\int_1^{+\infty}\frac{u^2}{r^{\gamma+2}}dr.
\ee
\end{lemma}

\begin{proof} We integrate by parts:
\bee
\int_1^{+\infty}\frac{u^2}{r^{\gamma+2}}dr=\frac{2}{\gamma+1}\int_1^{+\infty}\frac{u\pa_ru}{r^{\gamma+1}}dr\leq \frac{2}{|\gamma+1|}\left(\int_1^{+\infty}\frac{u^2}{r^{\gamma+2}}dr\right)^{\frac 12}\left(\int_1^{+\infty}\frac{(\pa_ru)^2}{r^{\gamma}}dr\right)^{\frac 12}
\eee
and \eqref{hardypoidsbis} follows.
\end{proof}

\begin{lemma}[Global Hardy for $\Delta$]
Then there exists $c>0$ such that $\forall u\in C^\infty_c(|x|> 1)$, 
\be
\label{estpoidsglobal}
\int|\Delta u|^2dx\geq c\int\left( \frac{|\nabla u|^2}{|x|^{2}}+\frac{|u|^2}{|x|^{4}}\right)dx.
\ee
\end{lemma}

\begin{proof} We decompose $u$ in spherical harmonics and consider $$\Delta_{m}u_m=\pa^2_ru_m+\frac2 r\pa_ru_m-\frac{m(m+1)}{r^2}, \ \ m\in \Bbb N.$$ We claim that for all $v\in \mathcal C^{\infty}_c((1,+\infty))$, 
\be
\label{keyestimate}
\int_1^{+\infty}|\Delta_m v|^2r^2dr\geq c\int_1^{+\infty}\left(\frac{|\pa_rv|^2}{r^{2}}+\frac{(1+m^4)|v|^2}{r^{4}}\right)r^2dr
\ee
with $c$ independent of $m$. Assume \eqref{keyestimate}, then 
$$\int\frac{ |\nabla u|^2}{r^{2}}dx\sim  \sum_{m\geq 0}\sum_{k=-m}^m\int \left(\frac{|\pa_ru_{m,k}|^2}{r^{2}}+\frac{m^2|u_{m,k}|^2}{r^{4}}\right)r^2dr$$ 
and hence summing \eqref{keyestimate} ensures \eqref{estpoidsglobal}.\\
To prove \eqref{keyestimate}, we factorize the Laplace operator:
$$\Delta_m=-A^*_mA_m \ \ \mbox{with}\ \ \left|\begin{array}{ll} A_m=-\pa_r-\frac{\gamma_m}{r}=-\frac{1}{r^{\gamma_m}}\pa_r(r^{\gamma_m}),\ \ \gamma_m=-m,\\
A^*_m=\pa_r+\frac{2-\gamma_m}{r}\pa_r=\frac{1}{r^{2-\gamma_m}}\pa_r(r^{2-\gamma_m}).\end{array}\right.$$
Hence from \eqref{hardypoidsbis}:
\bee
&&\int_1^{+\infty}(\Delta_mv)^2r^2dr= \int_1^{+\infty}(A^*_mA_mv)^2r^2dr=\int_1^{+\infty}\frac{1}{r^{2-2\gamma_m}}(\pa_r(r^{2-\gamma_m}A_mv))^2dr\\
& \geq & \left(\frac{2-2\gamma_m+1}{2}\right)^2\int_1^{+\infty}(A_mv)^2dr=\left(\frac{2-2\gamma_m+1}{2}\right)^2\int_1^{+\infty}\frac{1}{r^{2\gamma_m}}(\pa_r(r^{\gamma_m}v))^2dr\\
& \geq & \left(\frac{2-2\gamma_m+1}{2}\right)^2\left(\frac{2\gamma_m+1}{2}\right)^2\int_1^{+\infty}\frac{v^2}{r^2}dr
\eee
since $\gamma_m=-m$ with $m\in \Bbb N$ which ensures that the forbidden value \eqref{nondegene} is never attained. We conclude that for some universal constant $\delta>0$ independent of $m$:
$$\int_1^{+\infty}(\Delta_mv)^2r^2dr\geq \delta (1+m^4)\int_1^{+\infty}\frac{v^2}{r^4}r^2dr.$$ 
Also, since we have also proved that 
$$\int_1^{+\infty}|A_mv|^2dr \lesssim \int_1^{+\infty}(\Delta_mv)^2r^2dr,$$
we infer
\bee
\int_1^{+\infty}\frac{(\pa_rv)^2}{r^{2}}r^2dr &\lesssim& \int_1^{+\infty}|A_mv|^2dr + \gamma_m^2\int \frac{v^2}{r^4}r^2dr\\
&\lesssim & \int_1^{+\infty}(\Delta_mv)^2 r^2dr
\eee
and \eqref{keyestimate} follows.
\end{proof}

%%%%%%%%%%%%%%%%%%%%%%%%%%%%%%%%%%%%%%%%%%%%%%%%%%%%%%%%%%%%%%

\section{Proof of \eqref{noneoneoneevoe}}
\label{sec:appendixcommutatorestimate}

%%%%%%%%%%%%%%%%%%%%%%%%%%%%%%%%%%%%%%%%%%%%%%%%%%%%%%%%%%%%%%

Let $$0<\nu<1, \ \ 1<p_1,p_2, p_3, p_4<+\infty,\ \ \frac{1}{2}=\frac1{p_1}+\frac 1{p_2}=\frac1{p_3}+\frac1{p_4}.$$ 
Using \eqref{eq:standarddefinitionbesovnorms}, we have 
\bee
\|\nabla^\nu(uv)\|_{L^2} &\sim& \|uv\|_{\dot{B}^\nu_{2, 2}}\\
&\sim & \left(\int_0^{+\infty}\left(\frac{\sup_{|y|\leq t}\|uv(\cdot-y)-uv(\cdot)\|_{L^{2}}}{t^\nu}\right)^2\frac{dt}t\right)^{\frac 12}\\
&\lesssim & \left(\int_0^{+\infty}\left(\frac{\sup_{|y|\leq t}\|u(\cdot-y)(v(\cdot-y)-v(\cdot))\|_{L^{2}}}{t^\nu}\right)^2\frac{dt}t\right)^{\frac 12}\\
&&+\left(\int_0^{+\infty}\left(\frac{\sup_{|y|\leq t}\|v(\cdot)(u(\cdot-y)-u(\cdot))\|_{L^{2}}}{t^\nu}\right)^2\frac{dt}t\right)^{\frac 12}\\
&\lesssim & \|u\|_{L^{p_4}}\left(\int_0^{+\infty}\left(\frac{\sup_{|y|\leq t}\|v(\cdot-y)-v(\cdot)\|_{L^{p_3}}}{t^\nu}\right)^2\frac{dt}t\right)^{\frac 12}\\
&&+\|v\|_{L^{p_2}}\left(\int_0^{+\infty}\left(\frac{\sup_{|y|\leq t}\|u(\cdot-y)-u(\cdot)\|_{L^{p_1}}}{t^\nu}\right)^2\frac{dt}t\right)^{\frac 12}\\
&\lesssim& \|u\|_{\dot{B}^\nu_{p_1, 2}}\|v\|_{L^{p_2}}+\|u\|_{L^{p_4}}\|v\|_{\dot{B}^\nu_{p_3, 2}}
\eee
which concludes the proof of \eqref{noneoneoneevoe}.

%%%%%%%%%%%%%%%%%%%%%%%%%%%%%%%%%%%%%%%%%%%%%%%%%%%%%%%%%%%%%%

\section{Proof of Lemma \ref{lemma:asymptoticbehavioratinfinityforLmn}}
\label{appendixlemmaun}

%%%%%%%%%%%%%%%%%%%%%%%%%%%%%%%%%%%%%%%%%%%%%%%%%%%%%%%%%%%%%%

The existence and uniqueness of $\phi_{n,m},\nu_m$ satisfying \eqref{cneoneoenove} and \fref{beahviorunum} is well known. Thus, we focus on their behaviour as $r\to +\infty$.\\

\noindent{\bf step 1} Inverting $\mathcal L_{m,\infty}$. Let $\ga_m$ be the solution to
$$\gamma^2_{m}-\gamma_{m}+pc_{\infty}^{p-1}-m(m+1)=0,$$
 the corresponding discriminant $\Delta_m$ is given by
\bea\label{eq:definitionofdiscriminantDeltam}
\Delta_m:=1-4pc_\infty^{p-1}+4m(m+1).
\eea
For $m=1$, 
\bea\label{eq:explicitcompuationofDelta1}
\Delta_1=\left(\frac{p+3}{p-1}\right)^2>0
\eea
and hence for all $m\geq 1$
$$\Delta_m\geq\Delta_1>0.$$
Therefore, $\gamma_m$ is real and we choose the smallest root\footnote{This is motivated by the fact that we obtain below the Kummer's equation with $b=-\ga_m+1/2$. This is equivalent to $-b=\pm \sqrt{\Delta_m}$. Since the Kummer function is not defined for $-b\in \mathbb{N}$, this justifies to consider the smallest root $\ga_m$.} so that $\ga_m$ is given by
$$\gamma_m=\frac{1 - \sqrt{\Delta_m}}{2}.$$ We now solve
$$\mathcal{L}_{\infty,m}(\psi)=0$$ through the change of variable and unknown
$$\psi(r)=\frac{1}{(2z)^{\frac{\ga_m}{2}}}w(z),\,\,\,\, z=\frac{r^2}{2}$$
which leads to
\bee
\mathcal{L}_{\infty, m}(\psi) &=& -\frac{2}{(2z)^{\frac{\ga}{2}}}\left(zw''(z)+\left(-\ga_m+\frac{3}{2}-z\right)w'(z)-\left(\frac{1}{p-1}-\frac{\ga_m}{2}\right)w(z)\right).
\eee
Thus, $\mathcal{L}_{\infty, m}(\psi)=0$ if and only if
$$
z\frac{d^2w}{dz^2}+(b-z)\frac{dw}{dz}-aw=0
$$
where we have used the notations
\bee
a=\frac{1}{p-1}-\frac{\ga_m}{2},\,\,\,\, b=-\ga_m+\frac{3}{2}.
\eee
Hence $w$ is a linear combination of two special functions, the Kummer's function $M(a,b,z)$ and the Tricomi function $U(a,b,z)$. These special functions have the following asymptotic behavior at infinity (see for example \cite{Handbookmathfunctions}):
\bee
M(a,b,z)\sim \frac{\Gamma(b)}{\Gamma(a)}z^{a-b}e^z,\,\,\,\, U(a,b,z)\sim z^{-a}\textrm{ as }z\to +\infty.
\eee
This allows us to infer the asymptotic for $w$ for $z\to 0_+$. Finally, since
$$\psi(r)=\frac{1}{r^{\ga_m}}w\left(\frac{r^2}{2}\right),$$
we infer from the asymptotic of $w$ the following asymptotic behavior for $\psi_{1,m}$ and $\psi_{2,m}$
$$\psi_{1,m}\sim \frac{1}{r^{\frac{2}{p-1}}}\textrm{ and }\psi_{2,m}\sim  r^{\frac{2}{p-1}-3}e^{\frac{r^2}{2}}\textrm{ as } r\to+\infty.$$

Consider the Wronskian $W$ which is defined as
\bee
W &:=& \psi_{1,m}'\psi_{2,m}-\psi_{2,m}'\psi_{1,m},
\eee
then without loss of generality since $W'=\left(r-\frac 2r \right)W$
$$W=\frac{1}{r^{2}}e^{\frac{r^2}{2}}.$$
We  deduce using the variation of constants that the solution $w$ to 
$$\mathcal{L}_{\infty,m}(u) = f,$$
is given by
\bee
u &=& \left(a_1+\int_r^{+\infty} f\psi_{2,m}{r'}^{2}e^{-\frac{{r'}^2}{2}}dr'\right)\psi_{1,m}+ \left(a_2-\int_r^{+\infty} f\psi_{1,m}{r'}^{2}e^{-\frac{{r'}^2}{2}}dr'\right)\psi_{2,m}.
\eee

\noindent{\bf step 2} Basis of $\mathcal L_{m,n}$ near $+\infty$. We now construct a solution to $\mathcal{L}_{n,m}(\varphi)=0$ near $+\infty$ by solving:
\bee
\mathcal{L}_{\infty, m}(\varphi) &=& \mathcal{L}_{n,m}(\varphi) +p(\Phi_n^{p-1}-\Phi_*^{p-1})= p(\Phi_n^{p-1}-\Phi_*^{p-1})\varphi
\eee
ie
\bee
\varphi &=& \left(a_1+\int_r^{+\infty} p(\Phi_n^{p-1}-\Phi_*^{p-1})\varphi\psi_{2,m}{r'}^{2}e^{-\frac{{r'}^2}{2}}dr'\right)\psi_{1,m}\\
&&+ \left(a_2-\int_r^{+\infty} p(\Phi_n^{p-1}-\Phi_*^{p-1})\varphi\psi_{1,m}{r'}^{2}e^{-\frac{{r'}^2}{2}}dr'\right)\psi_{2,m}.
\eee
To construct the solution $\varphi_1$ with the choice $a_1=1$ and $a_2=0$ we solve the fixed point equation
\bea\label{eq:easylinearfixedpointbis}
\varphi_1=\psi_{1,m}+\widetilde{\varphi}_1,\,\,\,\,\widetilde{\varphi}_1=\mathcal{G}\left(\widetilde{\varphi}_1\right)
\eea
where
\bee
\mathcal{G}\left(\widetilde{\varphi} \right)(r) &=& \left(\int_r^{+\infty}p(\Phi_n^{p-1}-\Phi_*^{p-1})\left(\psi_{1,m}+\widetilde{\varphi} \right)(r')\psi_{2,m}{r'}^2e^{-\frac{{r'}^2}{2}}dr'\right)\psi_{1,m}\\
&&-\left(\int_r^{+\infty}p(\Phi_n^{p-1}-\Phi_*^{p-1})\left(\psi_{1,m}+\widetilde{\varphi} \right)(r')\psi_{1,m}{r'}^2e^{-\frac{{r'}^2}{2}}dr'\right)\psi_{2,m}.
\eee
Recall that we have in view of Corollary \ref{cor:consequenceforPhinofproponexistence} 
$$\lim_{n\to +\infty}\sup_{r\geq 1}r^{\frac{2}{p-1}}|\Phi_n(r) - \Phi_*(r)| = 0.$$
Thus, for $n\geq N$ large enough, we infer
$$|\Phi_n(r) - \Phi_*(r)|\leq \frac{1}{r^{\frac{2}{p-1}}}\textrm{ for }r\geq 1.$$
so that
$$|p(\Phi_n^{p-1}-\Phi_*^{p-1})|\lesssim \frac{1}{r^2}.$$
We infer for $r\geq 1$
\bee
\left|\mathcal{G}\left(\widetilde{\varphi}\right)(r)\right| &\lesssim& \frac{1}{r^{\frac{2}{p-1}}}\left(\int_r^{+\infty}{r'}^{\frac{2}{p-1}-3}\left(\frac{1}{{r'}^{\frac{2}{p-1}}}+|\widetilde{\varphi}(r')|\right)dr'\right)\\
&+&r^{\frac{2}{p-1}-3}e^{\frac{r^2}{2}}\left(\int_r^{+\infty}\frac{1}{{r'}^{\frac{2}{p-1}}}e^{-\frac{{r'}^2}{2}}\left(\frac{1}{{r'}^{\frac{2}{p-1}}}+|\widetilde{\varphi}(r')|\right)dr'\right)\\
&\lesssim& \frac{1}{r^{2+\frac{2}{p-1}}} + \frac{1}{r^{\frac{2}{p-1}}}\left(\int_r^{+\infty}{r'}^{\frac{2}{p-1}-3}|\widetilde{\varphi}(r')|dr'\right)\\
& +&r^{\frac{2}{p-1}-3}e^{\frac{r^2}{2}}\left(\int_r^{+\infty}\frac{1}{{r'}^{\frac{2}{p-1}}}e^{-\frac{{r'}^2}{2}}|\widetilde{\varphi}(r')|dr'\right)
\eee
and
\bee
\left|\mathcal{G}\left(\widetilde{\varphi}_{(1)}\right)(r)-\mathcal{G}\left(\widetilde{\varphi}_{(2)}\right)(r)\right| &\lesssim&  \frac{1}{r^{\frac{2}{p-1}}}\left(\int_r^{+\infty}{r'}^{\frac{2}{p-1}-3}|\widetilde{\varphi}_{(1)}(r')-\widetilde{\varphi}_{(2)}(r')|dr'\right)\\
&+& r^{\frac{2}{p-1}-3}e^{\frac{r^2}{2}}\left(\int_r^{+\infty}\frac{1}{{r'}^{\frac{2}{p-1}}}e^{-\frac{{r'}^2}{2}}|\widetilde{\varphi}_{(1)}(r')-\widetilde{\varphi}_{(2)}(r')|dr'\right)
\eee
 Thus, for $R\geq 1$ large enough, the Banach fixed point theorem applies in the space corresponding to the norm
 $$\sup_{r\geq R}r^{1+\frac{2}{p-1}}\left|\widetilde{\varphi}\right|(r).$$
 Hence, there exists a unique solution $\widetilde{\varphi}_1$ to \eqref{eq:easylinearfixedpointbis} and 
 $$\sup_{r\geq R}r^{1+\frac{2}{p-1}}\left|\widetilde{\varphi}_1\right|(r) \lesssim 1.$$
Hence, $\varphi_1$ satisfies $\mathcal{L}_{n,m}(\varphi_1)=0$ and
$$\varphi_1\sim \frac{1}{r^{\frac{2}{p-1}}},\textrm{ as }r\to +\infty.$$
 The behaviour of the other solution at infinity is computed using the Wronskian relation
$$
W = \varphi_1'\varphi_2-\varphi_2'\varphi_1=-\frac{1}{r^{2}}e^{\frac{r^2}{2}}$$
and hence
\bee
\left(\frac{\varphi_2}{\varphi_1}\right)' =  -\frac{W}{\varphi_1^2}= \frac{1}{r^{2}\varphi_1^2}e^{\frac{r^2}{2}}
\eee
from which
$$\varphi_2(r)=\varphi_1(r)\int_1^r\frac{1}{{r'}^{2}\varphi_1^2(r')}e^{\frac{{r'}^2}{2}}dr'\sim r^{\frac{2}{p-1}-3}e^{\frac{r^2}{2}}\textrm{ as } r\to+\infty$$
and \eqref{beahvriour} is proved.\\

\noindent{\bf step 3} Behaviour of $\nu_m$ at $+\infty$. First, consider the solution $\varphi$ to
\bea\label{eq:equationasymptotictoHm}
-\pr^2_r\varphi - \frac{2}{r}\pr_r\varphi +\frac{m(m+1)}{r^2}-\frac{pc_\infty^{p-1}}{r^2}\varphi=f.
\eea
The homogeneous equation admits the basis of solutions
$$\varphi_+=\frac{1}{r^{\frac{1+\sqrt{\Delta_m}}{2}}},\,\,\,\, \varphi_-=\frac{1}{r^{\frac{1-\sqrt{\Delta_m}}{2}}}$$
and the corresponding Wronskian is given by
$$W(r)=\varphi_+'(r)\varphi_-(r)-\varphi_-'(r)\varphi_+(r)=-\frac{1}{r^2}.$$
Using the variation of constants, the solutions to \eqref{eq:equationasymptotictoHm} are given by
$$\varphi(r)=\left(a_1-\int_r^{+\infty}f\varphi_-{r'}^2dr'\right)\varphi_++\left(a_2+\int_r^{+\infty}f\varphi_+{r'}^2dr'\right)\varphi_-.$$
Now, the equation $H_m(\phi)=0$ can be written as
$$-\pr^2_r\phi - \frac{2}{r}\pr_r\phi +\frac{m(m+1)}{r^2}\phi-\frac{pc_\infty^{p-1}}{r^2}\phi = p\left(Q^{p-1}(r)-\frac{c_\infty^{p-1}}{r^2}\right)\phi(r),$$
i.e. \eqref{eq:equationasymptotictoHm} with 
$$f=p\left(Q^{p-1}(r)-\frac{c_\infty^{p-1}}{r^2}\right)\phi(r).$$
We construct the solution $\phi_{m,1}$ to $H_m(\phi_{m,1})=0$ with the choice $a_1=1$ and $a_2=0$ by solving the fixed point equation
\bea\label{eq:easylinearfixedpointter}
\phi_{m,1}=\varphi_++\widetilde{\phi},\,\,\,\,\widetilde{\phi}=\mathcal{F}\left(\widetilde{\phi}\right)
\eea
where
\bee
\mathcal{F}\left(\widetilde{\phi}\right)(r) &=& -\left(\int_r^{+\infty}p\left(Q^{p-1}(r')-\frac{c_\infty^{p-1}}{{r'}^2}\right)\left(\varphi_++\widetilde{\phi}\right)(r')\varphi_-{r'}^2dr'\right)\varphi_+\\
&&+\left(\int_r^{+\infty}p\left(Q^{p-1}(r')-\frac{c_\infty^{p-1}}{{r'}^2}\right)\left(\varphi_++\widetilde{\phi}\right)(r')\varphi_+{r'}^2dr'\right)\varphi_-.
\eee
Recall that
\bee
Q(r)=\frac{c_{\infty}}{r^{\frac{2}{p-1}}}+\frac{c_1\sin\left(\om\log(r)+c_2\right)}{r^{\frac{1}{2}}}+o\left(\frac{1}{r^{\frac{1}{2}}}\right)\textrm{ as }r\to +\infty
\eee
so that
\bee
\left|p\left(Q^{p-1}(r)-\frac{c_\infty^{p-1}}{{r}^2}\right)\right|\lesssim \frac{1}{r^{1+s_c}}\textrm{ for }r\geq 1.
\eee
We infer for $r\geq 1$
\bee
\left|\mathcal{F}(\widetilde{\phi})(r)\right| &\lesssim& \frac{1}{r^{\frac{1+\sqrt{\Delta_m}}{2}}}\left(\int_r^{+\infty} \frac{1}{{r'}^{s_c-1}}\left(\frac{1}{r^{\frac{1+\sqrt{\Delta_m}}{2}}}+\left|\widetilde{\phi}\right|(r')\right)\frac{1}{{r'}^{\frac{1-\sqrt{\Delta_m}}{2}}}dr'\right)\\
&& +\frac{1}{r^{\frac{1-\sqrt{\Delta_m}}{2}}}\left(\int_r^{+\infty} \frac{1}{{r'}^{s_c-1}}\left(\frac{1}{r^{\frac{1+\sqrt{\Delta_m}}{2}}}+\left|\widetilde{\phi}\right|(r')\right)\frac{1}{{r'}^{\frac{1+\sqrt{\Delta_m}}{2}}}dr'\right)\\
&\lesssim& \frac{1}{r^{s_c-1}}\frac{1}{r^{\frac{1+\sqrt{\Delta_m}}{2}}} + \frac{1}{r^{\frac{1+\sqrt{\Delta_m}}{2}}}\left(\int_r^{+\infty} \frac{1}{{r'}^{s_c-1}}\frac{1}{{r'}^{\frac{1-\sqrt{\Delta_m}}{2}}}\left|\widetilde{\phi}\right|(r')dr'\right)\\
&& +\frac{1}{r^{\frac{1-\sqrt{\Delta_m}}{2}}}\left(\int_r^{+\infty} \frac{1}{{r'}^{s_c-1}}\frac{1}{{r'}^{\frac{1+\sqrt{\Delta_m}}{2}}}\left|\widetilde{\phi}\right|(r')dr'\right)
\eee
and
\bee
\left|\mathcal{F}(\widetilde{\phi}_1)(r)-\mathcal{F}(\widetilde{\phi}_2)(r)\right|  &\lesssim&  \frac{1}{r^{\frac{1+\sqrt{\Delta_m}}{2}}}\left(\int_r^{+\infty} \frac{1}{{r'}^{s_c-1}}\frac{1}{{r'}^{\frac{1-\sqrt{\Delta_m}}{2}}}\left|\widetilde{\phi}_1-\widetilde{\phi}_2\right|(r')dr'\right)\\
&& +\frac{1}{r^{\frac{1-\sqrt{\Delta_m}}{2}}}\left(\int_r^{+\infty} \frac{1}{{r'}^{s_c-1}}\frac{1}{{r'}^{\frac{1+\sqrt{\Delta_m}}{2}}}\left|\widetilde{\phi}_1-\widetilde{\phi}_2\right|(r')dr'\right).
\eee
 Thus, for $R\geq 1$ large enough, the Banach fixed point theorem applies in the space corresponding to the norm
 $$\sup_{r\geq R}r^{\frac{s_c-1}{2}}r^{\frac{1+\sqrt{\Delta_m}}{2}}|\widetilde{\phi}|(r)$$
 and yields a unique solution $\widetilde{\phi}$ to \eqref{eq:easylinearfixedpointter} with
 $$\sup_{r\geq R}r^{\frac{s_c-1}{2}}r^{\frac{1+\sqrt{\Delta_m}}{2}}|\widetilde{\phi}|(r)\leq 1.$$
 Hence, $\phi_{m,1}$ satisfies $H_m(\phi_{m,1})=0$ and
\be
\label{uohfeihfeohe}
\phi_{m,1}\sim \frac{1}{r^{\frac{1+\sqrt{\Delta_m}}{2}}},\textrm{ as }r\to +\infty.
\ee
 The other independent solution $\phi_{m,2}$ to  $H_m(\phi_{m,2})=0$ is computed through the Wronskian relation
\bee
W &:=& \phi_{m,1}'\phi_{m,2}-\phi_{m,2}'\phi_{m,1}= - \frac 1{r^2}
\eee
ie $$\phi_{m,2}(r)=\phi_{m,1}(r)\int_1^r\frac{1}{{r'}^{2}\phi_{m,1}^2(r')}dr'\sim \frac{1}{r^{\frac{1-\sqrt{\Delta_m}}{2}}}\textrm{ as } r\to+\infty.$$
Since $\nu_m$ is a linear combination of $\phi_{m,1}$ and $\phi_{m,2}$, we infer 
\be
\label{vnknveonoen}
\nu_m(r)\sim \frac{c_{m,+}}{r^{\frac{1+\sqrt{\Delta_m}}{2}}}+ \frac{c_{m,-}}{r^{\frac{1-\sqrt{\Delta_m}}{2}}}\textrm{ as }r\to +\infty
\ee
for some constant $c_{m,+}$ and $c_{m,-}$.\\
\noindent{\em case $m=1$}: By translation invariance
\be
\label{honezero}
H_1(Q')=0\textrm{ and }Q'(r) = Q''(0)r(1+O(r^2))
\ee
Hence, by uniqueness of $\nu_1$, we infer
$$\nu_1(r)=\frac{Q'(r)}{Q''(0)}<0\ \ \mbox{on}\ \ (0,+\infty)$$
where we used from standard ODE arguments $Q''(0)<0$ and 
\be
\label{qdecay}
Q'<0\ \ \mbox{on}\ \ (0,+\infty).
\ee
\noindent{\em case $m=2$}: From \eqref{honezero}, \eqref{qdecay} and standard Sturm Liouville oscillation arguments for central potentials \cite{RS}, the quadratic form $(H_1u,u)$ is positive on $\dot{H}^1_{\rm rad}(0,+\infty)$ and hence for $m\ge 2$, $H_m>H_1$ is definite positive, and hence $\nu_m>0$ on $(0,+\infty)$. Moreover, If $c_{m,-}=0$ in \eqref{vnknveonoen}, then $\nu_m\in \dot{H}^1_{\rm rad}$ satisfies $(H_m\nu_m,\nu_m)=0$ which is a contradiction, hence the leading order behaviour \eqref{beahviorunum}.\\

\noindent{\bf step 4} Completing the basis.\\
\noindent{\em case $m=2$}. Let $\phi_m$ be the solution to $H_m(\phi_m)=0$ constructed above with the behaviour \eqref{uohfeihfeohe}. At the origin, the equation $H_m\psi$ reads $$A^*_mA_m\psi=V\psi,$$ with $$A_m v=r^m\pa_r\left(\frac{v}{r^m}\right), \ \ A^*_m=\frac{v}{r^{m+1}}\pa_r(r^{m+1}v)$$ and $V\in L^{\infty}$ and hence all solutions on $(0,\delta)$ with $0<\delta\ll1 $ are of the form $$\psi=c_0r^m+\frac{c_1}{r^{m+1}}+r^m\int_r^{\delta}\frac{d\tau}{\tau^{2m+1}}\int_0^r \tau^{m+1} V\psi d\tau$$ through an elementary fixed point argument. Hence 
\bea\label{eq:auxbehaviorphimmgeq2}
\phi_m=\frac{c_1+O(r^2)}{r^{m+1}}.
\eea 
Assume by contradiction that $c_1=0$. Then, the fixed point above leads to $\phi_m=O(r^m)$. Hence $\phi_m$ is a zero of $H_m$ in $\dot{H^1}_{\rm rad}$ which is a contradiction. Thus, $c_1\neq 0$ and together with \eqref{eq:auxbehaviorphimmgeq2}, we have obtained \eqref{havviournuonebis}.\\

\noindent{\em case $m=1$}. We let $\phi_1$ be given by the Wronskian relation $$\phi_1=\nu_1(r)\int_r^1\frac{d\tau}{\tau^2\nu^2_1(\tau)}d\tau\sim \left|\begin{array}{ll}\frac{c}{r^2}\ \ \mbox{as}\ \ r\to 0, \ \ c\neq 0,\\ \frac{1}{r^{\frac{1-\sqrt{\Delta_1}}{2}}}\ \ \mbox{as}\ \ r\to+\infty,\end{array}\right.$$
which is \eqref{havviournuone}.\\

\noindent{\bf step 5} Proof of \eqref{vnvnonvnornor}. Let 
$$\kappa_{n,m}:=\mu_n^{-m}\varphi_{n,m}(\mu_nr).$$
Then, since $\varphi_{n,m}$ satisfies $\mathcal{L}_{n,m}(\varphi_{n,m})=0$, we infer
$$-\pr^2_r\kappa_{n,m}-\frac{2}{r}\pr_r\kappa_{n,m}+\frac{m(m+1)}{r^2}\kappa_{n,m} -p\left(\mu_n^{\frac{2}{p-1}}\Phi_n(\mu_nr)\right)^{p-1}\kappa_{n,m}=-\mu_n^2\Lambda\kappa_{n,m} .$$
This yields
\bee
H_m(\kappa_{n,m}) &=& f_{n,m}:= p\left(\left(\mu_n^{\frac{2}{p-1}}\Phi_n(\mu_nr)\right)^{p-1}-Q^{p-1}(r)\right)\kappa_{n,m} -\mu_n^2\Lambda\kappa_{n,m}.
\eee
Since $H_m(\nu_m)=0$, we infer 
\bee
H_m\left(\kappa_{n,m}- \nu_m\right) &=& f_{n,m}.
\eee
We let  $(\nu_m,\phi_m)$ be the completed fundamental basis for $H_m$ so that 
$$\kappa_{n,m}- \nu_m=\left(a_1-\int_0^rf_{n,m}\phi_m{r'}^2dr'\right)\nu_m+\left(a_2+\int_0^rf_{n,m}\nu_m{r'}^2dr'\right)\phi_m.$$
Since 
$$\nu_m(r)= r^m(1+O(r^2))\textrm{ and }\varphi_{n,m}(r)= r^m(1+O(r^2))\textrm{ as }r\to 0_+,$$
we infer
$$\kappa_{n,m}(r)-\nu(r)= O(r^{m+2})$$
and hence \eqref{havviournuone}, \eqref{havviournuonebis} implies $a_1=a_2=0$ and:
$$\kappa_{n,m}- \nu_m=-\left(\int_0^rf_{n,m}\phi_m{r'}^2dr'\right)\nu_m+\left(\int_0^rf_{n,m}\nu_m{r'}^2dr'\right)\phi_m.$$

In order to estimate $f_{n,m}$, recall from Corollary \ref{cor:consequenceforPhinofproponexistence} that we have
\bee
\sup_{r\leq r_0}\left|\Phi_n(r) -  \frac{1}{\mu_n^{\frac{2}{p-1}}}Q\left(\frac{r}{\mu_n}\right)\right| \lesssim \mu_n^{s_c-1}
\eee
This yields
\bea\label{eq:usefulestimatetocontrolfnm}
\sup_{r\leq \frac{r_0}{\mu_n}}\left|p\left(\left(\mu_n^{\frac{2}{p-1}}\Phi_n(\mu_nr)\right)^{p-1}-Q^{p-1}(r)\right)\right| \lesssim \frac{\mu_n^{s_c+1}}{r_0^{2-\frac{2}{p-1}}}. 
\eea
Also, we rewrite $f_{n,m}$ as
\bea\label{eq:usefulestimatetocontrolfnmbis}
f_{n,m} &=& p\left(\left(\mu_n^{\frac{2}{p-1}}\Phi_n(\mu_nr)\right)^{p-1}-Q^{p-1}(r)\right)\nu_m -\mu_n^2\Lambda\nu_m\\
\nonumber&&+p\left(\left(\mu_n^{\frac{2}{p-1}}\Phi_n(\mu_nr)\right)^{p-1}-Q^{p-1}(r)\right)(\kappa_{n,m}-\nu_m) -\mu_n^2\Lambda(\kappa_{n,m}-\nu_m).
\eea
$0\leq r\leq 1$. In view of the asymptotic behavior as $r\to 0_+$ \eqref{havviournuone}, \eqref{havviournuonebis} of the basis of solutions $\nu_m,\phi_m$, and after integrating by parts the term $\Lambda(\kappa_{n,m}-\nu_m)$, we have for $0\leq r\leq 1$ using \eqref{eq:usefulestimatetocontrolfnm} and \eqref{eq:usefulestimatetocontrolfnmbis}: \bee
|\kappa_{n,m}- \nu_m|(r)&\lesssim& \mu_n^2r^2|\kappa_{n,m}- \nu_m|(r)\\
&+&\left(\frac{\mu_n^{s_c+1}}{r_0^{2-\frac{2}{p-1}}}+ \mu_n^2\right)\Bigg(r^{m+2}+r^m\left(\int_0^r|\kappa_{n,m}- \nu_m|{r'}^{1-m}dr'\right)\\
&+&r^{-m-1}\left(\int_0^r|\kappa_{n,m}- \nu_m|{r'}^{m+2}dr'\right)\Bigg).
\eee
Using again  the asymptotic behavior of $\nu_m$ as $r\to 0_+$, we infer for all $m\geq 1$
\bea\label{eq:controlofthedifferencekappanmnumonrleq1}
\sup_{0\leq r\leq 1}\frac{|(\kappa_{n,m}- \nu_m)(r)|}{|\nu_m(r)|} &\lesssim& \frac{\mu_n^{s_c+1}}{r_0^{2-\frac{2}{p-1}}}+ \mu_n^2.
\eea
In particular, this yields
\bea\label{eq:controloftheintegraloffnmon0to1}
\int_0^1|f_{n,m}|{r'}^{1-m}dr'+\int_0^1|f_{n,m}|{r'}^{m+2}dr' &\lesssim&  \frac{\mu_n^{s_c+1}}{r_0^{2-\frac{2}{p-1}}}+ \mu_n^2.
\eea

Next, we consider the region $r\geq 1$. In view of the asymptotic behavior at infinity \eqref{havviournuone}, \eqref{havviournuonebis}, \eqref{beahviorunumbis}, \eqref{beahviorunum}, after integrating by parts the term $\Lambda(\kappa_{n,m}-\nu_m)$ and using also \eqref{eq:controloftheintegraloffnmon0to1}, we have
\bee
|\kappa_{n,m}- \nu_m|&\lesssim& \mu_n^2r^2|\kappa_{n,m}- \nu_m|\\
& + &  \frac{1}{(1+r)^{\frac{1+\sqrt{\Delta_m}}{2}}}\left(\frac{\mu_n^{s_c+1}}{r_0^{2-\frac{2}{p-1}}}+ \mu_n^2+\int_1^r|f_{n,m}|\frac{{r'}^2}{(1+r')^{\frac{1-\sqrt{\Delta_m}}{2}}}dr'\right)\\
&&+\frac{1}{(1+r)^{\frac{1-\sqrt{\Delta_m}}{2}}}\left(\frac{\mu_n^{s_c+1}}{r_0^{2-\frac{2}{p-1}}}+ \mu_n^2+\int_1^r|f_{n,m}|\frac{{r'}^2}{(1+r')^{\frac{1+\sqrt{\Delta_m}}{2}}}dr'\right).
\eee
After integrating by parts the term $\Lambda(\kappa_{n,m}-\nu_m)$, and in view of the asymptotic behavior of $\nu_m$ as $r\to +\infty$ as well as \eqref{eq:usefulestimatetocontrolfnm}, we deduce
\bee
&& |(\kappa_{n,m}- \nu_m)(r)|\\
&\lesssim& \frac{1}{(1+r)^{\frac{1+\sqrt{\Delta_m}}{2}}}\Bigg(\int_1^r\left(\frac{\mu_n^{s_c+1}}{r_0^{2-\frac{2}{p-1}}}|\nu_m| + \mu_n^2|\Lambda\nu_m|+\left(\frac{\mu_n^{s_c+1}}{r_0^{2-\frac{2}{p-1}}}+ \mu_n^2\right)|\kappa_{n,m}-\nu_m|\right)\\
&&\times\frac{{r'}^2}{(1+r')^{\frac{1-\sqrt{\Delta_m}}{2}}}dr' +\frac{\mu_n^{s_c+1}}{r_0^{2-\frac{2}{p-1}}}+ \mu_n^2\Bigg)\\
&&+\frac{1}{(1+r)^{\frac{1-\sqrt{\Delta_m}}{2}}}\Bigg(\int_1^r\left(\frac{\mu_n^{s_c+1}}{r_0^{2-\frac{2}{p-1}}}|\nu_m| + \mu_n^2|\Lambda\nu_m|+\left(\frac{\mu_n^{s_c+1}}{r_0^{2-\frac{2}{p-1}}}+ \mu_n^2\right)|\kappa_{n,m}-\nu_m|\right)\\
&&\times\frac{{r'}^2}{(1+r')^{\frac{1+\sqrt{\Delta_m}}{2}}}dr' +\frac{\mu_n^{s_c+1}}{r_0^{2-\frac{2}{p-1}}}+ \mu_n^2\Bigg).
\eee
\noindent{\em case $m\geq 2$}: We estimate from \eqref{beahviorunum}:
\bee
&& \frac{|(\kappa_{n,m}- \nu_m)(r)|}{|\nu_m(r)|}\\
&\lesssim& \frac{1}{(1+r)^{\sqrt{\Delta_m}}}\Bigg\{\int_1^r\left(\frac{\mu_n^{s_c+1}}{r_0^{2-\frac{2}{p-1}}}+ \mu_n^2\right)\left(\frac{1}{(1+r')^{\frac{1-\sqrt{\Delta_m}}{2}}}+|\kappa_{n,m}-\nu_m|\right)\\
&&\times\frac{{r'}^2}{(1+r')^{\frac{1-\sqrt{\Delta_m}}{2}}}dr'  +\frac{\mu_n^{s_c+1}}{r_0^{2-\frac{2}{p-1}}}+ \mu_n^2\Bigg\}\\
&&+\int_1^r\left(\frac{\mu_n^{s_c+1}}{r_0^{2-\frac{2}{p-1}}}+ \mu_n^2\right)\left(\frac{1}{(1+r')^{\frac{1-\sqrt{\Delta_m}}{2}}}+|\kappa_{n,m}-\nu_m|\right)\frac{{r'}^2}{(1+r')^{\frac{1+\sqrt{\Delta_m}}{2}}}dr'\\
&& +\frac{\mu_n^{s_c+1}}{r_0^{2-\frac{2}{p-1}}}+ \mu_n^2.
\eee
This yields
\bee
\sup_{1\leq r\leq\frac{r_0}{\mu_n}}\frac{|(\kappa_{n,m}- \nu_m)(r)|}{|\nu_m(r)|} &\lesssim& r_0^2\left(1+\frac{\mu_n^{s_c-1}}{r_0^{2-\frac{2}{p-1}}}\right)
\eee
which together with \eqref{eq:controlofthedifferencekappanmnumonrleq1} concludes the proof of \eqref{vnvnonvnornor} for $n\geq N$ large enough and $m\geq 2$.\\
\noindent{\em case $m=1$} We estimate using \eqref{beahviorunumbis}, \eqref{havviournuone}:
\bee
&& \frac{|(\kappa_{n,1}- \nu_1)(r)|}{|\nu_1(r)|}\\
&\lesssim& \int_1^r\left(\frac{\mu_n^{s_c+1}}{r_0^{2-\frac{2}{p-1}}}+ \mu_n^2\right)\left(\frac{1}{(1+r')^{\frac{1+\sqrt{\Delta_1}}{2}}}+|\kappa_{n,1}-\nu_1|\right)\frac{{r'}^2}{(1+r')^{\frac{1-\sqrt{\Delta_1}}{2}}}dr'\\
&& +\frac{\mu_n^{s_c+1}}{r_0^{2-\frac{2}{p-1}}}+ \mu_n^2\\
&&+(1+r)^{\sqrt{\Delta_1}}\Bigg(\int_1^r\left(\frac{\mu_n^{s_c+1}}{r_0^{2-\frac{2}{p-1}}}+ \mu_n^2\right)\left(\frac{1}{(1+r')^{\frac{1+\sqrt{\Delta_1}}{2}}}+|\kappa_{n,1}-\nu_1|\right)\frac{{r'}^2}{(1+r')^{\frac{1+\sqrt{\Delta_1}}{2}}}dr' \\
&&+\frac{\mu_n^{s_c+1}}{r_0^{2-\frac{2}{p-1}}}+ \mu_n^2\Bigg).
\eee
This yields\footnote{Here, we use the fact that
$$\sqrt{\Delta_1}-1 =\frac{4}{p-1}<1$$
since $p>5$, so that
$$\int_0^r\frac{{r'}^2}{ (1+r')^{1+\sqrt{\Delta_1}}}\lesssim (1+r)^{2-\sqrt{\Delta_1}}.$$}
\bee
\sup_{1\leq r\leq\frac{r_0}{\mu_n}}\frac{|(\kappa_{n,1}- \nu_1)(r)|}{|\nu_1(r)|} &\lesssim& r_0^2\left(1+\frac{\mu_n^{s_c-1}}{r_0^{2-\frac{2}{p-1}}}\right)
+\left(\frac{r_0}{\mu_n}\right)^{\sqrt{\Delta_1}}\left(\frac{\mu_n^{s_c+1}}{r_0^{2-\frac{2}{p-1}}}+ \mu_n^2\right)
\eee
and hence, together with \eqref{eq:controlofthedifferencekappanmnumonrleq1} and the fact that\footnote{Indeed, we have in view of \eqref{eq:explicitcompuationofDelta1}
$$\sqrt{\Delta_1}=\frac{p+3}{p-1}=2 - \frac{p-5}{p-1}<2$$
since $p>5$.} $\sqrt{\Delta_1}<2$, we have for $n\geq N$ large enough 
\bee
\sup_{0\leq r\leq\frac{r_0}{\mu_n}}\frac{|(\kappa_{n,1}- \nu_1)(r)|}{|\nu_1(r)|} &\lesssim& r_0^2.
\eee
The corresponding estimates for first order derivatives are obtained in the same way, and \eqref{vnvnonvnornor} is proved.

%%%%%%%%%%%%%%%%%%%%%%%%%%%%%%%%%%%%%%%%%%%%%%%%%%%%%%

\section{Proof of Lemma \ref{lemma:comparisionforrleqr0toodeH}}
\label{appendixphinzero}

%%%%%%%%%%%%%%%%%%%%%%%%%%%%%%%%%%%%%%%%%%%%%%%%%%%%%%%%%%%%%%%%%%%%%%

\noindent{\bf step 1} Proof of \eqref{uniformroximity}. Let 
$$\kappa_n :=\varphi_{n,0}(\mu_nr).$$
Then, since $\varphi_{n,0}$ satisfies $\mathcal{L}_{n,0}(\varphi_{n,0})=0$, we infer
$$-\pr^2_r\kappa_n -\frac{2}{r}\pr_r\kappa_n  -p\left(\mu_n^{\frac{2}{p-1}}\Phi_n(\mu_nr)\right)^{p-1}\kappa_n =-\mu_n^2\Lambda\kappa_n.$$
This yields
\bee
H(\kappa_n) &=& f_n
\eee
where we have introduced the notation
\bee
f_n &:=& p\left(\left(\mu_n^{\frac{2}{p-1}}\Phi_n(\mu_nr)\right)^{p-1}-Q^{p-1}(r)\right)\kappa_n -\mu_n^2\Lambda\kappa_n.
\eee
Since $H(\Lambda Q)=0$, we infer
\bee
H\left(\kappa_n - \frac{p-1}{2}\Lambda Q\right) &=& f_n.
\eee
Recall the solution $\rho$ to $H(\rho)=0$ constructed in Lemma \ref{lemma:homogeneoussolutionsofH} such that $(\Lambda Q, \rho)$ forms a basis of solutions of $H(w)=0$, then the solution to $$H(w) = f$$
is given by 
\bee
w &=&  \left(a_1+\int_0^r f\rho {r'}^{2}dr'\right)\Lambda Q+ \left(a_2-\int_0^r f\Lambda Q {r'}^{2}dr'\right)\rho.
\eee
We infer
\bee
\kappa_n - \frac{p-1}{2}\Lambda Q &=&  \left(a_1+\int_0^r f_n\rho {r'}^{2}dr'\right)\Lambda Q+ \left(a_2-\int_0^r f_n\Lambda Q {r'}^{2}dr'\right)\rho.
\eee
Since $\Lambda Q$ is a smooth function at $r=0$ with 
$$\Lambda Q(0)=\frac{2}{p-1}\neq 0,$$
we infer from the Wronskian relation that $\rho$ has the following asymptotic behavior 
$$\rho\sim\frac{c}{r}\textrm{ as }r\to 0_+$$
for some constant $c\neq 0$, and hence, we must have $a_2=0$. Furthermore, since we have
\bee
\left(\kappa_n - \frac{p-1}{2}\Lambda Q\right)(0)=0,\,\,\,\, \Lambda Q(0)=\frac{2}{p-1}\neq 0
\eee
we infer $a_1=0$. Hence, we have
\bee
\kappa_n - \frac{p-1}{2}\Lambda Q &=&  \left(\int_0^r f_n\rho {r'}^{2}dr'\right)\Lambda Q - \left(\int_0^r f_n\Lambda Q {r'}^{2}dr'\right)\rho.
\eee

In order to estimate $f_n$, recall from Corollary \ref{cor:consequenceforPhinofproponexistence} that we have
\bee
\sup_{r\leq r_0}\left|\Phi_n(r) -  \frac{1}{\mu_n^{\frac{2}{p-1}}}Q\left(\frac{r}{\mu_n}\right)\right| \lesssim \mu_n^{s_c-1}
\eee
This yields
\bea\label{eq:usefulestimatetocontrolfn}
\sup_{r\leq \frac{r_0}{\mu_n}}\left|p\left(\left(\mu_n^{\frac{2}{p-1}}\Phi_n(\mu_nr)\right)^{p-1}-Q^{p-1}(r)\right)\right| \lesssim \frac{\mu_n^{s_c+1}}{r_0^{2-\frac{2}{p-1}}}. 
\eea
Also, we rewrite $f_n$ as
\bea\label{eq:usefulestimatetocontrolfnbis}
f_n&=& p\left(\left(\mu_n^{\frac{2}{p-1}}\Phi_n(\mu_nr)\right)^{p-1}-Q^{p-1}(r)\right)\frac{p-1}{2}\Lambda Q -\mu_n^2 \frac{p-1}{2}\Lambda^2 Q\\
\nonumber&&+p\left(\left(\mu_n^{\frac{2}{p-1}}\Phi_n(\mu_nr)\right)^{p-1}-Q^{p-1}(r)\right)\left(\kappa_n - \frac{p-1}{2}\Lambda Q\right) -\mu_n^2\Lambda\left(\kappa_n - \frac{p-1}{2}\Lambda Q\right).
\eea

We start with the region $0\leq r\leq 1$. In view of the asymptotic behavior for $\Lambda Q$ and $\rho$:
$$\Lambda Q\sim \frac{2}{p-1}\textrm{ and }\rho\sim \frac{c}{r}\textrm{ as }r\to 0+,$$
we infer
\bee
\left|\kappa_n - \frac{p-1}{2}\Lambda Q\right| &\lesssim&  \int_0^r |f_n|{r'} dr'+\frac{1}{r}\left(\int_0^r |f_n|{r'}^2 dr'\right).
\eee
Together with \eqref{eq:usefulestimatetocontrolfn} and \eqref{eq:usefulestimatetocontrolfnbis} and integrating by parts the term $\Lambda(\kappa_n - (p-1)/2\Lambda Q)$, we deduce
\bee
\left|\kappa_n - \frac{p-1}{2}\Lambda Q\right| &\lesssim& \left(\frac{\mu_n^{s_c+1}}{r_0^{2-\frac{2}{p-1}}}+ \mu_n^2\right)\Bigg(1+ \int_0^r \left|\kappa_n - \frac{p-1}{2}\Lambda Q\right| {r'} dr'\\
&&+\frac{1}{r}\left(\int_0^r \left|\kappa_n - \frac{p-1}{2}\Lambda Q\right| {r'}^2 dr'\right)\Bigg).
\eee
We infer 
\bea\label{eq:estimateofthefifferencebetweenkappanabdLambdaQforrleq1}
\sup_{0\leq r\leq 1}\left|\kappa_n - \frac{p-1}{2}\Lambda Q\right| &\lesssim& \frac{\mu_n^{s_c+1}}{r_0^{2-\frac{2}{p-1}}}+ \mu_n^2.
\eea
In particular, this yields
\bea\label{eq:estimateforintegraloffnbetweenoand1}
 \int_0^1 |f_n|{r'} dr'+\int_0^1 |f_n|{r'}^2 dr' &\lesssim& \frac{\mu_n^{s_c+1}}{r_0^{2-\frac{2}{p-1}}}+ \mu_n^2.
\eea

Next, we consider the region $r\geq 1$. Recall the asymptotic behavior at infinity of $\Lambda Q$ and $\rho$ given by Lemma \ref{lemma:homogeneoussolutionsofH}
\bee
\Lambda Q(r)\sim\frac{c_7\sin\left(\om\log(r)+c_8\right)}{r^{\frac{1}{2}}},\,\,\,\,\rho(r)\sim\frac{c_9\sin\left(\om\log(r)+c_{10}\right)}{r^{\frac{1}{2}}}\textrm{ as }r\to +\infty,
\eee
where $c_7, c_9\neq 0$, $c_8, c_{10}\in\mathbb{R}$. We infer for $r\geq 1$
\bee
\left|\kappa_n - \frac{p-1}{2}\Lambda Q\right| &\lesssim&  \left( \frac{\mu_n^{s_c+1}}{r_0^{2-\frac{2}{p-1}}}+ \mu_n^2+\int_1^r |f_n|\frac{{r'}^{2}}{(1+{r'})^{\frac{1}{2}}}dr'\right)\frac{1}{(1+r)^{\frac{1}{2}}}.
\eee
After integrating by parts the term $\Lambda(\kappa_n - (p-1)/2\Lambda Q)$, and together with \eqref{eq:usefulestimatetocontrolfn} and \eqref{eq:usefulestimatetocontrolfnbis}, we deduce
\bee
&&(1+r)^{\frac{1}{2}}\left|\kappa_n - \frac{p-1}{2}\Lambda Q\right| \\
&\lesssim&  \left(\frac{\mu_n^{s_c+1}}{r_0^{2-\frac{2}{p-1}}}+ \mu_n^2\right)\left(1+\int_1^r \left(\frac{1}{(1+{r'})^{\frac{1}{2}}}+\left|\kappa_n - \frac{p-1}{2}\Lambda Q\right|\right)\frac{{r'}^{2}}{(1+{r'})^{\frac{1}{2}}}dr'\right)\\
&\lesssim&  \left(\frac{\mu_n^{s_c+1}}{r_0^{2-\frac{2}{p-1}}}+ \mu_n^2\right)(1+r)^2+\left(\frac{\mu_n^{s_c+1}}{r_0^{2-\frac{2}{p-1}}}+ \mu_n^2\right)\left(\int_1^r \left|\kappa_n - \frac{p-1}{2}\Lambda Q\right|\frac{{r'}^{2}}{(1+{r'})^{\frac{1}{2}}}dr'\right).
\eee
This yields
\bee
\sup_{1\leq r\leq\frac{r_0}{\mu_n}}(1+r)^{\frac{1}{2}}\left|\kappa_n - \frac{p-1}{2}\Lambda Q\right| &\lesssim& r_0^2\left(1+\frac{\mu_n^{s_c-1}}{r_0^{2-\frac{2}{p-1}}}\right)
\eee
which together with \eqref{eq:estimateofthefifferencebetweenkappanabdLambdaQforrleq1} implies
\bee
\sup_{0\leq r\leq\frac{r_0}{\mu_n}}(1+r)^{\frac{1}{2}}\left|\kappa_n - \frac{p-1}{2}\Lambda Q\right| &\lesssim& r_0^2\left(1+\frac{\mu_n^{s_c-1}}{r_0^{2-\frac{2}{p-1}}}\right).
\eee
Hence, we have for $n\geq N$ large enough 
\bee
\sup_{0\leq r\leq r_0}\left(1+\frac{r}{\mu_n}\right)^{\frac{1}{2}}\left|\varphi_{n,0}(r) - \frac{p-1}{2}\Lambda Q\left(\frac{r}{\mu_n}\right)\right| &\lesssim& r_0^2.
\eee

\noindent{\bf step 2} Proof of \eqref{eronvoinv}. Recall from Lemma \ref{lemma:comparisionforrleqr0toodeH} that we have for $n\geq N$ large enough 
\bee
\sup_{0\leq r\leq r_0}\left(1+\frac{r}{\mu_n}\right)^{\frac{1}{2}}\left|\varphi_{n,0}(r) - \frac{p-1}{2}\Lambda Q\left(\frac{r}{\mu_n}\right)\right| &\lesssim& r_0^2.
\eee
Also, recall that
$$\Lambda Q(r)\sim \frac{c_7\sin(\om\log(r)+c_8)}{r^{\frac{1}{2}}}\textrm{ as }r\to +\infty$$
and that $r_{\Lambda Q,n}<r_0/\mu_n$ introduced in Corollary \ref{cor:consequenceforPhinofproponexistence} denotes the last zero of $\Lambda Q$ before $r_0/\mu_n$. 
This yields
$$\Big|\om\log(r_{1,n})-\om\log(\mu_n)+c_8 -(\om\log(r_{\Lambda Q,n})+c_8)\Big| \lesssim r_0^2$$
and hence
\bee
r_{1,n} &=& \mu_n r_{\Lambda Q,n} e^{O(r_0^2)}= \mu_n r_{\Lambda Q,n}(1+ O(r_0^2)).
\eee
Furthermore, since we have from the proof of Corollary \ref{cor:consequenceforPhinofproponexistence}  that 
$$e^{-\frac{3\pi}{2\omega}}\frac{r_0}{\mu_n}\leq r_{\Lambda Q, n}\leq \frac{r_0}{\mu_n},$$
and 
\bee
r_{0,n} &=& \mu_n r_{\Lambda Q,n}(1+ O(r_0^2)),
\eee
we deduce
$$r_{1,n}=r_{0, n}+O(r_0^3)$$
and 
$$e^{-\frac{2\pi}{\omega}}r_0\leq r_{1, n}\leq r_0.$$

\end{appendix}

%%%%%%%%%%%%%%%%%%%%%%%%%%%%%%%%%%%%%%%%%%%%%%%%%%%%%%%%%%%%%%%%%%%%%%%%%%


\begin{thebibliography}{10}

\bibitem{buddone} Budd, C.; Norbury, J., Semilinear elliptic equations and supercritical growth, J. Differential
Equations 68 (1987), no. 2, 169--197.

\bibitem{buddselfsim} Budd, C.; Qi, Y-W., The existence of bounded solutions of a semilinear elliptic equations, J. Differential Equations 82 (1989) 207--218.

\bibitem{bizon}
P. Biernat, P. Bizo{\'n}, Shrinkers, expanders, and the unique continuation beyond generic blowup in the heat flow for harmonic maps between spheres, Nonlinearity \textbf{24} (8), 2011, 2211--2228.

\bibitem{cazenavebook} Cazenave, T.; Semilinear Schr\"odinger equations, Courant Lecture Notes in Mathematics, 10, NYU, CIMS, AMS 2003.

\bibitem{cazenaveshattah} Cazenave, T,; Shatah, J.; Tahvildar-Zadeh, A. S., Harmonic maps of the hyperbolic space and development of singularities in wave maps and Yang-Mills fields, Ann. Inst. H. Poincaré Phys. Théor. 68 (1998), no. 3, 315--349. 

\bibitem{Co} Collot, C, Type II blow up manifolds for the energy supercritical wave equation. arXiv preprint arXiv:1407.4525.

\bibitem{Co2} Collot, C,  Non radial type II blow up for the energy supercritical semilinear heat equation, preprint 2016.

\bibitem{CMR} Collot, C.; Merle. F; Rapha\"el, P., Dynamics near the ground state for the energy critical nonlinear heat equation in large dimension, preprint 2016.

\bibitem{Cor} Corlette, K., Wald, R. M. Morse Theory and Infinite Families of Harmonic Maps Between Spheres. Communications in Mathematical Physics (2001), 215(3), 591-608.

\bibitem{martelmulti} C\^ote, R.; Martel, Y.; Merle, F., Construction of multi-soliton solutions for the L2-supercritical gKdV and NLS equations, Rev. Mat. Iberoam. 27 (2011), no. 1, 273--302. 

\bibitem{wei} Dancer, E. N.; Guo, Z.; Wei, J, Non-radial singular solutions of the Lane-Emden equation in $\Bbb R^N$, Indiana Univ. Math. J. 61 (2012), no. 5, 1971--1996.

\bibitem{Di} Ding, W. Y., Ni, W. M. (1985). On the elliptic equation $\ Delta u+ Ku^{(n+ 2)/(n-2)}= 0$ and related topics. Duke mathematical journal, 52(2), 485-506.

\bibitem{donninger1} 
Donninger, R., On stable self-similar blowup for equivariant wave maps, Comm. Pure Appl. Math. 64 (2011), 1095--1147.
 
\bibitem{donninger2} 
Donninger, R., Stable self-similar blowup in energy supercritical Yang-Mills theory, Math. Z. 278 (2014), 1005--1032.

\bibitem{donninger} Donninger, R.; Schörkhuber, B.,Stable blow up dynamics for energy supercritical wave equations, Trans. Amer. Math. Soc. 366 (2014), no. 4, 2167--2189.

\bibitem{DS} Donninger, R.; Schörkhuber, B.,Stable blow up for the supercritical Yang-Mills heat flow, arXiv preprint arXiv:1604.07737.

\bibitem{DKSW} Donninger, R.; Krieger, J.; Szeftel, J.; Wong, W., Codimension one stability of the catenoid under the vanishing mean curvature flow in Minkowski space. Duke Math. J. 165 (2016), 723--791.

\bibitem{FHV}  Filippas, S., Herrero, M. A., Velazquez, J. J. (2000, December). Fast blow-up mechanisms for sign-changing solutions of a semilinear parabolic equation with critical nonlinearity. In Proceedings of the Royal Society of London A: Mathematical, Physical and Engineering Sciences (Vol. 456, No. 2004, pp. 2957-2982). The Royal Society.

\bibitem{germain} Germain, P.; Rupflin, M., Selfsimilar expanders of the harmonic map flow, Ann. Inst. H. Poincaré Anal. Non Linéaire 28 (2011), no. 5, 743--773.

\bibitem{Gi} Giga, Y. (1986). On elliptic equations related to self-similar solutions for nonlinear heat equations. Hiroshima mathematical journal, 16(3), 539-552.

\bibitem{Gi1} Giga, Y.,  Kohn, R. V. (1985). Asymptotically self similar blow up of semilinear heat equations. Communications on Pure and Applied Mathematics, 38(3), 297--319.

\bibitem{Gi2} Y. Giga and R.V. Kohn, Characterizing blowup using similarity variables, Indiana Univ. Math. J., 36 (1987), 1--40.

\bibitem{Gi3} Giga, Y.,  Kohn, R. V. (1989). Nondegeneracy of blowup for semilinear heat equations. Communications on Pure and Applied Mathematics, 42(6), 845--884.

\bibitem{Gi4}  Giga, Y., Matsui, S. Y., Sasayama, S. (2002). Blow up rate for semilinear heat equation with
subcritical nonlinearity.

\bibitem{HV}
M. Herrero, J. Velazquez, A blow up result for the semilinear heat equations in the supercritical case, preprint 1992.

\bibitem{HR} Hadzic, M.; Rapha\"el, P., On melting and freezing for the radial Stefan problem, preprint 2015.

\bibitem{Jo} Joseph, D. D., Lundgren, T. S. (1973). Quasilinear Dirichlet problems driven by positive sources. Archive for Rational Mechanics and Analysis, 49(4), 241-269.

\bibitem{KS} Krieger, J.; Schlag, W., Non-generic blow-up solutions for the critical focusing NLS in 1-D. J. Eur. Math. Soc. (JEMS) 11 (2009), no. 1, 1--125.

\bibitem{koch} H. Koch, Self-similar solutions to super-critical gKdV, Nonlinearity \textbf{28} (3), 2015, 545--575.

\bibitem{lepin} Lepin L.A., Self-similar solutions of a semilinear heat equation, Mat. Model. 2 (1990) 63--7.

\bibitem{YiLi} Y. Li. Asymptotic Behavior of Positive Solutions of Equation $\Delta u +K(x)u^p=0$ in $\mathbb{R}^n$. J. Differential Equations, 95(2):304-330, 1992.

\bibitem{mamerle} Martel, Y.; Merle, F.,  A Liouville theorem for the critical generalized Korteweg-de Vries equation, J. Math. Pures Appl. (9) 79 (2000), no. 4, 339--425.

\bibitem{MMNR} Martel, Y.; Merle, F.; Nakanishi, K.; Rapha\"el, P., Codimension one threshold manifold for the critical gKdV equation, arXiv:1502.04594.

\bibitem{MaMeRa} Martel, Y.; Merle, F.; Rapha\"el, P.,  Blow up for the critical generalized Korteweg–de Vries equation. I: Dynamics near the soliton, Acta Math. 212 (2014), no. 1, 59--140.

\bibitem {MaMe1}  Matano, H.; Merle, F., Classification of type I and type II behaviors for a supercritical nonlinear heat equation, J. Funct. Anal. 256 (2009), no. 4, 992–1064.

\bibitem{MaMe2}  Matano, H.; Merle, F., On nonexistence of type II blowup for a supercritical nonlinear heat equation, Comm. Pure Appl. Math. 57 (2004), no. 11, 1494–1541.

\bibitem{MaMe3} Matano, H., Merle, F., Threshold and generic type I behaviors for a supercritical nonlinear heat equation. Journal of Functional Analysis (2011), 261(3), 716-748.

\bibitem{meraphannals} Merle, F.; Rapha\"el, P., The blow-up dynamic and upper bound on the blow-up rate for critical nonlinear Schrödinger equation, Ann. of Math. (2) 161 (2005), no. 1, 157--222.

\bibitem{MZduke} Merle, F.; Zaag, H., Stability of the blow-up profile for equations of the type $u_t=\Delta u+|u|^{p-1}u$, Duke Math. J. 86 (1997), no. 1, 143--195.

\bibitem{MeZa2} Merle, F.,  Zaag, H. (1998). Optimal estimates for blowup rate and behavior for nonlinear heat equations. Communications on pure and applied mathematics, 51(2), 139-196.

\bibitem{MZ} Merle, F.; Zaag, H., Existence and classification of characteristic points at blow-up for a semilinear wave equation in one space dimension, Amer. J. Math. 134 (2012), no. 3, 581--648.

\bibitem{MRR} Merle, F., Rapha\"el, P.; Rodnianski, I., Type II blow up for the energy super critical NLS,  {\em to appear in Cambridge Math. Jour}.

\bibitem{Mizo} Mizoguchi, N., Type-II blowup for a semilinear heat equation, Adv. Differential Equations 9 (2004), no. 11-12, 1279--1316. 

\bibitem{Mizoselfsim} Mizoguchi, N., Nonexistence of backward self-similar blowup solutions to a supercritical semilinear heat equation, J. Funct. Anal. 257 (2009), no. 9, 2911--2937.

\bibitem{Mizo2} Mizoguchi, N. (2011). Nonexistence of type II blowup solution for a semilinear heat equation.
Journal of Differential Equations, 250(1), 26-32.

 \bibitem{NS} Nakanishi, K.; Schlag, W.,  Invariant manifolds and dispersive Hamiltonian evolution equations, Zurich Lectures in Advanced Mathematics, European Mathematical Society (EMS), Zürich, 2011.

\bibitem{Handbookmathfunctions}
F. Olver, D. Lozier, R. Boisvert, C. Clark, NIST Handbook of Mathematical Functions, Cambridge, Cambridge University Press, 2010.

\bibitem{RaphRod} Rapha\"el, P., Rodnianski, I., Stable blow up dynamics for critical corotational wave maps and the equivariant Yang Mills problems, {\em Publ. Math. Inst. Hautes Etudes Sci.} 115 (2012), 1--122.

\bibitem{RS} Reed, S.; Simon, B., Methods of modern mathematical physics, I-IV, Functional analysis,  Second edition, Academic Press, Inc., New York, 1980.

\bibitem{Sch} Schweyer, R. Type II blow-up for the four dimensional energy critical semi linear heat equation. Journal of Functional Analysis (2012), 263(12), 3922-3983.

\bibitem{troy} Troy, W., The existence of bounded solutions of a semilinear heat equation, SIAM J. Math. Anal. 18 (1987), 332--336.
\end{thebibliography}
\end{document}